\documentclass{amsart}
\usepackage[utf8]{inputenc}
\usepackage[left=1.2in, right=1.2in, top=1in]{geometry}
\usepackage{esint}
\usepackage{amsmath,amssymb,amsthm,mathrsfs,color,times,textcomp,verbatim,mathtools}
\allowdisplaybreaks
\usepackage{xcolor}
\usepackage[colorlinks=true]{hyperref}
\hypersetup{urlcolor=blue, citecolor=red, linkcolor=blue}
\usepackage[numbers, sort&compress]{natbib}
\theoremstyle{plain}
\numberwithin{equation}{section}
\newtheorem{thm}{Theorem}[section]
\newtheorem{lem}{Lemma}[section]
\newtheorem{prop}{Proposition}[section]

\newtheorem{cor}{Corollary}[section]
\newtheorem{rem}{Remark}[section]

\usepackage{graphicx}
\numberwithin{equation}{section}

\def\Xint#1{\mathchoice
	{\XXint\displaystyle\textstyle{#1}}%
	{\XXint\textstyle\scriptstyle{#1}}%
	{\XXint\scriptstyle\scriptscriptstyle{#1}}%
	{\XXint\scriptscriptstyle\scriptscriptstyle{#1}}%
	\!\int}
\def\XXint#1#2#3{{\setbox0=\hbox{$#1{#2#3}{\int}$ }
		\vcenter{\hbox{$#2#3$ }}\kern-.6\wd0}}

\def\dashint{\Xint-}

\title[$T$-curvature flow]{\bf Prescribed $T$-curvature flow on the four-dimensional unit ball}

\author{Pak Tung Ho}

\address{Department of Mathematics, Tamkang University, Tamsui, New Taipei City 251301, Taiwan}

\email{paktungho@yahoo.com.hk}
\author{Cheikh Birahim NDIAYE}
\address{Department of Mathematics Howard University Annex 3, Graduate School of Arts and Sciences DC 20059 Washington, USA}
\email{cheikh.ndiaye@howard.edu}
\author{Liming Sun}
\address{State Key Laboratory of Mathematical Sciences, Academy of Mathematics and Systems Science, Chinese Academy of Sciences, Beijing 100190, China}
\email{lmsun@amss.ac.cn}

\author{Heming Wang}
\address{State Key Laboratory of Mathematical Sciences, Academy of Mathematics and Systems Science, Chinese Academy of Sciences, Beijing 100190, China}
\email{hmwang@amss.ac.cn}

\date{\today \,(Last Typeset)}

\newcommand{\pa}{\partial}
\renewcommand{\d}{\mathrm{d}}
\newcommand{\R}{\mathbb{R}}
\renewcommand{\S}{\mathbb{S}}
\newcommand{\cS}{\mathcal{S}}
\newcommand{\B}{\mathbb{B}}
\newcommand{\cB}{\mathcal{B}}
\newcommand{\cH}{\mathcal{H}}
\newcommand{\N}{\mathbb{N}}
\renewcommand{\H}{\mathcal{H}}
\renewcommand{\P}{\mathbb{P}}
\newcommand{\cP}{\mathcal{P}}
\newcommand{\la}{\langle}
\newcommand{\ra}{\rangle}
\newcommand{\Si}{\Sigma}
\newcommand{\si}{\sigma}

\newcommand{\morse}{\operatorname{morse}\mkern1mu}
\newcommand{\vol}{\operatorname{vol}\mkern1mu}
\newcommand{\be}{\begin{equation}}      
	\newcommand{\ee}{\end{equation}} 
\newcommand{\al}{\alpha}
\newcommand{\var}{\varepsilon} 
\newcommand{\lam}{\lambda}
\newcommand{\Lam}{\Lambda}
\newcommand{\kap}{\kappa}

\begin{document}

	\begin{abstract}
		In this paper, we study the prescribed $T$-curvature problem on the unit ball $\B ^4$ of $\R ^4$ via the  $T$-curvature flow approach.  By combining the Ache-Chang \cite{AC2017}'s inequality with the Morse theoretical approach of Malchiodi-Struwe \cite{MS2006}, we establish existence results under strong Morse type inequalities at infinity. As a byproduct of our argument, we also prove the exponential convergence of the $T$-curvature flow on $\B ^4$, starting from a $Q$-flat and minimal metric conformal to  the standard Euclidean metric, to an extremal metric of the Ache-Chang \cite{AC2017}'s inequality whose explicit expression was  derived by Ndiaye-Sun \cite{NS2024}.
	\end{abstract}
	\maketitle
	\tableofcontents
	\section{Introduction}
		The problem of finding a conformal metric on a Riemannian manifold with certain prescribed
	curvature function has been extensively investigated in  recent decades. This study seeks to understand the analytic theory of conformally invariant operators and the geometric structure of the underlying manifold; see, e.g., \cite{N2024,N2009,S1984,CQ1997-1,CQ1997-2,J2013,B1985,GJMS1992}.  In particular, the case of the unit sphere is of special interest, where the action of a noncompact conformal group gives rise to essential obstructions to the  existence of conformal metrics;  see, e.g., \cite{H1990,L1995,L1996,XZ2016,S2005,MS2006}. Before introducing the main problem studied in this paper, we first recall some fundamental geometric notions and conformal quantities.
	
	In the context of 2-dimensional geometry,  the Laplace-Beltrami operator on a compact  surface $(\Si,g) $ governs the transformation law of the Gaussian curvature. In fact, under
	a conformal change of metric $\tilde{g}=e^{2u}g$, one has
	\be\label{conformal-1}
	\Delta_{\tilde{g}}=e^{-2u}\Delta_g\quad \text{ and }\quad   -\Delta_{g}u+K_ {g}=K_{\tilde{g}}e^{2u} \quad \mbox{ in }\, \Si ,
	\ee 
	where $\Delta_g=\operatorname{div}_g \nabla_g $ and $K_{g}$ denote the Laplace-Beltrami operator and  Gaussian curvature of $(\Si, g)$, respectively, while $\Delta_{\tilde{g}}$ and $K_ {\tilde{g}}$ are the corresponding objects associated with $(\Si, \tilde{g})$. If  $(\Si ,g)$ has a smooth  boundary  $\pa \Si$, with $\frac{\pa}{\pa \nu_g}$ denoting the outward normal derivative on $\pa\Si$, then under the same conformal change,  the pair $(\Delta_{g},\frac{\pa}{\pa \nu_g})$  governs the transformation laws of both the Gaussian curvature $K_{g}$ in $\Si $ and the geodesic
	curvature $k_g$ of $\pa\Si$:
	\be \label{conformal-2}
	\left\{
	\begin{aligned}
		&   \Delta_{\tilde{g}}=e^{-2u}\Delta_g, \\
		&	\frac{\pa }{\pa \nu_{\tilde{g}}}=e^{-u}\frac{\pa }{\pa \nu_g},
	\end{aligned}
	\right. \quad\text{ and } \quad \left\{
	\begin{aligned}	 &-\Delta_{g}u+K_ {g}=K_{\tilde{g}}e^{2u} && \mbox{ in }\, \Si,  \\ &\frac{\pa  u}{\pa \nu_g}+k_{g}=k_{\tilde{g}}e^{u}  &&\mbox{ on }\, \pa \Si ,	\end{aligned}
	\right.
	\ee 
	where  $\frac{\pa }{\pa \nu_{\tilde{g}}}$ and $k_{\tilde{g}}$ denote the normal derivative and the geodesic curvature of $\pa\Si$ associated with $(\Si, \tilde{g})$. Moreover, we have the Gauss–Bonnet formula
	\be\label{GB1}
	\int_{\Si} K_g \,\d v_g+\int_{\pa\Si} k_g \, \d s_g=2 \pi \chi(\Si),
	\ee
	where $\chi(\Si)$ is the Euler-Poincar\'e characteristic of $\Si$, $\d v_g$ is the element area of $\Si$ and $\d s_g$ is the line element of $\pa \Si$. Thus, \eqref{GB1} is a topological invariant and hence a conformal invariant.

	In the search for higher-order conformally invariant operators, Paneitz
	in 1983 (see \cite{P2008}) discovered a remarkable fourth-order differential operator on  compact four-manifolds, defined by
	\be\label{def:P_g^4}
	P_g^4:=(-\Delta_g)^2-\operatorname{div}_g\Big(\frac{2}{3}R_g g-2\operatorname{Ric}_g\Big)\d,
	\ee
	where $\d $ is De Rham differential,   and $R_g$ and  $\operatorname{Ric}_g$ are   the scalar curvature and the Ricci tensor associated with the metric $g$, respectively.    The corresponding $Q$-curvature (denoted by $Q^4_g$ or $Q_g$), 
	introduced by Branson (see \cite{BO1991,B1985}),  is expressed as
	\be\label{def:Q_g}
	Q_g:=-\frac{1}{6}(\Delta_gR_g-R_g^2+3|\operatorname{Ric}_g|^2).
	\ee 
	Similarly to the role  of the  Laplace-Beltrami operator and the Gaussian curvature in two-dimensions, as shown in \eqref{conformal-1}, the Paneitz operator governs the conformal transformation laws of geometric quantities under a change of metric. Indeed, for  a conformal change of metric $\tilde{g}=e^{2 u} g$, one has
	\be\label{conformal-P4}
	P_{\tilde{g}}^4=e^{-4u}P_g^4\quad \text{ and }\quad P_g^4 u+ Q_g=Q_{\tilde{g}} e^{4 u} \quad \text { on }\, M ,
	\ee 
	where $P_{\tilde{g}}^4$ and $Q_{\tilde{g}}$ denote the  Paneitz
	operator and $Q$-curvature associated with  the metric $\tilde{g}$. 
	
	More generally, a metrically defined operator   $A$ on a compact manifold $(M,g)$ is said to be conformally covariant if, under the conformal change of metric $\tilde{g}=e^{2 u} g$, the pair of corresponding operators $A_g$ and $A_{\tilde{g}}$ are related by
	\be\label{confromal}
	A_{\tilde{g}}(\phi)=e^{-b u} A_g(e^{a u} \phi), \quad \forall \,\phi \in C^{\infty}(M),
	\ee
	for some constants $a$ and $b$. Basic examples of such second order operators include the Laplacian $\Delta_g$ in dimension $n =2$, and the conformal Laplacian $-\frac{4(n-1)}{n-2} \Delta_g+R_g$ in dimensions $n \geq 3$, where $R_g$ is  the scalar curvature of the metric $g$. In particular, the Paneitz operator $P^4_g$  is conformally covariant in the sense of  \eqref{confromal} with $(a,b)=(0,4)$.

	On compact manifolds of even dimension $n$, the existence of a conformally covariant operator $P^n_g$ satisfying $P_n^{\tilde{g}}=e^{-n u}P^n_g$ was established by Graham, Jenne, Mason, and Sparling \cite{GJMS1992}. On the sphere  $\S^n$, explicit formulas for  $P^n_{g_{\S^n}}$ on $\S^n$ were given by Branson  \cite{B1987} and Beckner \cite{B1993}:
	\be\label{operatorP3}
	P^n_{g_{\S^n}}=  \left\{\begin{aligned}
		&\Pi_{k=0}^{\frac{n-2}{2}}(-\Delta_{g_{\S^n}}+k(n-k-1)) && \text { if $n$ is even}, \\ & \Big(-\Delta_{g_{\S^n}}+\Big(\frac{n-1}{2}\Big)^2\Big)^{1 / 2} \Pi_{k=0}^{\frac{n-3}{2}}(-\Delta_{g_{\S^n}}+k(n-k-1)) && \text { if $n$ is odd.} 
	\end{aligned}\right.
	\ee

	When $n=3$ or 4, one can further associate to $P_g^n$ a natural curvature quantity $Q^n_g$ of order $n$, characterized by the conformal transformation law 
	\be\label{Q-curvature-equation}
	P_{g}^n u+Q_{g}^n=Q_{\tilde g}^n e^{n u} \quad \text { on } \,M ,
	\ee
under the conformal change $\tilde{g}=e^{2u} g$. In particular, on the standard sphere $\S^n$, if  $\tilde{g}$ 
	is isometric to the standard metric $g_{\S^n}$, then equation \eqref{Q-curvature-equation} reduces to
	\be\label{Q-curvature-eq} 
	P_{g_{\S^n}}^n u+(n-1)!=(n-1)! e^{n u} \quad \text { on } \,\S^n .
	\ee

		Before stating the problem investigated in this paper, we need to describe important boundary
		analogues of the Paneitz operator and of $Q$-curvature for compact manifolds with
		boundary in dimension four. Let  $(M, \pa M, g)$ be  a smooth  four-manifold with boundary, and denote by $\hat{g}$  the restriction of $g$ to $\pa M$. Let $\nu_g$ be the unit \emph{outward normal} vector field along the boundary, and let us use Greek indices to denote directions tangent to the boundary, that is, $g_{\al \beta}=\hat{g}_{\al \beta}$. With this convention, the second fundamental form is given by $\amalg_{\al \beta}=\frac{1}{2} \frac{\pa}{\pa \nu_g} g_{\al \beta}$ and  the mean curvature is defined as
		\be\label{def:mean_curvature}
		H_g:=\frac{1}{3}g^{\al \beta} \amalg_{\alpha \beta}.
		\ee
		The boundary operator $P_g^{3,b}$ and a boundary curvature $T_g$, introduced by Chang and Qing  \cite{CQ1997-1}, are defined respectively by
		\begin{align}
			P_g^{3,b}\phi :=&\,-\frac{1}{2}\frac{\pa \Delta_g\phi}{\pa \nu_g} -\Delta_{\hat{g}}\frac{\pa \phi }{\pa \nu_g}
			-2H_g\Delta_{\hat{g}}\phi+\amalg_{\al \beta}(\nabla_{\hat{g}})_{\alpha} (\nabla_{\hat{g}})_{\beta}\phi \nonumber\\&\, +\la \nabla_{\hat{g}}H_g,\nabla_{\hat{g}} \phi \ra _{\hat{g}}+(F-2J)\frac{\pa \phi}{\pa \nu_g}	\label{def:P_g^3}
		\end{align}
		and 
		\be\label{def:T_g}
		T_g:=-\frac{1}{2} \frac{\pa J}{\pa \nu_g} +3J H_g-\la G, \amalg\ra+3 H_g^3-\frac{1}{3} \operatorname{tr}_{\hat{g}}(\amalg^3)+ \Delta_{\hat g} H_g,
		\ee
		where $\phi $ is any smooth function on $M$, $F=\operatorname{Ric}_{n n}$, $J=\frac{R_g}{6}|_{\pa M}$ is the restriction to the boundary of the ambient scalar curvature $R_g$, $G_\beta^\al=R_{n \beta n}^\al$,  $\langle G, \mathrm{II}\rangle=R_{\alpha n \beta n} \Pi_{\alpha \beta}$, and $\Pi_{\alpha \beta}^3=\Pi_{\alpha \gamma} \Pi^{\gamma \delta} \Pi_{\delta \beta}$.
		We point out that $(P_g^{3,b}, T_g)$ depends not only on the intrinsic geometry of $\pa M$, but also on the extrinsic geometry of $\pa M$ in $M$; see \cite{CQ1997-1,CQ1997-2}.

		Similar to the  operators $(\Delta_g,\frac{\pa}{\pa\nu_g})$ and curvatures  $(K_g, k_g)$  on a Riemannian surface $(\Si,\pa\Si,g)$,  which govern the conformal transformation laws \eqref{conformal-2},  the pair $(P_g^4, P_g^{3,b})$ plays an analogous role for $(Q_g,T_g)$ on a four-manifold $(M, \pa M, g)$. Indeed, under the conformal change $\tilde{g}=e^{2u}g$, one has
		\be \label{conformal-3}
		\left\{
		\begin{aligned}
			&  P_{\tilde{g}}^4=e^{-4u}P^4_{g}, \\
			&P_{\tilde{g}}^{3,b}=e^{-3u}P^{3,b}_{g},
		\end{aligned}
		\right. \quad\text{ and }\quad   \left\{
		\begin{aligned}
			& P^4_{g}u+Q_{g}=Q_{\tilde{g}}e^{4u} && \hbox{ in }\, M, \\
			& P^{3,b}_{g}u+T_{g}=T_{\tilde{g}}e^{3u} && \hbox{ on }\,\pa  M,\\
			&\frac{\pa u}{\pa \nu_{g}} +H_{g}=H_{\tilde{g}}e^{u}&& \hbox{ on }\, \pa  M.
		\end{aligned}
		\right.
		\ee 
		In addition to these resemblances, the classical Gauss–Bonnet formula \eqref{GB1} extends to four dimensions as the Gauss–Bonnet–Chern theorem:
		\be \label{GBC-2}
		\int_M(Q_g+|W_g|^2)\, \d v_g +2\int_{\pa  M}(T_g+Z_g)\, \d s _{\hat g}=8\pi^2\chi(M),
		\ee 
		where $W_g$ denotes the Weyl tensor of  $(M,g)$, $Z_g$ (see \cite{CQ1997-1} for the definition) is a pointwise conformal
		invariant  that vanishes if $\pa M$ is totally geodesic, and 
		$\chi(M)$ is the Euler-characteristics of $M$.  
		Setting
		\[
		\kap _{P_g^4}:=\int_M Q_g \,\d v_g \quad \text { and } \quad \kap _{P_g^{3,b}}:=\int_{\pa M} T_g \,\d s_{\hat g},
		\]
		it follows from \eqref{GBC-2} and the pointwise conformal invariance of $W_g \,\d v_g$ and $Z_g \,\d s_{\hat g}$ that the sum $\kap _{P_g^4}+\kap _{P_g^{3,b}}$ is itself conformally invariant. We denote this invariant quantity by
		\be\label{kap-sum}
		\kap _{(P^4_g, P^{3,b}_{g})}:=\kap _{P_g^4}+\kap _{P_g^{3,b}} .
		\ee

		Building on the boundary conformal quantities introduced above, it is natural to study the prescribed 
		$T$-curvature problem on four-manifolds with boundary:  whether a given compact four-dimensional  Riemannian manifold $(M,g)$  with boundary $\pa M$  admits a conformal metric $\tilde{g}$ such that  $Q_{\tilde{g}}=0$ in $M$,  $T_{\tilde{g}}$ equals a prescribed function $f$ on $\pa M$, and $(M,\tilde{g})$ has minimal boundary. 
		
		Let $\mathbb{P}_g^{4,3}  $ be the bilinear form acting on functions satisfying a homogeneous Neumann condition as in  \cite{N2009,N2024}.  From the viewpoint of a Liouville-type problem, we refer to the assumptions  $\operatorname{ker} \mathbb{P}_g^{4,3} \simeq \R$ and $\kap_{(P_g^4, P_g^{3,b})} \notin 4 \pi^2 \N$ as the nonresonant case, while the situation  $\operatorname{ker} \mathbb{P}_g^{4,3} \simeq \R$ and $\kap_{(P_g^4, P_g^{3,b})} \in 4 \pi^2 \N$ is called the resonant case.  To the best of our knowledge, the first existence result was obtained by Chang and Qing \cite{CQ1997-2} under the assumptions that $\mathbb{P}_g^{4,3}$ is nonnegative, $\operatorname{ker} \mathbb{P}_g^{4,3} \simeq \R$, and $\kap_{(P_g^4, P_g^{3,b})}<4 \pi^2$. An alternative proof based on geometric flow methods was later provided by Ndiaye \cite{N2009}.  Subsequently, Ndiaye  \cite{N2008} developed a variant of the min-max scheme of Djadli-Malchiodi \cite{DM2008} to extend the result of Chang-Qing \cite{CQ1997-2}, showing that the problem is solvable whenever $\operatorname{ker} \mathbb{P}_g^{4,3} \simeq \R$ and $\kap_{(P_g^4, P_g^{3,b})}\notin 4\pi^2\N$. The resonant case was then treated by   Ndiaye \cite{N2024} via Bahri's critical points at infinity \cite{B1989}.     More recently,  Ho  \cite{H2025}  established an existence result under the assumptions $\kap_{(P_g^4, P_g^{3,b})}=0$, $\operatorname{ker}\mathbb{P}_{g}^{4,3} \simeq \R$, and $\mathbb{P}_{g}^{4,3}$ is nonnegative.

		In this paper  we are interested in   the prescribed $T$-curvature problem on the four-dimensional unit ball, where the noncompact group of conformal transformations of the ball acts on the equation giving rise to Kazdan-Warner type obstructions, just as in the celebrated prescribed scalar curvature (or Nirenberg) problem; see, e.g., \cite{L1995-boundary}.
		More precisely, let $(\B^4,g_{\B^4})$ be the four-dimensional unit ball endowed  with the  Euclidean metric $g_{\B^4}$. For a  given  smooth positive function  $f$ defined on the boundary $\pa \B^4=\S^3$, we aim to find a metric $g$ conformal to $g_{\B^4}$ such that
		\be\label{maineq}
		\left\{\begin{aligned}
			&Q_g=0&&\mbox{ in }\,\B ^4, \\ &T_g=f && \mbox{ on }\, \S ^3, \\& H_g=0&& \mbox{ on }\,\S ^3.
		\end{aligned}\right.
		\ee 
		Here $Q_g$, $T_g$, and $H_g$  denote  the $Q$-curvature, the 
		$T$-curvature, and the mean curvature associated with $g$, respectively.

		Let $[g_{\B^4}]:=\{e^{2u}g_{\B^4}:u\in C^{\infty}(\B^4)\}$ denote the conformal class of $g_{\B^4}$. Set $g=e^{2u}g_{\B^4}\in [g_{\B^4}]$. Then we deduce from \eqref{conformal-3} that   problem \eqref{maineq} is equivalent to finding $u$ satisfying the following conditions:
		\be \label{maineq1}
		\left\{\begin{aligned}
			&P^4_{g_{\B ^4}}u=0&&\mbox{ in }\,\B ^4,\\
			&P^{3,b}_{g_{\B ^4}}u+T_{g_{\B ^4}}=fe^{3u}&&\mbox{ on }\,\S ^3,\\&\frac{\pa u}{\pa\nu_{g_{\B ^4}}}+H_{g_{\B ^4}}=0&&\mbox{ on }\,
			\S ^3.
		\end{aligned}\right.
		\ee 
		From the definitions in    \eqref{def:P_g^4}-\eqref{def:Q_g} and  \eqref{def:mean_curvature}-\eqref{def:T_g},  it is straightforward to verify that
		\be \label{Q_T_H}
		\left\{\begin{aligned}
			&Q_{g_{\B ^4}}=0&&\mbox{ in }\,\B ^4,\\
			&T_{g_{\B ^4}}=2&&\mbox{ on }\,\S ^3,\\
			&H_{g_{\B ^4}}=1 &&\mbox{ on }\,\S ^3,
		\end{aligned}\right.
		\ee 
		and
		\be \label{P^4&P^3}
		\left\{\begin{aligned}
			&P^4_{g_{\B ^4}}=\Delta_{g_{\B ^4}}^2&&\mbox{ in }\,\B ^4,\\
			& P^{3,b}_{g_{\B ^4}}=-\frac{1}{2}\frac{\pa  \Delta_{g_{\B ^4}}}{\pa \nu_{g_{\B ^4}}}
			-\Delta_{g_{\S ^3}}\frac{\pa }{\pa \nu_{g_{\B ^4}}}
			-\Delta_{g_{\S ^3}} &&\mbox{ on }\,\S ^3,
		\end{aligned}\right.
		\ee 
		where $\frac{\pa }{\pa\nu_{g_{\B ^4}}}  $ denotes the outward normal derivative with respect to the unit normal vector $\nu_{g_{\B ^4}}$.

				Problem  \eqref{maineq}  is remarkably flexible and difficult to solve for several reasons. A first difficulty is the absence of a homogeneous Neumann boundary condition in the boundary value formulation \eqref{maineq1}. Secondly, by integrating by parts, one readily obtains a simple necessary condition, namely $\max_{\S^3} f>0$.
				Beyond these, further obstructions arise. These include topological constraints, such as the one in \eqref{GBC-2}, as well as a Kazdan-Warner type identity\footnote{Indeed, we rely on the equivalence between   $P^{3,b}_{g}$ and the third-order Beckner operator $\cP^{3}_{\hat g}$ on $\S^3$, since  $P^{3,b}_g$ is not an intrinsic operator on $\S ^3$; see Section \ref{sec:2} for further discussion.} (see, e.g., \cite[p.205]{CY1995}) 
				\be\label{KW}
				\int_{\S^3}\la \nabla_{g_{\S^3}} f,\nabla_{g_{\S^3}}x_i\ra_{g_{\S^3}}\, \d s_{\hat g}=0,\quad 1\leq i\leq 4,
				\ee where 
				$(x_1,x_2,x_3,x_4)$ denote the restrictions of the standard coordinate functions on  $\R^4$ to $\S^3$. The main objective of this paper is to establish the following existence theorem, which provides a sufficient condition for the solvability of problem \eqref{maineq}.
				\begin{thm}\label{thm:main}
					Suppose that $f: \S ^3\to\R $ is a positive smooth function with only nondegenerate
					critical points  such that $\Delta_{g_{\S ^3}}f(a)\neq 0$
					at any critical  point $a$.
					For $i=0,\ldots,3$, let 
					\be\label{14}
					m_i:=\#\{a\in\S ^3: \nabla_{g_{\S ^3}}f(a)=0, \Delta_{g_{\S ^3}}f(a)<0, \morse(f,a)=3-i \},
					\ee
					where  $\morse(f,a)$ denotes the Morse index of $f$ at the  critical point $a$. 
					If   the following system 
					\be \label{eq:system}
					\left\{\begin{aligned}
						&	m_0=1+k_0,&& \\
						&m_i=k_{i-1}+k_i, && 1 \leq i \leq 3, \\
						&	k_3=0,&&
					\end{aligned}\right.
					\ee 
					has no solution with coefficients $k_i\geq 0$,  then problem \eqref{maineq} admits a solution.
				\end{thm}
				
				\begin{rem}
					The existence result stated in Theorem \ref{thm:main} for the prescribed $T$-curvature problem on compact four-dimensional Riemannian manifolds with boundary was previously obtained by  Ndiaye \cite{N2024}. In this paper, we present an alternative proof based on a geometric flow method, in the spirit of the scheme developed by Malchiodi-Struwe \cite{MS2006}.
				\end{rem}

				As a consequence of the proof of Theorem   \ref{thm:main}, we have the following result.
				\begin{cor}\label{cor:1}
					Under the same assumptions in Theorem \ref{thm:main}. 
					If
					\be \label{13}
					\sum_{\{a\in\S ^3 :\,\nabla_{g_{\S ^3}}f(a)=0, \Delta_{g_{\S ^3}}f(a)<0\}}(-1)^{\morse (f,a)}\neq -1,
					\ee 
					then problem \eqref{maineq} admits a solution.  
				\end{cor}

				Theorem \ref{thm:main} and Corollary \ref{cor:1} are  in the spirit of \cite{S2005,MS2006,XZ2016,H2012,CX2011}and are established using the so-called \emph{flow approach}, which dates back to the pioneering works of Struwe \cite{S2002} and Brendle \cite{B2003}.  
				Moreover, within the same framework, we establish exponential convergence of the
				$T$-curvature flow   \eqref{flow3} (defined in Section \ref{sec:2}) on 
				$\B^4$.  More precisely, starting from a  $Q$-flat and minimal metric conformal to the standard Euclidean metric, the flow converges exponentially fast to an extremal metric for the sharp Ache-Chang inequality \cite{AC2017}.
				
				We introduce the set  \[[g_{\B^4}]_0:= \{g\in [g_{\B^4}]:  Q_g=0 \text{ in }\,\B^4, H_{g}=0 \text{ on }\,\S^3, \vol(\S^3,\hat g)=2\pi^2\}.\]
				The exponential convergence result can  be stated as follows.
				\begin{thm}\label{thm:4.1}
					Suppose that $g_0\in [g_{\B ^4}]_0$, and let $f \equiv  2$.
					Then the $T$-curvature flow   \eqref{flow3} (defined in Section \ref{sec:2})  converges exponentially
					fast to a limiting metric $g_\infty=e^{2w_\infty} g_{\B ^4}$ satisfying  \be\label{thm:4.1-1}
					\left\{\begin{aligned}
						&Q_{g_{\infty}}=0 &&\mbox{ in } \, \B^4,\\
						&H_{g_{\infty}}=0&& \mbox{ on } \, \S^3,\\ &T_{g_{\infty}}=	2&&\mbox{ on } \, \S^3,
					\end{aligned}\right.
					\ee
					in the sense that
					$\|w(t)-w_\infty\|_{H^4(\B ^4, g_{\B ^4})}\leq C e^{-\delta t}$
					for some constants $C>0$ and $\delta>0$. Moreover, the limiting profile $w_\infty$ is explicitly characterized in the work of Ndiaye and Sun \cite{NS2024}.
				\end{thm}

				We now describe the organization of the paper and briefly outline the main ideas of the proofs.
				
				One difficulty in studying problem \eqref{maineq} via the flow  approach is that  the $T$-curvature $T_g$ is defined only on $\S ^3$, whereas the boundary operator $P^{3,b}_g$ is not intrinsic to
				$\S^3$, as it depends on the embedding of $\S^3$ in $(\B^4,g)$.  This non-intrinsic nature complicates the analysis of the evolution of $T_g$.  To address this issue, we exploit the relationship between $P^{3,b}_{g}$ and the third-order Beckner operator $\cP^{3}_{\hat g}$ on $\S ^3$.  We would like also to remark that in general we do not know if there exists conformally covariant intrinsic operator of  order 3 on compact 3-manifolds.
				
				Another analytical difficulty arises from the absence of a homogeneous Neumann boundary condition.  To overcome this, we employ a special metric known as the \emph{adapted
					metric},  originally introduced by Case and Chang \cite{CC2016} in the context of compactifying conformally compact   Poincar\'e-Einstein manifolds. We denote this metric by $g_*$,  which belongs to  the conformal class of the flat metric on the Euclidean unit ball  and has a totally geodesic boundary. In the model case $(\B^4,\S^3)$,  an explicit formula for $g_*$ was provided by Ache-Chang \cite[Proposition 2.2]{AC2017}, namely  
				\be\label{adapted-metric}
				g_{*}=e^{2\rho(x)} g_{\B^4}, \quad \rho(x):=\frac{(1-|x|^2}{2},
				\ee which satisfies 
				\be\label{adapted-metric-property} 
				\left\{
				\begin{aligned}
					&Q_{g_*}=0&&\mbox{ in }\,\B ^4,\\
					&T_{g_*}=T_{g_{\B ^4}}&&\mbox{ on }\,\S ^3,\\
					&H_{g_*}=0 &&\mbox{ on }\,\S ^3,
				\end{aligned}
				\right.
				\ee 
				see \cite[Lemma 6.1]{AC2017}. 
				
				Let  $w:=u-\rho$. Then $g=e^{2u}g_{\B^4}=e^{2w}g_{*}$, and  equation  \eqref{maineq1} is equivalent to finding $w$ satisfying \be \label{maineq2}
				\left\{\begin{aligned}
					&P^4_{g_{*}}w=0&&\mbox{ in }\,\B ^4,\\
					&P^{3,b}_{g_{*}}w+T_{g_{*}}=fe^{3w}&&\mbox{ on }\,\S ^3,\\&\frac{\pa w}{\pa\nu_{g_{*}}}=0&&\mbox{ on }\,
					\S ^3.
				\end{aligned}\right.
				\ee 
				Moreover, using  \eqref{conformal-3}, \eqref{adapted-metric-property} and the  identity $\frac{\pa }{\pa \nu_{g_*}}=e^{-\rho(x)}\frac{\pa }{\pa \nu_{g_{\B ^4}}}$, we may rewrite \eqref{maineq2} as
				\be \label{maineq3}
				\left\{\begin{aligned}
					&P^4_{g_{\B^4}}w=0&&\mbox{ in }\,\B ^4,\\
					&P^{3,b}_{g_{\B^4}}w+T_{g_{\B^4}}=fe^{3w}&&\mbox{ on }\,\S ^3,\\&\frac{\pa w}{\pa\nu_{g_{\B^4}}}=0&&\mbox{ on }\,
					\S ^3,
				\end{aligned}\right.
				\ee
				which, in the flat metric  $g_{\B^4}$, takes the form of a system with a homogeneous Neumann boundary condition.  
				
		We introduce the space
				\[
				\H_g:=\Big\{ \phi\in H^2(\B ^4,g):\,\frac{\pa \phi}{\pa  \nu_{g}}=0\, \mbox{ on }\,\S^3\Big\},
				\]
				It is well known that problem \eqref{maineq3} admits   a variational
				structure. Its associated energy functional is given by
				\be\label{variation}
				E_f [w]:=E[w]-\frac{16\pi^2}{3}\log \Big(\dashint_{\S^3}fe^{3 \hat w}\, \d s _{g_{\S^3}}\Big)
				\ee
				for $w\in\H_{g_{\B^4}} \cap H^4(\B^4,g_{\B^4})$, where 
				\be\label{EW}
				E[w]:=\P_{g_{\B ^4}}^{4, 3}(w,w)+4\int_{\S^3} T_{g_{\B^4}}\hat w\, \d s _{g_{\S^3}}.
				\ee
				with   the bilinear form
				\be\label{P43}
				\P^{4, 3}_{g_{\B ^4}}(\phi, \psi):=\la  P^4_{g_{\B ^4}}\phi, \psi\ra _{L^2(\B ^4, g_{\B ^4})}+2\la  P^{3,b}_{g_{\B ^4}}\phi, \psi\ra _{L^2(\S ^3, g_{\S ^3})},
				\ee 
				for $\phi,\psi\in \H_{g_{\B^4}} \cap H^4(\B^4,g_{\B^4})$.
				By integration by parts, \eqref{P43} admits the representation
				\be\label{P43-1}
				\P^{4, 3}_{g_{\B ^4}}(\phi, \psi)=\int_{\B^4}\Delta_{g_{\B^4}}\phi \Delta_{g_{\B^4}}\psi \, \d v_{g_{\B^4}}+2\int_{\S ^3}
				\la \nabla_{g_{\S ^3}}\phi, \nabla_{g_{\S ^3}} \psi\ra _{g_{\S ^3}} \, \d s _{g_{\S ^3}},
				\ee
			which extends naturally to the whole space $\H_{g_{\B^4}}$.

				We now briefly outline the main strategy of the proofs of  Theorem \ref{thm:main} and Corollary \ref{cor:1}.  Let $u_0\in C^{\infty}(\B^4)$ 
				be such that the conformal metric $g_0:=e^{2 u_0} g_{\B^4}$ belongs to the normalized class $[g_{\B^4}]_0$.
				We consider the one-parameter family of conformal metrics $g(t):=e^{2u(t,\cdot)}g_{\B^4}\in[g_{\B^4}]_0$ evolving by
				\[
				\left\{\begin{aligned}
					&\frac{\pa u}{\pa t}=\al(t) f-T_{g(t)}&&\mbox{ on }\,\S ^3,\\
					&u(0,\cdot)=u_0&& \mbox{ in }\,\B ^4,
				\end{aligned}\right.
				\]
				where the normalization factor $\al(t)$ is determined by
				$\al(t) \int_{\S^3} f \,\d s_{\hat g(t)}=2 \pi^2$. This choice of the coefficient $\alpha(t)$ allows the boundary volume of $(\S^3, 
				\hat g(t))$  to remain independent of $t$. Using \eqref{maineq3}, we reformulate the flow in terms of the function $w$, and prove that the resulting flow \eqref{flow3} exists globally for all $t\ge 0$.  In Section  \ref{sec:4}, we perform the blow-up analysis to show  that one of the following two alternatives must occur: either the flow \eqref{flow3} converges  in $H^{4}(\B^4,g_{\B^4})$, or there will be a concentration phenomenon  and the associated normalized flow $v(t)$ (see \eqref{flow5}) converges to 0 in $H^{4}(\B^4,g_{\B^4})$ as $t \to +\infty$. In the latter case, the evolving boundary metric $\hat g(t)$ concentrates at a single point, while the corresponding normalized boundary metric becomes asymptotically round on $\S^3$. 
				This dichotomy constitutes the concentration--compactness phenomenon for the present flow \eqref{flow3}.

				With this concentration--compactness result at hand, we prove Theorem \ref{thm:main} by a contradiction argument. 
				Starting from Section \ref{sec:5}, we assume that the flow  \eqref{flow3} does not converge; in particular,  $f$ is not realized along the flow. 
				We then show that, in the divergent case, the flow exhibits the following asymptotic behavior as $t\to+\infty$:
				i). the boundary metric $\hat g(t)$ concentrates at critical points $Q$ of $f$ where $\Delta_{\S^3} f(Q)<0$; ii). the energy functional $E_f[w(t)] \to -\frac{16\pi^2}{3}\log f(Q)$, where $w(t)$ is a solution of the flow \eqref{flow3}. Finally, in Section \ref{sec:7}, we complete the proof of Theorem \ref{thm:main} via a Morse-theoretic contradiction argument, which proceeds through the following steps.
				
				\textbf{Step 1.} We introduce the sublevel sets of the energy functional $E_f$, denoted by $E_f^{\beta}$ (see \eqref{sublevelset}). 
				Using the blow-up analysis and the asymptotic expansions obtained in Sections \ref{sec:2}--\ref{sec:6}, we characterize the changes in the topology of $E_f^{\beta}$ near the critical points of $f$, leading to Proposition \ref{prop:5.11}.

				\textbf{Step 2.} Proposition \ref{prop:5.11} implies that $N_0:=E_f^{\beta_0}$ is contractible for a suitable choice of $\beta_0$. 
				Moreover, the flow induces a homotopy equivalence between $N_0$ and a limiting space $N_{\infty}$, whose homotopy type is that of a point with cells of dimension $3-\morse(f,p)$ attached for each critical point $p$ of $f$ satisfying $\Delta_{\S^3}f(p)<0$.

				\textbf{Step 3.} Applying standard Morse theory, we obtain the Morse polynomial identity
				\be\label{94-1}
				\sum_{i=0}^3 t^i m_i
				=
				1+(1+t)\sum_{i=0}^3 t^i k_i,
				\ee 
				where $m_i$ is defined in \eqref{14} and   $k_i\ge 0$. 
				Comparing coefficients on both sides of \eqref{94-1} yields a nontrivial solution to the system \eqref{eq:system}, contradicting the hypothesis of Theorem \ref{thm:main}. 
				This completes the proof of Theorem \ref{thm:main}. 
				Corollary \ref{cor:1} follows immediately from \eqref{94-1}.

				The remainder of the paper is organized as follows. In Section \ref{sec:2}, we introduce the flow equation and
				study its elementary properties. Asymptotic  behavior  of  a  normalized flow  is discussed in Section \ref{sec:3} and
				the blow-up analysis is performed   in Section \ref{sec:4}. Then we derive the spectral decomposition  and complete the proof of Theorem \ref{thm:4.1} in Section \ref{sec:5}. 
				The shadow flow analysis is presented in Section \ref{sec:6}. Finally, Theorem \ref{thm:main} and  
				Corollary \ref{cor:1}  are proved in Section \ref{sec:7}.  
				
				\medskip
				
				\noindent\textbf{Notations.}\quad 
				The main notations used throughout the paper are listed as follows: 
				\begin{itemize} 
					\item $\N =\{1,2,\ldots\}$ denotes the set of positive integers, and 	$\N _0=\N \cup \{0\}$. We also set $\R^+=(0,+\infty)$.
					\item  The symbol $C>0$ denotes a generic positive constant which can vary from different occurrences or equations. 
					and   $C(\al, \beta, \ldots)$ means that the positive constant $C$ depends on $\al, \beta, \ldots$.
					\item For a Riemannian manifold $(M,g)$ with boundary $\pa M$, the induced metric on $\pa M$ is denoted by $\hat g$. 
					For a function $h$ defined on $M$, we write $\hat h:=h|_{\pa M}$ for its restriction to the boundary (e.g., $\pa\B^4=\S^3$).
					\item $\frac{\pa }{\pa \nu_g}$ denotes the outward normal derivative on the boundary $\pa M$ of a Riemannian manifold $(M,g)$.  
					
					\item For $1 \leq p<\infty$, $L^p(M, g)$ denotes the usual Lebesgue space consisting of measurable functions $h$ on $M$ such that $\int_{M}|h|^p\, \d v_g<+\infty$. When $p=\infty$, $L^{\infty}(M, g)$ denotes the space of essentially bounded measurable functions on $M$.		
					
					\item For $k \in \N_0$ and $1 \leq p<\infty$, the Sobolev space $W^{k, p}(M, g)$ is defined by \[W^{k, p}(M, g):=\left\{u \in L^p(M, g): \nabla^j u \in L^p(M, g) \text{ for all } 0 \leq j \leq k\right\},\] where $\nabla^j u$ denotes the $j$-th covariant derivative of $u$ with respect to $g$.	If $p=2$, we write $H^k(M, g):=W^{k, 2}(M, g)$ for simplicity.
					\item  For a Riemannian manifold $(M,g)$, we use $\nabla_g$ and $\Delta_g$ to denote the covariant derivative and the Laplace-Beltrami operator, respectively, and we use $\nabla$ and $\Delta$ to denote the usual gradient and Laplacian in the Euclidean space.	 
					\item For $h\in L^1(M,g)$,  we denote its average over $M$ by $\dashint_{M}h\, \d v_g:=\frac{1}{\vol(M,g)}\int_{M}h\, \d v_g$.
					
					\item  	For $p\in\S^3$, we write $B_r(p):=\{x\in\S^3: \operatorname{dist}_{g_{\S^3}}(x,p)<r\}$ for the geodesic ball in $(\S^3,g_{\S^3})$. For $x\in\R^3$, we denote the Euclidean ball by $B(x,r):=\{y\in\R^3: |y-x|<r\}$. 
					\item	$\morse (h,a)$ represents the Morse index of a function $h$ at its critical point $a$. 
					
					\item 	For a set $A$, the notation $\#A$ denotes its cardinality.
					\item  	 We use the standard asymptotic
					notation: $g = O(f)$ means that $|g/f|\leq C$ for some constant $C > 0$, while $g = o(f)$ means that
					$g/f \to 0$.	
				\end{itemize}

				\section{The flow equation}\label{sec:2}

				Let $f$ be a smooth positive function on $\S^3$  and set $0<m_f:=\inf_{\S^3}  f \leq  M_f:=\sup_{\S^3}f $.
				Motivated by the works of Struwe \cite{S2005}, Brendle \cite{B2003,B2002}  and Malchiodi--Struwe \cite{MS2006},  we consider the following  evolution of conformal metrics $g(t)$ for $t\geq 0$:
				\be\label{flow}
				\left\{
				\begin{aligned}
					&\frac{\pa g}{\pa t}=(\al(t)f -T_g)g && \text{ on }\,\S^3 \mbox{ for all }\,t\geq  0,\\&g(0)=e^{2u_0}g_{\B^4},&&
				\end{aligned}
				\right.
				\ee
				together with the constraints
				\be\label{QHt=0}
				\left\{
				\begin{aligned}
					&Q_g=0&&\mbox{ in }\,\B ^4  \mbox{ for all }\,t\geq  0,
					\\&H_g=0&&\mbox{ on }\,\S ^3\mbox{ for all }\,t\geq  0.\\
				\end{aligned}
				\right.
				\ee 
				Here $\al(t)$ is a real-valued function defined for $t\geq 0$ and $u_0\in C^{\infty}(\B^4)$. 
				Throughout the paper, unless otherwise specified, $T_g$, $Q_g$ and $H_g$  denote respectively the $T$-curvature on $\S^3$, $Q$-curvature on $\B^4$ and mean curvature on $\S^3$,  all computed with respect to the evolving metric $g(t)$.  
				For technical reasons, we assume that the initial metric $g_0:=g(0)$ satisfies
				\be\label{16}
				\vol(\S^3,\hat g_0)=\vol(\S^3, g_{\S^3})=2\pi^2.
				\ee

				Since the flow \eqref{flow} preserves the conformal structure of $(\B^4,g_{\B^4})$, the evolving metric can be written in the form $g(t)=e^{2u(t)}g_{\B^4}$, where
				$u(t)\in C^{\infty}(\B^4)$ satisfies the initial condition $u(0)=u_0$. Here and throughout, we adopt the notation
				$u(t)(x):=u(t,x)$ for all $x\in \B^4$. In terms of the conformal factor $u$, the
				flow \eqref{flow} reduces to the boundary evolution equation
				\be\label{17}
				\left\{
				\begin{aligned}
					&\pa_t  u=(\al(t)f -T_g) &&\text{ on }\,\S^3 \mbox{ for all }\,t\geq  0,\\
					&u(0)=u_0,&&
				\end{aligned}
				\right.
				\ee
				where   $T_g$ is the $T$-curvature of $g(t)$.  By the conformal transformation law of
				$T$-curvature in \eqref{conformal-3}, we have
				\be\label{18}
				T_g=e^{-3u}(P_{g_{\B^4}}^{3,b}u+T_{g_{\B^4}}).
				\ee
				
				For convenience, we choose the factor $\al(t)$ such that the boundary volume (area)
				of $\S^3$ with respect to the conformal metric $\hat g(t)$ remains constant along the flow \eqref{flow}-\eqref{20}.  More precisely, we impose
				\[
				0=\frac{\d }{\d t}\Big(\int_{\S^3} e^{3\hat u}\, \d s _{ g_{\S^3}}\Big)=3\al(t)\int_{\S^3}  fe^{3\hat u}\, \d s _{ g_{\S^3}}-3\int_{\S^3}T_g e^{3\hat u} \, \d s _{ g_{\S^3}}
				\]
				for all $t\geq 0$.   Thus,  a natural choice is
				\be\label{alpha_eq}
				\al(t)=\frac{\int_{\S^3}T_g \, \d s _{\hat g}}{\int_{\S^3} f \, \d s _{\hat g}}\quad \text{ for all }\, t\geq 0.
				\ee
				In particular, it follows from \eqref{16} that
				\be\label{20}
				\vol(\S^3,\hat g)=2\pi^2\quad \text{ for all }\, t\geq 0.	
				\ee
				
				Since $(\B ^4,g)$ is conformally flat, the Weyl tensor of $g$ vanishes, i.e., $W_g\equiv 0$.
				Together with \eqref{GBC-2} and the condition $Q_g=0$ in \eqref{QHt=0}, we obtain
				\[
				\int_{\S ^3}(T_g+Z_g)\, \d s _{\hat g} =4\pi^2\chi(\B ^4)=4\pi^2
				\]
				along the flow \eqref{flow}-\eqref{20}. On the other hand, by \eqref{QHt=0}, the conformal invariance  of \eqref{kap-sum}, and the fact that $\kap _{(P^4_{g_{\B ^4}}, P^{3,b}_{g_{\B ^4}})}=4\pi^2$, we can further deduce that
				\be\label{11}
				\int_{\S ^3}T_g\, \d s _{\hat g} =4\pi^2.
				\ee 
				Combining \eqref{11}  with \eqref{alpha_eq},  we  conclude that
				\be \label{19}
				\al(t) \int_{\S ^3}f\, \d s _{\hat g}=\int_{\S ^3}T_g\, \d s _{\hat g}=4\pi^2\quad \text{ for all }\, t\geq 0.
				\ee

				From \eqref{20} and  \eqref{19},
				we obtain uniform upper and lower bounds for $\al $:
				\be \label{42}
				\frac{2}{M_f} \leq\al \leq\frac{2}{m_f}\quad \text{ for all }\, t\geq 0.
				\ee 
				Differentiating \eqref{19} with respect to $t$  and using \eqref{17}, we derive that 
				\be \label{37}
				\al _t\int_{\S ^3}f\, \d s _{\hat g}+3\al \int_{\S ^3}f(\al  f-T_g)\, \d s _{\hat g}=0,
				\ee
				where $\al _t:=\frac{\d \al}{\d t} $.
				Combining \eqref{37} with \eqref{19}, we further obtain
				\be \label{alpha_t}
				\int_{\S ^3}f(T_g-\al  f)\, \d s _{\hat g}=\frac{4\pi^2}{3}\cdot\frac{\al _t}{\al ^2}.
				\ee 
				
				It follows from \eqref{conformal-3} and \eqref{Q_T_H} that the flow \eqref{flow}-\eqref{20} can be equivalently rewritten in terms of the conformal factor $u$ as
				\be \label{flow2}
				\left\{
				\begin{aligned}
					&P^4_{g_{\B ^4}}u=0&&\mbox{ in }\,\B ^4\mbox{ for all }\,t\geq  0,\\
					&\pa _t u=\al  f-T_g=\al  f-e^{-3u}(P^{3,b}_{g_{\B ^4}}u+T_{g_{\B ^4}})&&\mbox{ on }\,\S ^3\mbox{ for all }\,t\geq  0,\\
					& \frac{\pa  u}{\pa \nu_{g_{\B ^4}}}+H_{g_{\B ^4}}=0&&\mbox{ on }\,\S ^3\mbox{ for all }\,t\geq  0,\\
					&u|_{t=0}=u_0&& \mbox{ in }\,\B ^4.
				\end{aligned}
				\right.
				\ee

				Recall that the adapted metric is given by  $g_*=e^{2\rho(x)}g_{\B^4}$ with $\rho(x)=(1-|x|^2)/2$,  which satisfies the properties in \eqref{adapted-metric-property}. Let $w(t)(x):=u(t)(x)-\rho(x)$, then  $g(t)=e^{2w(t)}g_{*}$. Following the same derivation as in  \eqref{maineq3},  the system  \eqref{flow2} can be rewritten as a flow with homogeneous Neumann boundary condition:
				\be \label{flow3}
				\left\{
				\begin{aligned}
					&P^4_{g_{\B ^4}}w=0&&\mbox{ in }\,\B ^4\mbox{ for all }\,t\geq  0,\\
					&\pa _t w=\al  f-T_g=\al  f-e^{-3w}(P^{3,b}_{g_{\B ^4}}w+T_{g_{\B ^4}})&&\mbox{ on }\,\S ^3\mbox{ for all }\,t\geq  0,\\
					& \frac{\pa  w}{\pa \nu_{g_{\B ^4}}}=0&&\mbox{ on }\,\S ^3\mbox{ for all }\,t\geq  0,\\
					&w|_{t=0}=w_0&& \mbox{ in }\,\B ^4,
				\end{aligned}
				\right.
				\ee 
				where $w_0(x)=u_{0}(x)-\rho(x)$.

				Let $\{\lam_k= k(k+2): k \in \N_0 \}$ be the eigenvalues of $-\Delta_{g_{\S^3}}$. The eigenspace  corresponding to $\lam_k$ has finite dimension $N_k$ and is spanned by spherical harmonics $Y_k^{\ell}$ of degree $k$, where $\ell=1, \ldots, N_k$ (see, e.g., \cite{S1970}). We normalize them such  that $\|Y_k^{\ell}\|_{L^2(\S^3,g_{\S^3})}=1$. The spherical harmonics $\{Y_k^{\ell}\}$ form an orthonormal basis of the Hilbert space $L^2(\S^3,g_{\S^3})$. In particular, given $u \in L^2(\S^3,g_{\S^3})$ we can write
				\be\label{expressionu}
				u=\sum_{k=0}^{\infty} \sum_{\ell=1}^{N_k} u_k^{\ell} Y_k^{\ell}, \quad u_k^{\ell} \in \R,
				\ee
				and $\|u\|_{L^2(\S^3,g_{\S^3})}^2=\sum_{k, \ell}(u_k^{\ell})^2$.

				One can define a third-order Beckner operator $\cP_{g_{\S^3}}^3$ as follows (see, e.g., \cite{B1993,B1995,B1987,CQ1997-1}). Given $u \in L^2(\S^3,g_{\S^3})$ with spherical harmonics expansion as in \eqref{expressionu}, such that
				\[
				\|u\|_{H^3(\S^3,g_{\S^3})}^2:=\|u\|_{L^2(\S^3,g_{\S^3})}^2+\sum_{k=1}^{\infty} \sum_{\ell=1}^{N_k}|u_k^{\ell}|^2(\lam_k+1) \lam_k^2<+\infty,
				\]
				we define
				\[
				\cP_{g_{\S^3}}^3 u:=\sum_{k=0}^{\infty}(\lam_k+1)^{\frac{1}{2}} \lam_k \sum_{\ell=1}^{N_k} u_k^{\ell} Y_k^{\ell} .
				\]
				Notice that on $H^3(\S^3,g_{\S^3})$ the operator $\cP_{g_{\S^3}}^3$ coincides with the operator $P_{g_{\S^3}}^3=(-\Delta_{g_{\S^3}}+1)^{\frac{1}{2}}(-\Delta_{\S^3})$ given by \eqref{operatorP3}, where the operator $(-\Delta_{g_{\S^3}}+1)^{\frac{1}{2}}$ is  understood in terms of spectral decomposition of the Laplace-Beltrami operator:
				\[
				(-\Delta_{g_{\S^3}}+1)^{\frac{1}{2}} u=\sum_{k=0}^{\infty} \sqrt{\lambda_k+1} \sum_{\ell=1}^{N_k} u_k^{\ell} Y_k^{\ell} \quad \,\text { for $u$  as in \eqref{expressionu}.}
				\]
				Therefore, $\cP_{g_{\S^3}}^3$ is the well-known intertwining operator on $\S^3$ (see, e.g., \cite{B1987}).
				
				Define the space
				\[
				H^{\frac{3}{2}}(\S^3,g_{\S^3}):=\Big\{u=\sum_{k=0}^{\infty} \sum_{\ell=1}^{N_k} u_k^{\ell} Y_k^{\ell} \in L^2(\S^3,g_{\S^3}): \sum_{k=0}^{\infty}(\lam_k+1)^{\frac{1}{2}} \lam_k \sum_{\ell=1}^{N_k}|u_k^{\ell}|^2<+\infty\Big\},
				\]
				endowed with the seminorm
				\[
				\|u\|_{\dot{H}^{3 / 2}}^2:=\sum_{k=0}^{\infty}\left(\lambda_k+1\right)^{\frac{1}{2}} \lambda_k \sum_{\ell=1}^{N_k}\left|u_k^{\ell}\right|^2
				\]
				and the norm
				\[
				\|u\|_{H^{3 / 2}}^2:=\|u\|_{L^2}^2+\|u\|_{\dot{H}^{3 / 2}}^2 .
				\]
				
				Since the operator $\cP_{g_{\S^3}}^3$ is self-adjoint and non-negative, one can define its square root $(\cP_{g_{\S^3}}^3)^{\frac{1}{2}}$ for functions $u \in H^{\frac{3}{2}}(\S^3,g_{\S^3})$ by
				\[
				(\cP_{g_{\S^3}}^3)^{\frac{1}{2}} u:=\sum_{k=1}^{\infty}(\lam_k \sqrt{1+\lam_k})^{\frac{1}{2}} \sum_{\ell=1}^{N_k} u_k^{\ell} Y_k^{\ell} .
				\]
		It is worth noting that $\|(\cP_{g_{\S^3}}^3)^{\frac{1}{2}} u\|_{L^2}=\|u\|_{\dot{H}^{3 / 2}}$. 
				
				Chang and Qing \cite{CQ1997-2} showed that  on  $(\B^4,\S^3,g_{\B^4})$, there exists a close relationship
				between the  boundary operator $P^{3,b}_{g_{\B^4}}$ and the Beckner operator $\cP^{3}_{g_{\S ^3}}$. More precisely,  if $w$ satisfies $P^4_{g_{\B^4}}w=0$ in $\B^4$ and we set
				$\hat w=w|_{\S^3}$, then by \cite[Lemma 6.2]{AC2017}  (see also Lemma 3.3, Corollary 3.1 in \cite{CQ1997-2}) one has
				\be \label{Beckner-3-equal}
				\cP^{3}_{g_{\S ^3}}\hat w=P^{3,b}_{g_{\B ^4}}w\quad \text{ on }\, \S^3.
				\ee 
				Since  $Q^3_{g_{\S ^3}}=T_{g_{\B ^4}}=2$ on  $\S^3$,  it follows from \eqref{Q-curvature-eq}, \eqref{conformal-3} and \eqref{Beckner-3-equal} that 
				\be \label{samecurvature}
				T_g=Q^3_{\hat g}\quad \text{ for all }\,g\in [g_{\B^4}],
				\ee 
				where $Q^3_{\hat g}$ is the $Q$-curvature associated with the  Beckner operator $\cP^{3}_{\hat g}$ on $(\S^3,\hat g)$.

				Note that  for any given $\tilde v\in C^{\infty}(\S ^3)$, there exists a unique function $\tilde w\in C^{\infty}(\B ^4)$ solving
				\[
				\left\{
				\begin{aligned}
					&P^4_{g_{\B ^4}}\tilde w=0&&\mbox{ in }\,\B ^4,\\
					&\frac{\pa  \tilde w}{\pa \nu_{g_{\B ^4}}}=0&&\mbox{ on }\,\S ^3,\\
					&\tilde w=\tilde v &&\mbox{ on }\,\S ^3.
				\end{aligned}
				\right.
				\]
				As a result, the system \eqref{flow3} is equivalent to the following evolution equation on  $\S^3$:
				\be \label{flow4}
				\left\{
				\begin{aligned}
					&\pa _t \hat w=\al  f-Q^3_{\hat g}=\al  f-e^{-3\hat w}(\cP^3_{g_{\S ^3}}\hat w+Q^3_{g_{\S ^3}})&&\mbox{ on }\,\S ^3\mbox{ for all }\,t\geq  0,\\
					&\hat w|_{t=0}=\hat w_0&&\mbox{ on }\,\S ^3,
				\end{aligned}
				\right.
				\ee 
				where  $\hat w_0=w_0|_{\S ^3}$, $\hat g=e^{2\hat w}g_{\S ^3}$ and $\al$ satisfies  $\al \int_{\S ^3}f \,\d s_{\hat g}=\int_{\S ^3} Q^3_{\hat g} \,\d s_{\hat g}$.

				Thanks to the equivalence between the flow systems \eqref{flow}-\eqref{20}, \eqref{flow2}, \eqref{flow3}, and \eqref{flow4}, 
				and in particular the homogeneous Neumann boundary condition in \eqref{flow3}, we will carry out computations by freely passing between these formulations whenever convenient. Moreover, as in \cite{MS2006,H2012,CX2011}, the global existence and uniqueness of smooth solutions to the flow \eqref{flow4} can be established by the same argument as in \cite{B2003}. Consequently, we obtain global existence and uniqueness of smooth solutions for the flows \eqref{flow}-\eqref{20}, \eqref{flow2}, \eqref{flow3}, and \eqref{flow4}.

				The following lemma shows that the energy functional $E_f  [w]$, defined in  \eqref{variation}, is non-increasing along
				the flow  \eqref{flow3}.
				\begin{lem}\label{lem:2.1}
					Let $w$ be any smooth solution to the flow  \eqref{flow3}. Then 
					\[
					\frac{\d}{\d t}E_f [w]=-4\int_{\S ^3}(\al  f-T_g)^2\, \d s _{\hat g}.
					\]
					In particular,  $E_f [w]$ is non-increasing along the flow.
				\end{lem}
				\begin{proof}
					Since $\P^{4, 3}_{g_{\B ^4}}$ is a symmetric bilinear form, we have
					\[
					\frac{\d}{\d t}E_f [w]=2\P^{4, 3}_{g_{\B ^4}}(w, \pa _t  w)+4\int_{\S ^3}T_{g_{\B ^4}}\pa _t \hat w\, \d s _{ g_{\S^3}}-16\pi^2\frac{\int_{\S ^3}fe^{3\hat w}\pa _t \hat w\, \d s _{ g_{\S ^3}}}{\int_{\S ^3}fe^{3\hat w}\, \d s _{ g_{\S ^3}}}.
					\]
					Using  \eqref{P43}, \eqref{alpha_eq}, \eqref{19} and \eqref{flow3}, we deduce
					\begin{align*}
						\frac{\d}{\d t}E_f [w]=&\,
						2\la  P^4_{g_{\B ^4}}w, \pa _t w\ra _{L^2(\B ^4, g_{\B ^4})}+4\la  P^{3,b}_{g_{\B ^4}}w, \pa _t \hat w\ra _{L^2(\S ^3,  g_{\S ^3})}\\
						&\,+4\int_{\S ^3}T_{g_{\B ^4}}\pa _t \hat w\, \d s _{ g_{\S^3}}-16\pi^2\frac{\int_{\S ^3}f \pa _t \hat w\, \d s _{\hat g} }{\int_{\S ^3}f \, \d s _{\hat g} }\\
						=&\,4\int_{\S ^3}(P^{3,b}_{g_{\B ^4}}w+T_{g_{\B ^4}})\pa _t \hat w\, \d s _{ g_{\S^3}}
						-4\al \int_{\S ^3}f \pa _t \hat w\, \d s _{\hat g}\\
						=&\,4\int_{\S ^3}(T_g-\al  f) \pa _t \hat w\, \d s _{ \hat g}
						=-4\int_{\S ^3}(\al  f-T_g)^2\, \d s _{ \hat g},
					\end{align*}
					which completes the proof.
				\end{proof}
				
				We now establish a uniform bound for the energy functional  $E_f[w]$  along the flow  \eqref{flow3}, using a Lebedev-Milin  type inequality due to  Ache-Chang \cite[Theorem B]{AC2017}.
				\begin{cor}\label{cor:E_f w_upperbdd}
					Let $w$ be any smooth solution to the flow   \eqref{flow3}. Then  
					\[-\frac{16\pi^2}{3}\log  M_f \leq E_f [w]\leq E_f [w_0]<+\infty,\]where $M_f:=\max_{\S ^3} f $ and $w_0\in H^4(\B^4,g_{\B^4})$. 
				\end{cor}

				\begin{proof}
					For each $t > 0$, it follows from Lemma \ref{lem:2.1} that 
					$E_f [w(t)]\leq E_f [w(0)]=E_f [w_0]<+\infty$
					for any initial data $w_0\in H^4(\B^4,g_{\B^4})$. On the other hand, by the sharp inequality of Ache-Chang \cite[Theorem B]{AC2017}, which is
					a generalization of the classical Lebedev–Milin inequality in \cite{LM1951}, we have 
					\[
					\log \Big(\dashint_{\S ^3}e^{3\hat w}\, \d s _{ g_{\S ^3}}\Big)\leq \frac{3}{16\pi^2}E[w].
					\]
					Consequently, we obtain the uniform lower bound
					\[
					E_f [w]\geq -\frac{16\pi^2}{3}\log\big(\max_{\S ^3} f \big),
					\]which completes the proof.
				\end{proof}

				Lemma \ref{lem:2.1} together with Corollary \ref{cor:E_f w_upperbdd} yields the estimate
				\be\label{21}
				\int_{0}^{+\infty}\int_{\S ^3}(\al  f-T_{g})^2\, \d s _{\hat g}\, \d t\leq \frac{1}{4}\Big(E_f [w_0]+\frac{16\pi^2}{3}\log M_f\Big)<+\infty.
				\ee 	
				Thus, there exists a sequence $\{t_k\}$ with $t_k\to +\infty$ as $k\to +\infty$, and the  corresponding metrics $g_k=g(t_k)$, such that 
				\be\label{22}
				\int_{\S ^3}(\al(t_k)  f -T_{g_{k}})^2\, \d s _{\hat g_{k}}\to 0\quad \text{ as }\, k\to +\infty.
				\ee 
				Therefore, if $\{g_k\}$ converges (up to a subsequence) to a metric $g_{\infty}\in [g_{\B^4}]$, then
				the $Q$-curvature, $T$-curvature, and mean curvature of $(\B ^4, g_{\infty})$
				satisfy
				\[\left\{\begin{aligned}
					&Q_{g_\infty}=0&&\mbox{ in }\,\B ^4,\\ & T_{g_\infty}=\al _{\infty}f   &&\mbox{ on }\, \S ^3,
					\\ & H_{g_\infty}=0&& \mbox{ on }\, \S ^3,
				\end{aligned}\right.
				\]
				where $\al _\infty>0$ is a positive constant in view of \eqref{42}. In particular, if $g_{\infty}=e^{2u_{\infty}}g_{\B ^4}$, then the conformally related metric  $\bar g_{\infty}=e^{2\bar u_{\infty}}g_{\B ^4}$, with  \[\bar u_{\infty}:=u_{\infty} +\frac{1}{3}\log \al _{\infty}\] provides a solution to problem  \eqref{maineq}, thereby prescribing $f$ as the  $T$-curvature. On the other hand, if $\{g_k\}$  does not converge along any sequence $t_k\to +\infty$, one expects a concentration phenomenon analogous to that studied in \cite{S2005,MS2006,CX2011,H2012,XZ2016}. 

	\section{Asymptotic  behavior  of the normalized flow}\label{sec:3}
	
	In this section, we investigate the asymptotic behavior of the solution $w(t)$ to the flow \eqref{flow3} and of the corresponding conformal metric $g(t)$ as $t\to +\infty$.

	\subsection{Normalized flow}\label{sec:3.1}

	Following \cite{XZ2016}, in the setting of  $(\B^4,\S^3,g)$, let $g=g(t):=e^{2u(t)}g_{\B^4}$ be a smooth family of metrics as  in Section \ref{sec:2}.  Then there exists a family of conformal diffeomorphisms
	\[
	\Phi(t):(\B^4,\S^3)\mapsto (\B^4,\S^3),
	\] mapping  $\S^3$ onto itself, such that
	\be\label{23}
	\int_{\S^3}x\, \d s _{\hat h}=0\quad \text{ for all }\, t\geq 0, 
	\ee
	where  $x=(x_1,x_2,x_3,x_4)$ is the restriction to $\S^3$ of the standard coordinate
	functions on $\R^4$, and $ h(t):=\Phi^*(t) g(t)$.   For notational simplicity, we  write $g=g(t)$, $h=h(t)$, and $\Phi=\Phi(t)$. The metric $h=\Phi^* g$ induced by this normalization \eqref{23} will be referred to as the  \emph{normalized metric}.

	Note that $g=e^{2w(t)}g_*$, where $g_*$ is the adapted metric defined in \eqref{adapted-metric}. Accordingly, if we write $h:=e^{2v}g_*$,  and denote by $\hat v:=v|_{\S ^3}$, $\hat w:=w|_{\S ^3}$ and $\hat \Phi:=\Phi|_{\S ^3}$ the corresponding boundary restrictions, then
	\be \label{24}
	\hat v=\hat w\circ \hat \Phi+\frac{1}{3}\log\det (\d\hat \Phi).
	\ee 
	Moreover, letting $\hat h:=h|_{\S ^3}$
	and using \eqref{20}, we obtain
	\be \label{26}
	\vol (\S ^3,\hat{h})=\vol (\S ^3,\hat{g})=2\pi^2\quad \text{ for all }\, t\geq 0.
	\ee 
	
	Since $\Phi$ is a conformal diffeomorphism and $ h=\Phi^* g$, the associated curvature quantities transform naturally under  pullback. In particular,  in view of \eqref{QHt=0},  we have 
	\be\label{QHT-conformal}
	\left\{
	\begin{aligned}
		&Q_h=Q_g\circ\Phi=0&&\mbox{ in }\,\B ^4 ,
		\\&H_{h}=H_g\circ \hat \Phi=0&&\mbox{ on }\,\S ^3,\\&T_h=T_g\circ\hat \Phi&&\mbox{ on }\,\S ^3.
	\end{aligned}
	\right.
	\ee 
	Now consider the conformal change of metric $h=e^{2v}g_*$.  It follows from \eqref{conformal-3}, \eqref{adapted-metric-property}, and \eqref{QHT-conformal} that $v$ satisfies
	\be \label{flow-g*}
	\left\{\begin{aligned}
		&P_{g_*}^4v=0&&\mbox{ in }\,\B ^4,\\
		&P_{g_*}^{3,b}v+T_{g_*}=T_h e^{3 v}&&\mbox{ on }\,\S ^3,
		\\ & \frac{\pa  v}{\pa \nu_{g_*}}=0&&\mbox{ on }\,\S ^3.	
	\end{aligned}\right.
	\ee 
	Since  $g_{*}=e^{2\rho}g_{\B^4}$ as in \eqref{adapted-metric}, we may rewrite \eqref{flow-g*} in terms of the background metric $g_{\B^4}$. Indeed, using again \eqref{conformal-3} together with \eqref{adapted-metric-property}, and noting that $\rho|_{\S^3}=0$, we deduce that  $v$ also satisfies
	\be \label{flow5}
	\left\{\begin{aligned}
		&P_{g_{\B ^4}}^4v=0&&\mbox{ in }\,\B ^4,\\
		&P_{g_{\B ^4}}^{3,b}v+T_{g_{\B ^4}}=T_h e^{3 v}&&\mbox{ on }\,\S ^3,\\ &\frac{\pa  v}{\pa \nu_{g_{\B ^4}}}=0&&\mbox{ on }\,\S ^3.
	\end{aligned}\right.
	\ee 
	Moreover, by virtue of \eqref{samecurvature}, we have
	\be \label{2.6}
	\int_{\S ^3}(T_h-\al  f\circ\hat \Phi) \cP^{3}_{\hat h}(T_h-\al  f\circ\hat\Phi)\, \d s _{\hat h}
	=\int_{\S ^3}(T_g-\al  f)\cP^{3}_{\hat g}(T_g-\al  f )\, \d s _{\hat g} .
	\ee 
	
	Let us define a functional on $H^{\frac{3}{2}}(\S ^3,g_{\S^3})$ by
	\[
	\hat{E}[\tilde \phi]:=\la \cP^{3}_{g_{\S^3}}\tilde \phi, \tilde \phi\ra_{L^2(\S^3,g_{\S^3})} +2\int_{\S^3}Q^3_{g_{\S ^3}}\tilde \phi \, \d s_{g_{\S ^3}}.
	\]
	By conformal invariance, it holds
	\be \label{2.7}
	\hat {E}[\hat v]=\hat {E}[\hat w],
	\ee 
	see \cite[Lemma 2.2]{XZ2016} or step 1 in the  proof of Theorem 4.1 in  \cite{CY1997} for further details.  Moreover, arguing as in \eqref{Beckner-3-equal}, we deduce from  the first equation in  \eqref{flow5} together with   \cite[Lemma 6.2]{AC2017} that 
	\be \label{samev}
	\cP^{3}_{g_{\S ^3}}\hat v=P^{3,b}_{g_{\B ^4}}v\quad \text{ on }\, \S^3.
	\ee 
	Consequently, using \eqref{EW}-\eqref{P43-1}, \eqref{Beckner-3-equal},  \eqref{flow5} and \eqref{samev} (recalling that  $T_{g_{\B^4}}=Q^3_{g_{\S^3}}$), we obtain
	\[
	E[v]= 
	2\hat {E}[\hat v]\quad \text{ and }\quad E[w]=2\hat {E}[\hat w],
	\]
	which, together with \eqref{2.7}, implies that 
	\be \label{sameenergy_vw}
	E[v]=E[w].
	\ee

	The Ache-Chang inequality \cite[Theorem B]{AC2017} also applies to $v$.
	Moreover, thanks to the vanishing center-of-mass condition  \eqref{23}, we can derive the following strengthened version.
	\begin{prop}\label{prop:3.1}
		There exists a constant  $a<1$ such that  for every $v\in H^2(\B^4,g_{\B^4})$ solving the system \eqref{flow5} and the corresponding metric $h=e^{2v}g_{*}$ satisfying \eqref{23}, one has
		\[
		\log \Big(\dashint_{\S ^3}e^{3\hat v}\, \d s _{ g_{\S ^3}}\Big)\leq \frac{3}{16\pi^2}\Big[a\P^{4, 3}_{g_{\B ^4}}(v,  v)+4\int_{\S ^3}T_{ g_{\B ^4}}\hat v\, \d s _{ g_{\S ^3}}\Big].
		\]
	\end{prop}
	\begin{proof}
		By \eqref{23} and Theorem 2.6 in \cite{WX1998}, there exists a constant  $a<1$ such that
		\begin{align*}
			\log \Big(\dashint_{\S ^3}e^{3\hat v} \, \d s_{g_{\S ^3}}\Big) \leq&\, \frac{3}{4}\Big[a\dashint_{\S ^3}\hat v\cP^{3}_{g_{\S ^3}}  \hat v\,\d s_{g_{\S ^3}} +4\dashint_{\S ^3}\hat v \,\d s_{g_{\S ^3}}\Big]
			\\=&\,\frac{3}{4}\Big[\frac{a}{2\pi^2}\la \cP^{3}_{g_{\S ^3}}  \hat v,\hat v\ra_{L^2(\S^3,g_{\S ^3})}  +4\dashint_{\S ^3}\hat v \,\d s_{g_{\S ^3}}\Big].
		\end{align*}
		The desired estimate then follows from \eqref{P43}, \eqref{samev}, and the first equation in   \eqref{flow5}.
	\end{proof}

	Consequently, we obtain the following a priori estimates for solutions  $v$ in \eqref{flow5} by adapting the argument of  \cite[Lemma 3.2]{MS2006}.
	\begin{lem}\label{lem:3.2}
		Let $v$ be any smooth solution of  \eqref{flow5}. Then there exist some  positive constants $C$ and $C(\si)$, independent of $t$, such that 
		\[ 
		\sup_{t\geq 0} \|v\|_{H^2(\B ^4,g_{\B^4})}\leq C,
		\]
		and, for any $\si\in \R $,
		\[
		\sup_{t\geq 0}\int_{\S^3}e^{3\si \hat v} \, \d s_{g_{\S^3 }} \leq C(\si).
		\]
	\end{lem}

	\begin{proof}
		By Lemma 3.2 in \cite{MS2006} (see also  \cite{H2012,CX2011}), we have 
		\[ 
		\sup_{t\geq 0} \|\hat v\|_{H^{\frac{3}{2}}(\S ^3,g_{\S^3})}\leq C,
		\]
		and, for any $\si\in \R $,
		\[ 
		\sup_{t\geq 0}\|e^{3\si \hat v}\|_{L^1(\S ^3,g_{\S^3})}\leq C(\si),
		\]
		where $C$ and $C(\si)$ are positive constants independent of $t$.  Lemma  \ref{lem:3.2} follows from  \eqref{flow5} and standard elliptic regularity for the biharmonic
		operator with Neumann boundary condition.
	\end{proof}

Using the uniform $H^{2}(\B^4,g_{\B^4})$-bound for $v$ obtained in Lemma \ref{lem:3.2},
and following the scheme of \cite[Section 4]{B2003} together with standard elliptic
regularity, one can further derive higher-order a priori estimates. In particular, for
each $T>0$ there exists a constant $C_T>0$ such that $\sup_{t\in[0,T]}\|v(t)\|_{H^{4}(\B^4,g_{\B^4})}\le C_T$.
Consequently, by the continuation argument in \cite[Section~5]{B2003}, the solution
exists for all $t\ge0$. By the equivalence of the formulations established above, global
existence also holds for the flows \eqref{flow2}, \eqref{flow3}, \eqref{flow4},
\eqref{flow-g*}, and \eqref{flow5}.

	\subsection{Curvature evolution and its $L^2$-convergence}
	In this subsection, we establish the evolution equation for the $T$-curvature along the flow and study the $L^2$-norm of the flow speed.
	
	Using \eqref{conformal-3},  \eqref{flow3},  and \eqref{Beckner-3-equal}, we  compute the evolution of the $T$-curvature on $\S^3$:
	\[
	\pa _t T_g=\pa_t(e^{-3 \hat w}(\cP^{3}_{g_{\S ^3}}\hat w+T_{g_{\B ^4}}))=-3(\pa _t \hat w) T_g+\cP^3_{\hat g}(\pa _t \hat w)
	=3T_g(T_g-\al  f)-\cP^{3}_{\hat g}(T_g-\al  f).
	\]
	Combining this with  \eqref{17}, \eqref{alpha_t}, and \eqref{flow3},   we obtain
	\begin{align}
		&\,\frac{\d}{\d t}\Big(\int_{\S ^3}|T_g-\al  f|^2\, \d s _{\hat g}\Big)\nonumber\\
		=&\,2\int_{\S ^3}(T_g-\al  f)(\pa _t T_g-\al _t f )\, \d s _{\hat g}
		-3\int_{\S ^3}(T_g-\al  f)^3\, \d s _{\hat g}\nonumber\\
		=&\,6\int_{\S ^3}T_g(T_g-\al  f)^2\, \d s _{\hat g}-2\int_{\S ^3}(T_g-\al  f)\cP^{3}_{\hat g}(T_g-\al  f)\, \d s _{\hat g}
		\nonumber	\\
		&\,
		-\frac{8\pi^2}{3}\Big(\frac{\al _t}{\al }\Big)^2-3\int_{\S ^3}(T_g-\al  f)^3\, \d s _{\hat g}.\label{38}
	\end{align}

	Next, we show the convergence of \eqref{22} is in fact uniform in time $t$. 
	
	For $t\geq 0$,  we introduce the quantities
	\be\label{62}
	F_2(t):=\int_{\S ^3}|T_g-\al  f|^2\, \d s _{\hat g}=\int_{\S ^3}|T_h-\al  f_{\hat \Phi}|^2\, \d s _{\hat h} 
	\ee
	and
	\be \label{63}
	G_2(t)
	:=\int_{\S ^3}(T_g-\al  f)\cP^{3}_{\hat g}(T_g-\al  f )\, \d s _{\hat g} 
	=\int_{\S ^3}(T_h-\al  f_{\hat \Phi})\cP^{3}_{\hat h}(T_h-\al  f_{\hat \Phi})\, \d s _{\hat h},
	\ee 
	in view of \eqref{2.6},  where  $f_{\hat\Phi}:=f\circ\hat\Phi$.

	\begin{lem}\label{lem:3.4}
		Along the equivalent flows \eqref{flow2}, \eqref{flow3}, \eqref{flow4}, \eqref{flow-g*}, and \eqref{flow5}, we have  \[F_2(t)\to 0 \quad \text{ as }\, t\to +\infty.\]
	\end{lem}
	\begin{proof}
		Note that $T_h=T_g\circ \hat \Phi $  for the normalized metric $h=\Phi^*g$.
		By the naturality of the Beckner operator under conformal diffeomorphisms, and by rewriting  \eqref{38} in the normalized coordinates, we obtain
		\begin{align}
			&\,\frac{\d F_2(t)}{\d t }+\frac{8\pi^2}{3}\Big(\frac{\al _t}{\al }\Big)^2\nonumber\\=&\,6\int_{\S ^3}T_h(T_h-\al  f_{\hat \Phi})^2\, \d s _{\hat h}-2\int_{\S ^3}(T_h-\al  f_{\hat \Phi})\cP^{3}_{\hat h}(T_h-\al  f_{\hat \Phi})\, \d s _{\hat h}
			\nonumber\\
			&\,-3\int_{\S ^3}(T_h-\al  f_{\hat \Phi})^3\, \d s _{\hat h}
			\nonumber	\\=&\,-2G_2(t) 
			+3\int_{\S ^3}(T_h-\al  f_{\hat \Phi})^3\, \d s _{\hat h} 
			+6\al \int_{\S ^3}f _{\hat \Phi}(T_h-\al  f_{\hat \Phi})^2\, \d s _{\hat h} .\label{39}
		\end{align}
		For any $\var _0>0$.  By \eqref{22}, there exists $t_0$ arbitrarily large such that
		\be \label{40}
		F_2(t_0)<\var_0.
		\ee 
		Then we choose $t_1> t_0$ such that
		\be \label{41}
		\sup_{t\in [t_0,t_1)} 	F_2(t)\leq 2\var _0.
		\ee 
		Now using \eqref{42}, \eqref{26} and \eqref{41}, we obtain  for $t\in [t_0,t_1)$ that
		\[
		\|T_h\|_{L^2(\S ^3,\hat h)} 
		\leq F_2(t)^{1/2}+\al  \|f_{\hat \Phi}\|_{L^2(\S ^3,\hat h)}
		\leq \sqrt{2\var _0}+\sqrt{2}\pi\frac{M_f }{m_f }:=C(f).
		\]
		This, together with H\"older inequality and Lemma \ref{lem:3.2},  implies that the right-hand side
		  of $\cP^{3}_{g_{\S ^3}}\hat v+T_{g_{\B ^4}}=T_h e^{3\hat v}$
		is uniformly bounded in $L^p(\S ^3,g_{\S ^3})$ for every $p\in (1,2)$.
		Standard $L^p$-estimates  then yield that
		$\hat v$ is bounded in $W^{3,p}(\S ^3,g_{\S ^3})$ for all $p\in (1,2)$,
		and hence also in $L^\infty(\S ^3)$.
		In particular, there exists a  constant $C>1$, independent of $t$,  such that
		\be \label{43}
		C^{-1}g_{\S ^3}\leq \hat h\leq Cg_{\S ^3}\quad \text{ for }\, t\in [t_0,t_1),
		\ee 
		where the inequality is understood in the sense of quadratic forms, i.e., uniformly for all tangent vectors.

		By  H\"older inequality, \eqref{43}, and the Sobolev's embedding $H^{\frac{3}{2}}(\S ^3,g_{\S ^3})
		\hookrightarrow L^4(\S ^3,g_{\S ^3})$, we deduce that
		\begin{align}		
			\int_{\S ^3}|T_h-\al  f_{\hat \Phi}|^3\, \d s _{\hat h} 
			\leq&\, \|T_h-\al  f_{\hat \Phi}\|_{L^2(\S ^3,\hat h)}\|T_h-\al  f_{\hat \Phi}\|^2_{L^4(\S ^3,\hat h)}\nonumber\\
			\leq &\,C\|T_h-\al  f_{\hat \Phi}\|_{L^2(\S ^3,\hat h)}\|T_h-\al  f_{\hat \Phi}\|^2_{L^4(\S ^3,g_{\S ^3})}\nonumber\\
			\leq&\, C\|T_h-\al  f_{\hat \Phi}\|_{L^2(\S ^3,\hat h)}\|T_h-\al  f_{\hat \Phi}\|^2_{H^{\frac{3}{2}}(\S ^3,g_{\S ^3})}\nonumber\\
			\leq &\, C\|T_h-\al  f_{\hat \Phi}\|_{L^2(\S ^3,\hat h)}
			\Big(\int_{\S ^3}(T_h-\al  f_{\hat \Phi})\cP^{3}_{g_{\S ^3}}(T_h-\al  f_{\hat \Phi})\, \d s _{g_{\S ^3}}\nonumber\\
			&\,+
			\int_{\S ^3}(T_h-\al  f_{\hat \Phi})^2\, \d s _{g_{\S ^3}}\Big)\nonumber\\
			\leq &\, C_0F_2(t)^{1/2}(G_2(t)+F_2(t)). \label{44}
		\end{align}
		If we  choose $\var _0>0$ such  that
		$18\var _0C_0^2\leq 1$,
		it follows from \eqref{42}, \eqref{39} and \eqref{44} that
		\[
		\frac{\d F_2(t)}{\d t}\leq (6\al M_f +1)F_2(t)	\leq C_1F_2(t) ,
		\]
		where $C_1:=  12 M_f/m_f+1$.
		The same estimate also holds in the original coordinates for the functions $w$ and $f$
		instead of $v$ and $f_{\hat \Phi}$.
		Integrating the above differential equation from $t_0$ to $t$ with $t\in [t_0,t_1)$, we then obtain
		\[
		F_2(t) \le F_2(t_0)+C_1\int_{t_0}^{t}F_2(s)\,\d s
		\leq  F_2(t_0) 
		+C_1\int_{t_0}^{+\infty}F_2(t) \,\d t.
		\]
		By \eqref{21} and \eqref{40}, the right hand side is less than $2\var _0$
	provided $t_0$ is chosen sufficiently large and satisfies \eqref{41}.
	Hence \eqref{41}   holds for all $t\geq t_0$.
		
		Finally, letting $t_0\to+\infty$ suitably, we conclude that
		\[\limsup_{t\to+\infty }F_2(t)
		\leq\,\limsup_{t_0\to+\infty}
		\Big(F_2(t_0)
		+C_1\int_{t_0}^{+\infty}F_2(t) \,\d t\Big)=0,
		\]
	which proves that $F_2(t)\to0$ as $t\to+\infty$.
	\end{proof}

	Differentiating \eqref{24} with respect to $t$ yields  
	\be \label{33}
	\pa _t \hat v=\pa _t \hat w\circ \hat \Phi+\frac{1}{3}e^{-3\hat v}\operatorname{div}_{g_{\S ^3}}(\hat\xi e^{3\hat v}),
	\ee 
	where the vector field $\hat\xi$ is defined by $\hat\xi:=\xi|_{\S ^3}=(\d\hat \Phi)^{-1}\frac{\d\hat \Phi}{\d t}$
	and  $\xi$ is given by $\xi:=(\d\Phi)^{-1}\frac{\d\Phi}{\d t}$;
	see, for instance, (33) in \cite{MS2006} or its analogue in \cite{H2012,CX2011}.

	Finally, we estimate the $L^{\infty}$-norm of the conformal vector fields  $\hat \xi(\cdot,t)$ in terms of $F_2(t)$.

	\begin{lem}\label{35}
		There exists a uniform constant $C>0$ such that
		\[
		\|\hat \xi(\cdot,t)\|_{L^{\infty}(\S ^3)}\leq C F_2(t)^{1/2}\quad \text{ for all  }\, t\geq 0.
		\]
	\end{lem}
	\begin{proof}	
		Differentiating \eqref{23} with respect to $t$
		and using   \eqref{33}, we obtain
		\begin{align}
			0=&\,\frac{\d}{\d t}\Big(\int_{\S ^3}x\, \d s _{\hat h}\Big)=\, 3\int_{\S ^3}x\pa _t \hat v\, \d s _{\hat h}\notag\\
			=&\,3\int_{\S ^3}x\pa _t \hat w\circ \hat\Phi \, \d s _{\hat h}+\int_{\S ^3}x\,\operatorname{div}_{g_{\S ^3}}(\hat \xi e^{3\hat v})\, \d s _{g_{\S ^3}}\notag\\
			=&\, 3\int_{\S ^3}x\pa _t \hat w\circ \hat\Phi \, \d s _{\hat h}-\int_{\S ^3}\hat \xi \, \d s _{\hat h}.\label{34}
		\end{align}
		Following the argument in  \cite{NS2024} (see also \cite{H2012,CX2011}), and combining \eqref{34} with \eqref{flow3} and Lemma \ref{lem:3.2}, we obtain the uniform estimate
		\[
		\|\hat \xi(\cdot,t)\|_{L^{\infty}(\S ^3)}^2\leq C\int_{\S ^3}|\pa _t \hat w\circ \hat \Phi|^2\, \d s _{\hat h}
		=C\int_{\S ^3}|\al  f-T_{g}|^2\, \d s _{\hat g},
		\]
		which completes the proof.
	\end{proof}

	\section{Blow-up analysis}\label{sec:4}	
	As in  \cite{MS2006,H2012,XZ2016,S2005,CX2011}, suppose that the flow \eqref{flow3} does not converge for any initial data $w_0$ with
	$E_f  [w_0] \leq  \beta$ for a suitable real number $\beta$ to be determined later, then the standard
	Morse theory argument leads to conclude that system \eqref{eq:system} would hold for some nonnegative
	$k_i$ with $0 \leq i \leq 3$, which contradicts our assumption. As the first step for
	this purpose, we need to focus on the blow-up analysis.

In view of \eqref{samecurvature} and Lemma \ref{lem:3.4},  we adapt the argument of  \cite[Proposition 1.4]{B2003} to give a  characterization of the asymptotic behavior of sequences of functions $\{w_k\}$ whose $T$-curvatures  are bounded in $L^2(\S^3, \hat g_k)$.

	\begin{lem}\label{lem:3.5}
		Let $\{w_k\}$ be a sequence of smooth functions on $\B^4$ with associated metrics $g_k=e^{2 w_k} g_{*}$. Assume that $Q_{g_k}=0$ in $\B^4$, $H_{g_k}=0$ on $\S^3$, $\vol(\S^3, \hat g_k)=2\pi^2$ and $\|T_{g_k}\|_{L^2(\S^3, \hat g_k)} \leq C$ for some constant $C>0$ independent of $k$. Then either there exists a subsequence (still denoted by $\{w_k\}$) which is bounded in $H^4(\B^4,g_{\B^4})$, or for every sequence $k \to  +\infty$ we can find a subsequence  (still denoted by $\{w_k\}$) and some points $p_1, \ldots, p_L \in \S^3$ such that for every $r>0$ and any $i \in\{1, \ldots, L\}$ there holds	
		\be\label{47}
		\liminf _{k \to  +\infty} \int_{B_r(p_i)}|T_{g_k}| \, \d s_{\hat g_k} \geq 2 \pi^2.
		\ee
		In the latter case,  a subsequence $\{w_k\}$ either converges in $H_{loc}^4(\B^4 \backslash\{p_1, \ldots, p_L\},g_{\B^4})$, or $w_k \to -\infty$ locally uniformly away from $p_1, \ldots, p_L$ as $k \to +\infty$.
	\end{lem}
	
	\begin{proof}
		Using the relation \eqref{samev}, the proof follows closely the arguments in 
		\cite[Proposition 1.4]{B2003} or \cite[Theorem 3.2]{S2002}. 
		The main difference  in the present setting is that one must address additional difficulties arising from the 
		pseudo-differential operator $\cP^{3}_{\hat g}$. 
		These issues will be treated in the proofs of the next two lemmas. 
		For brevity, we omit the details here.
	\end{proof}
	
	Following the same  strategy as the proof of Lemma 3.6 in \cite{MS2006}, we  have
	\begin{lem}\label{lem:3.6-1}
		Let $\{\hat w_k\}$ be a sequence of smooth functions on $\S ^3$, and let 
		$\hat g_k=e^{2\hat w_k}g_{\S ^3}$ be the corresponding sequence of metrics.
		Assume that $\vol (\S^3,\hat{g}_k)=2\pi^2$ and 
		$\|Q^3_{\hat g_k}-Q^3_\infty\|_{L^2(\S ^3,\hat{g}_k)}\to 0$ as $k\to+\infty$
		for some smooth positive function $Q^3_\infty$ on $\S ^3$.
		Let $\hat\Phi_k:\S^3\to\S^3$ be conformal diffeomorphisms such that, with
		$\hat h_k:=\hat\Phi_k^*\hat g_k=e^{2\hat v_k}g_{\S^3}$, the normalization  \eqref{23}
		holds. Then, after passing to a subsequence, one of the following two possibilities must occur:\\
		\\  (i) Compactness: $\hat w_k\to \hat w_\infty$ in $H^{\frac{3}{2}}(\S ^3,g_{\S^3})$, and
		$Q^3_{\hat g_\infty}=Q^3_\infty$ for $\hat g_\infty:=e^{2\hat w_\infty}g_{\S^3}$.
		\\	(ii) Concentration: There exists a point $p\in \S ^3$ such that 
		\be\label{48-1} 
		\d s_{\hat g_k}\rightharpoonup 2\pi^2\delta_p\quad \mbox{ as }\, k\to+\infty
		\ee 
		in the weak sense of measures, and 
		\[
		\left\{\begin{aligned}
			&	\hat v_k\to 0&&\mbox{ in } \, H^{\frac{3}{2}}(\S ^3,g_{\S ^3}),\\
			&Q^3_{\hat h_k}\to 2&& \mbox{ in } \, L^2(\S ^3,g_{\S ^3}).	
		\end{aligned}\right.
		\]
		In this case,  the boundary maps $\hat \Phi_k$ converge weakly in $H^{\frac{3}{2}}(\S ^3,g_{\S ^3})$ 
		to the constant map $ p$. 
	\end{lem}
	
	\begin{proof} 
		First, Lemma \ref{lem:3.5} can be applied to $\hat w_k$ in view of \eqref{samecurvature}. If $\hat w_k$ is uniformly bounded in $H^{\frac{3}{2}}(\S^3, g_{\S^3})$, from the assumption  $Q^3_{\hat g_k} \to  Q^3_{\infty}$ in $L^2(\S^3, g_{\S^3})$ as $k \to + \infty$, we may use \eqref{Q-curvature-eq}
		and standard elliptic estimates for the pseudo-differential operator $\cP^3_{g_{\S^3}}$
		to obtain, after passing to a subsequence, $\hat w_k \to  \hat w_{\infty}$ in $H^{\frac{3}{2}}(\S^3, g_{\S^3})$ as $k \to  +\infty$.  Consequently, $\hat g_k=e^{2\hat w_k}g_{\S^3}\to \hat g_\infty:=e^{2\hat w_\infty}g_{\S^3}$ in
		$H^{\frac{3}{2}}(\S^3,g_{\S^3})$ as $k\to+\infty$.
		
		Otherwise, if concentration occurs in the sense of \eqref{47}, we need to show the desired behavior in (ii). For each $k$, select $p_k \in \S^3$ and $r_k>0$ such that
		\be\label{49-1}
		\sup _{p \in \S^3} \int_{B_{r_k}(p)}|Q^3_{\hat g_{k}}| \d s_{\hat g_k}=\int_{B_{r_k}(p_k)}|Q^3_{\hat g_{k}}| \d s_{\hat g_k}=\pi^2,
		\ee 
		then $r_k \to  0$ as $k \to  +\infty$ by \eqref{47}. Also it may assume $p_k \to  p$ as $k \to  +\infty$, thus $p$ cannot be any of these blow-up points $p_l$, $1 \leq l \leq L$ described in Lemma \ref{lem:3.5}. For brevity, one may regard $p$ as the north pole $\mathcal{N}$ on $\S^3$.
		
		Denote by $\tilde{\Phi}_k: \S^3 \to  \S^3$ the conformal diffeomorphisms mapping the upper hemisphere $\S_{+}^3 \equiv \S^3 \cap\{x^{4}>0\}$ into $B_{r_k}(p_k)$ and taking the equatorial sphere $\pa \S_{+}^3$ to $\pa B_{r_k}(p_k)$.  Consider the sequence of functions $\tilde{w}_k:\S^3\to \R$
		\[\tilde{w}_k=\hat w_k\circ \tilde \Phi_k+\frac{1}{3}\log (\operatorname{det}(\d \tilde \Phi_k)),\]
		which satisfy the equation
		\[
		\cP_{g_{\S^3}}^3 \tilde{w}_k+Q^3_{g_{\S^3}}=\hat{Q}_{k} e^{3 \tilde{w}_k} \quad \text { on }\, \S^3,
		\]
		where $\hat{Q}_k=Q^3_{\hat g_k} \circ \tilde{\Phi}_k$. From the selection of $r_k, p_k$ and \eqref{49-1}, by applying Lemma \ref{lem:3.5} and \eqref{samecurvature} to $\tilde{w}_k$, we conclude that $\tilde{w}_k \to  \tilde{w}_{\infty}$ in $H_{loc}^3(\S^3 \backslash\{\mathcal{S}\}, g_{\S^3})$ as $k \to  +\infty$, where $\mathcal{S}$ is the south pole on $\S^3$. Meanwhile, $\hat{Q}_k \to  Q^3_{\infty}(p)$ almost everywhere as $k \to  +\infty$.  Using the map  \be\label{PsiR^3}
		\hat \Psi(y):=\frac{1}{1+|y|^2}(2y_1,2y_2,2y_3,1-|y|^2),\quad y=(y_1,y_2,y_3)\in\R ^3,
		\ee
		and defining 
		\[
		\hat w_k:=\tilde w_k \circ \hat \Psi+\frac{1}{3} \log (\operatorname{det} \d \hat \Psi),
		\]
		we obtain $\hat w_k$ converging in $H_{loc}^3(\R^3)$ to a function $\hat{w}_{\infty}$, which satisfies the equation
		\be\label{51-1}
		(-\Delta_{\R^3})^{\frac{3}{2}} \hat{w}_{\infty}=Q^3_{\infty}(p) e^{3 \hat{w}_{\infty}} \quad \text { in }\, \R^3 .
		\ee
		Moreover, by Fatou's lemma we get
		\be\label{52-2} 
		\int_{\R^3} e^{3 \hat{w}_{\infty}}\, \d y \leq  \liminf _{k \to  +\infty} \int_{\R^3} e^{3 \tilde{w}_k} \,\d y=2\pi^2 .
		\ee

\textbf{Claim:} $\hat{w}_{\infty} $ is of the form   
		\be\label{classification-eq}
		\hat  w_{\infty}(y)=\log \Big(\frac{2\lam}{1+\lam^2|y-y_0|^2}\Big)-\frac{1}{3}\log \frac{Q^3_{\infty}(p)}{2}
		\ee
		for some $\lam>0$ and $y_0\in\R^3$.

		Indeed, by \cite[Theorem 4]{JMMX2015}, we only need to prove that 
		\[
		\lim_{|y|\to +\infty}\Delta_{\R^3} \hat  w_{\infty}(y)=0
		\]	 
		For simplicity, one uses $ w_{\infty}$ instead of  $\hat  w_{\infty}$.

		Firstly, we deduce from  \cite[Lemma 18]{JMMX2015} that 
		\[
		- \Delta_{\R^3} w_{\infty}(y)\geq 0 \quad \text{ in }\,\R^3.
		\]
		Set
		\[
		v(y):=-\frac{1}{2\pi^2}\int_{\R^3}\log \Big(\frac{|y-x|}{|x|}\Big)Q^3_{\infty}(p)e^{3w_{\infty}(x)}\, \d x.
		\] 
		Obviously, $v(y)$ satisfies that
		\[(-\Delta_{\R^3})^{\frac{3}{2}}v(y)=Q^3_{\infty}(p)e^{3w_{\infty}(y)}\quad \text{ in }\, \R^3.\]
		It follows from \cite[Theorem 3]{JMMX2015} (see also \cite[p.1763, Remark]{Z2004}) that 
		\[\lim_{|y|\to +\infty}\Delta_{\R^3}v(y) =0. \]
		It is easy to see that $(-\Delta_{\R^3})^{\frac{3}{2}}(v-w_{\infty})=0$ in $\R^3$, and by \cite[Lemma  15]{JMMX2015}  we have $-\Delta_{\R^3} (v-w_{\infty})\equiv b$  in $\R^3$ for some  constant $b$.  It remains to prove $b=0$.

		We claim that, for any
		$r >0$ small and $q \in \S^3$,  there exists a constant $C_0>0$,
		independent of $k$, such that
		\be\label{observation}
		\Big|\int_{B_r(q)} \Delta_{\S^3} \hat w_k \, \d s_{g_{\S^3}}\Big| \leq C_0 r .
		\ee

Indeed, using the equation
		\[
		(-\Delta_{g_{\S^3}}+1)^{1/2}(-\Delta_{g_{\S^3}} \hat  w_k)=Q_{\hat g_{k}}^3 e^{3 \hat w_k}-Q^3_{g_{\S^3}}:=F_k\quad \text { on }\, \S^3 ,
		\]
and letting  $G(\cdot, z)$ be the Green function of the operator $(-\Delta_{g_{\S^3}}+1)^{1/2}$. It is known that  $G(x,z)>0$ for $x\neq z$, and moreover, 
		\[
		G(x, z) \sim  c\operatorname{dist}_{g_{\S^3}}(x,z)^{-2} \quad \text { as }\,\operatorname{dist}_{g_{\S^3}}(x,z) \to  0,
		\]
		where $c>0$ is a positive constant.
		More precisely,
		\[
		G(x, z)=c\operatorname{dist}_{g_{\S^3}}(x,z)^{-2}+h(x, z),
		\]
		where $h(y, z)$ is a smooth function defined on $\S^3$. 
		
		By Green's formula, we obtain
		\[
		-\Delta_{\S^3} \hat w_k(z)=\int_{\S^3} G(x, z)F_k(x) \, \d s_{g_{\S^3}}(x).
		\]
		Then, for any $q \in \S^3$ and $r>0$, by Fubini's theorem we have
	\[\Big|\int_{B_r(q)}(-\Delta_{\S^3} \hat w_k(z)) \, \d s_{g_{\S^3}}(z)\Big|
	\leq  \int_{\S^3}|F_k(x)|\Big(\int_{B_r(q)} G(x, z) \,\d s_{g_{\S^3}}(z)\Big)\, \d s_{g_{\S^3}}(x).\]	
	Since 	the singularity $\operatorname{dist}(x,z)^{-2}$ is integrable on the
		three-dimensional sphere, we have the uniform estimate
		\be\label{G-integral-est}
			\sup_{x\in\S^3}\int_{B_r(q)}G(x,z)\,ds_{g_{\S^3}}(z)\le Cr.
		\ee	
		Indeed, if $x\in B_{2r}(q)$, then $B_r(q)\subset B_{3r}(x)$ and
		\[
		\int_{B_r(q)} \operatorname{dist}(x,z)^{-2}\,\d s_{g_{\S^3}}(z)
		\le
		\int_{B_{3r}(x)} \operatorname{dist}(x,z)^{-2}\,\d s_{g_{\S^3}}(z)
		\le C\int_0^{3r}\rho^{-2}\rho^2\,\d \rho
		\le Cr.
		\]
		If instead $x\notin B_{2r}(q)$, then $\operatorname{dist}(x,z)\ge r$ for all $z\in B_r(q)$, hence
		\[
		\int_{B_r(q)} \operatorname{dist}(x,z)^{-2}\,\d s_{g_{\S^3}}(z)
		\le r^{-2}\vol(B_r(q))
		\le Cr.
		\]
		Thus \eqref{G-integral-est} holds.
		
		Combining \eqref{G-integral-est} with the uniform $L^1$-bound	$\|F_k\|_{L^1(\S^3)}\le C$,
		which follows from $\vol(\S^3,\hat g_k)=2\pi^2$ and the convergence
		$\|Q^3_{\hat g_k}-Q^3_\infty\|_{L^2(\S^3,\hat g_k)}\to0$ together with the smoothness of
		$Q^3_\infty$, we prove  the validity of \eqref{observation}.

		Let $\phi:=v-w_\infty$. Then $-\Delta_{\R^3}\phi\equiv b$ in $\R^3$.
		Assume by contradiction that $b\neq 0$. Since $\lim_{|y|\to\infty}\Delta_{\R^3}v(y)=0$, we have
		\[
		\Delta_{\R^3} w_\infty(y)=\Delta_{\R^3}v(y)+b \to b \qquad \text{as }|y|\to\infty.
		\]
		In particular, if $b>0$, there exists $R_0>0$ such that $\Delta_{\R^3}w_\infty \ge b/2$ in $\R^3\setminus B(0,R_0)$, hence
	\be\label{lowerbound}
		\int_{B(0,d)}\Delta_{\R^3}w_\infty(y)\,\d y \ge c\,b\,d^3
		\quad \text{ for all } \,d\gg 1.
\ee
		On the other hand, by the blow-up procedure (via stereographic coordinates) and the local
		expansion of $g_{\S^3}$ around $p_k$, for any fixed $d>0$ we have
		\[
		\Big|\int_{B(0,d)}\Delta_{\R^3}w_\infty(y)\,\d y\Big|
	\le 	\liminf_{k\to\infty}\frac1{r_k}\Big|\int_{B_{r_k d}(p_k)}\Delta_{\S^3}\hat w_k \, \d s_{g_{\S^3}}\Big|
		\le C_0 d,
		\]
	where we used \eqref{observation} in the last inequality.	This contradicts \eqref{lowerbound} by letting $d\to\infty$. The case $b<0$ can be ruled out analogously.
		Therefore $b=0$, and we complete the proof of Claim.

		Since $\|Q^3_{\hat g_k}-Q^3_\infty\|_{L^2(\S ^3,\hat{g}_k)}\to 0$ as $k\to+\infty$, we deduce from \eqref{classification-eq} that
		\begin{align*}
			4\pi^2  =&\,\int_{\R^3} Q^3_{\infty}(p) e^{3 \hat w_{\infty}} \,\d y \leq  \int_{B_r(p)} Q^3_{\infty}\, \d s_{\hat{g}_k}+o(1) \\
			\leq &\,\int_{\S^3} Q_{\infty}^3 \, \d s_{\hat{g}_k}+o(1) \leq \int_{\S^3}(Q_{\hat g_{k}}^3+|Q_{\hat g_{k}}^3-Q_{\infty}^3|) \, \d s_{\hat{g}_k}+o(1) \\
			\leq&\,\int_{\S^3} Q_{\hat g_{k}}^3  \, \d s_{\hat{g}_k}+o(1)=4\pi^2+o(1).
		\end{align*}
	Moreover, 	since $\dashint_{\S ^3}\,  \d s_{\hat g_k}=1$, it implies that $p$ is the only concentration point of $\hat w_k$ and that
		\[
		Q^3_{\infty}(p)=2\quad  \text{ and }\quad \d s_{\hat g_k} \rightharpoonup 2\pi^2 \delta_p \quad \text { as }\, k \to  +\infty
		\]
		in the weak sense of measures, which proves  \eqref{48-1}.  By \eqref{48-1} and Proposition 3.4 of \cite{CD1987}, it follows that the family of conformal transformations $\hat \Phi_k$ associated with normalized metrics converges to a constant map $p$. In essence, using conformal invariance of $\cP_{g_{\S^3}}^3$, one has
		\begin{align*}
			\int_{\S^3} \hat \Phi_k \cP_{g_{\S^3}}^3(\hat  \Phi_k) \, \d s_{g_{\S^3}}  =&\,\int_{\S^3} x \cP_{(\hat \Phi_k^{-1})^*(g_{\S^3})}^3(x)  \, \d s_{(\hat \Phi_k^{-1})^*(g_{\S^3})} \\
			=&\,\int_{\S^3} x \cP_{g_{\S^3}}^3(x) \, \d s_{g_{\S^3}} \\
			=&\,\int_{\R^3} \hat \Psi \cP_{\hat\Psi^*(g_{\S^3})}^3(\hat \Psi)  \, \d_{\hat\Psi^* (g_{\S^3})} \\
			=&\,\int_{\R^3}|(-\Delta_{\R^3})^{\frac{3}{4}} \hat \Psi |^2 \, \d y<\infty.
		\end{align*}
		Hence, up to a subsequence, $\hat \Phi_k$ converges weakly in $H^{\frac{3}{2}}(\S^3, g_{\S^3})$ to the constant map $p$ as $k \to  +\infty$.
		
		From \eqref{48-1}, as $k \to  +\infty$ we have
		\[
		\|Q^3_{\infty} \circ \hat  \Phi_k-Q^3_{\infty}(p)\|_{L^2(\S^3, \hat h_k)}=\|Q^3_{\infty}-Q^3_{\infty}(p)\|_{L^2(\S^3, \hat g_k)} \to  0
		\]
		and then
		\begin{align*}
			\|Q^3_{\hat h_k} -Q^3_{\infty}(p)\|_{L^2(\S^3,\hat  h_k)} =&\,\|Q^3_{\hat h_k} -Q^3_{\infty} \circ \hat  \Phi_k\|_{L^2(\S^3, \hat h_k)}+o(1) \\
			=&\,\|Q^3_{\hat g_k}-Q^3_{\infty}\|_{L^2(\S^3, \hat g_k)}+o(1) \to  0,
		\end{align*}
		where $o(1) \to  0$ as $k \to  +\infty$. Consequently, one may apply Lemma \ref{lem:3.5} to $v_k$ and metrics $h_k$. Furthermore, condition \eqref{23} excludes concentration in the sense of  \eqref{47}. Thus a subsequence $\hat v_k \to  \hat v_{\infty}$, $\hat h_k \to \hat h_{\infty}$ in $H^{\frac{3}{2}}(\S^3, g_{\S^3})$ as $k \to  +\infty$. Since $\hat h_{\infty}$ has constant $Q$-curvature $Q^3_{\infty}(p)= 2$ and $\dashint_{\S^3} \,\d s_{\hat h_{\infty}}=1$, $\hat v_{\infty}$ has to be the constant 0, thus the proof is complete.
	\end{proof}

	Lemma \ref{lem:3.5} admits a considerably sharper form if one assumes that the $T$-curvatures associated with $\{w_k\}$ converge in $L^2$ to some smooth limiting function $T_\infty>0$.   To handle the pseudo-differential nature of $\cP^{3}_{\hat g}$, we adopt an alternative method based on the Dirichlet-to-Neumann map, realized via a biharmonic extension to the upper half-space $\R^4_+$.	
	\begin{lem}\label{lem:3.6-2}
		Let $\{w_k\}$ be a sequence of smooth functions on $\B ^4$, and let 
		$g_k=e^{2w_k}g_{*}$ be the corresponding sequence of metrics.
		Assume that  $Q_{g_k}=0$ in $\B^4$, $H_{g_k}=0$ on $\S^3$, and $\vol (\S^3,\hat{g}_k)=2\pi^2$. 	Assume also that
		$\|T_{g_k}-T_\infty\|_{L^2(\S ^3,\hat{g}_k)}\to 0$ as $k\to+\infty$
		for some smooth function $T_\infty$ on $\S ^3$.
		Let $h_k=\Phi_k^*(g_k)=e^{2v_k}g_{\B ^4}$ be the associated sequence of normalized metrics as in  Section \ref{sec:3.1}.
		Then, after passing to a subsequence, one of the following two possibilities must occur:\\
		(i) Compactness: $w_k\to w_\infty$ in $H^4(\B ^4,g_{\B^4})$, where $g_\infty:=e^{2w_\infty}g_{*}$ satisfies
		\[
		\left\{\begin{aligned}
			&Q_{g_{\infty}}=0 &&\mbox{ in } \, \B^4,\\
			&H_{g_{\infty}}=0&& \mbox{ on } \, \S^3,\\ &T_{g_{\infty}}=	T_\infty&&\mbox{ on } \, \S^3.
		\end{aligned}\right.
		\] \\ 
		(ii) Concentration:  There exists a point $p\in \S ^3$ such that 
		\[
		\d s _{\hat g_k}\rightharpoonup 2\pi^2\delta_p\quad \mbox{ as }\,k\to+\infty
		\]
		in the weak sense of measures. In addition,
		\[
		\left\{\begin{aligned}
			&v_k\to 0&&\mbox{ in }\, H^4(\B ^4,g_{\B ^4}),\\
			&T_{h_k}\to 2&&\mbox{ in }\, L^2(\S ^3,g_{\S ^3}).
		\end{aligned}\right.
		\]
		In this case, the boundary maps $\hat \Phi_k$ converge weakly in $H^{\frac{3}{2}}(\S ^3,g_{\S ^3})$ 
		to the constant map $p$. 
	\end{lem}

	Let us introduce  the M\"obius transformation that maps the upper half–space
	\[\R_{+}^{4}:=\{Y=(y', y_4)\in\R^3\times\R: y_4\geq 0\}\]
	onto the unit ball
	\[\B^{4}=\{X=(x',x_4)\in\R^3\times\R:|X|\leq 1\}.\]
	For convenience, we denote this map by  $\Psi$, and write it explicitly as
	\be\label{Mobius-map}
	\Psi (Y):=\frac{2(Y+e_4)}{|Y+e_4|^2}-e_4, 
	\ee
	where $e_4:=(0, 0, 0,1) \in \R^{4}$. Its inverse  is given by 
	\[
	\Psi^{-1} (X)=\frac{2(X+e_4)}{|X+e_4|^2}-e_4 .
	\]
	It is well-known that
	\begin{gather*}
		\Psi^*(|\d X|^2)=\Big(\frac{2}{|Y+e_4|^2}\Big)^2|\d Y|^2 , \\
		|y'|^2+y_4^2=1-\frac{4 x_{4}}{|x'|^2+(x_{4}+1)^2}, \quad|X|^2=1-\frac{4 y_4}{|y'|^2+(y_4+1)^2},
	\end{gather*}
	see, e.g.,  \cite{NS2024}. The restriction of $\Psi $ to the boundary  $\pa \R_{+}^{4}=\R ^3$  is given explicitly by \eqref{PsiR^3}, 
	which is exactly the inverse stereographic projection from the south pole $\cS$. Similarly, the restriction of $\Psi^{-1}$  to   $\pa \B^4=\S^3 $ coincides with the stereographic projection from $\cS$. In particular,
	\be \label{75}
	\Psi^{-1}(x)=\frac{(x_1,x_2,x_3)}{1+x_4},\quad x=(x_1,x_2,x_3,x_4)\in\S ^3.
	\ee 
	Using M\"obius transformation \eqref{Mobius-map}, we can find a function $\tilde{w}$ such that $(\B^4 \backslash\{\cS\}, g)$ is isometric to $(\R_{+}^4, e^{2 \tilde{w}}|\d Y|^2)$ through the conformal relation $\Psi^*(e^{w}g_{*})=e^{2 \tilde{w}}|\d Y|^2$,
	where $w$ and $\tilde{w}$ are related by
	\begin{align*}
		\tilde{w}=&\,w \circ \Psi+\frac{1}{2}-\frac{1}{2} \frac{|y'|^2+(y_4-1)^2}{|y'|^2+(y_4+1)^2}+\log \Big(\frac{2}{(y')^2+(1+y_4)^2}\Big), \\
		w=&\,\tilde{w} \circ \Psi^{-1} -\frac{1-|X|^2}{2}+\log \Big(\frac{2}{|x'|^2+(1+x_4)^2}\Big).
	\end{align*}
	By the isometry, we may regard  \eqref{maineq} as  an equation on $\R_{+}^4$ for the pullback metric $\Psi ^* g$. Taking $|\d Y|^2$ as the background metric and applying the conformal transformation laws for $P_g^4$, $P_g^{3,b}$ and $H_g$ in \eqref{conformal-3}, we rewrite system  \eqref{maineq} as 
	\[
	\left\{\begin{aligned}&\Delta^2 \tilde{w}=0 && \text { in }\, \R_{+}^4, \\& \pa_{y_4} \Delta \tilde{w}=2f e^{3 \tilde{w}} && \text { on }\, \R^3, \\ &\pa_{y_4} \tilde{w}=0 && \text { on }\, \R^3.	\end{aligned}\right.
	\]
	Moreover, the isometry also implies the finite-volume conditions
	\[
	\int_{\R_{+}^4} e^{4 \tilde{w}(y', y_4)} \d y' \d y_4=\operatorname{vol}(\B^4, g)<+\infty, \quad \int_{\R^3} e^{3 \tilde{w}(y', 0)} \,\d y'=\vol (\S^3,\hat g)<+\infty.
	\]

	\begin{proof}[Proof of Lemma \ref{lem:3.6-2}]
	The proof of the compactness alternative (i) is identical to the one in Lemma \ref{lem:3.6-1}, so we omit it. Assume therefore that concentration occurs in the sense of Lemma \ref{lem:3.5}, we proceed to show that this necessarily yields the behavior described in (ii).

		By Lemma \ref{lem:3.5},  we can choose $p_k \in \S^3$ and $r_k>0$ such that
		\be\label{49}
		\sup _{p \in \S^3} \int_{B_{r_k}(p)}|T_{g_k}| \, \d s_{\hat g_k} =\int_{B_{r_k}(p_k)}|T_{g_k}| \, \d s_{\hat g_k} = \pi^2.
		\ee
		In view of \eqref{47},  we have $r_k \to 0$ as $k \to +\infty$.  Up to a subsequence, we may also assume  $p_k \to p\in \S^3$.

		Let $\tilde{\Phi}_k: \B^4 \to  \B^4$ be conformal diffeomorphisms whose boundary restrictions
		$\tilde\Phi_k|_{\S^3}$ mapping the upper hemisphere $\S_{+}^3$ into $B_{r_k}(p_k)$ and taking the equatorial sphere $\pa \S_{+}^3$ to $\pa B_{r_k}(p_k)$. Consider the sequence of functions $\tilde{w}_k: \B^4 \to  \R$ defined by
	\[
		\tilde{w}_k=w_k \circ \tilde{\Phi}_k+\frac{1}{4} \log (\operatorname{det}(\d \tilde{\Phi}_k)).
	\]
Then $\tilde w_k$ solves
		\[
		\left\{\begin{aligned}
			&P^4_{g_{\B^4}}\tilde{w}_k=0&&\mbox{ in }\,\B ^4,\\
			&P^{3,b}_{g_{\B^4}}\tilde{w}_k+T_{g_{\B^4}}=\widetilde{T}_{g_k}e^{3\tilde{w}_k}&&\mbox{ on }\,\S ^3,\\&\frac{\pa \tilde{w}_k}{\pa\nu_{g_{\B^4}}}=0&&\mbox{ on }\,
			\S ^3,
		\end{aligned}\right.
		\]
		where $\widetilde{T}_{g_k}:=T_{g_k} \circ \tilde{\Phi}_k$. 
		By  Lemma \ref{lem:3.5} and the choice of $p_k, r_k$, together with \eqref{49}, we obtain  $\tilde{w}_k \to w_{\infty}$ in $H_{loc}^4(\B^4 \backslash\{\cS\})$ and  $\widetilde{T}_{g_k} \to  T_{\infty}(p)$ almost everywhere as $k \to  +\infty$. 
		
Let $\Psi:\R^4_+\to\B^4$ be the M\"obius transform \eqref{Mobius-map} and define  $W_k$ on $\R^4_+$ by
	\be\label{2.31}
		W_k=\tilde{w}_k \circ \Psi+\frac{1}{2}-\frac{1}{2} \frac{|y'|^2+(y_4-1)^2}{|y'|^2+(y_4+1)^2}+\log \Big(\frac{2}{(y')^2+(1+y_4)^2}\Big).
		\ee 
 Then $W_k\to W_\infty$ in $H^4_{loc}(\R^4_+)$, and $W_\infty$ satisfies
		\be\label{51}
		\left\{\begin{aligned}
			&\Delta^2W_{\infty}=0 && \text { in }\, \R_{+}^4, \\ &\pa_{y_4} \Delta W_{\infty}= 2 T_{\infty}(p)e^{3 W_{\infty}} && \text { on }\, \R^3, \\ &\pa_{y_4} W_{\infty}=0 && \text { on }\, \R^3.
		\end{aligned}\right.
		\ee
		By \eqref{20} and Fatou's lemma,
		\be\label{52-1}
		\int_{\R ^4_+} e^{4 W_{\infty}(y',y_4)} \, \d y'\, \d y_4 \leq \liminf _{k \to  +\infty} \int_{\R ^4_+} e^{4 W_k(y',y_4)} \, \d y'\, \d y_4=\vol(\B^4,g)
		\ee
		and 
		\be\label{52}
		\int_{\R ^3} e^{3 W_{\infty}(y',0)}  \, \d y' \leq \liminf _{k \to  +\infty} \int_{\R ^3} e^{3 W_k(y',0)}\, \d y'=\vol(\S^3,\hat g)=2 \pi^2.
		\ee

 Since $w_k$ is smooth on $\B^4$, then \eqref{2.31} leads to $W_{\infty}(y',0)=o(|y'|^2)$. Therefore the classification theorem \cite[Theorem 1.2]{NS2024} applies to \eqref{51}-\eqref{52}, yielding
		\[
		W_{\infty}(y',y_4)=\log \Big(\frac{2 \lam}{(\lam+y_4)^2+|y'-a|^2}\Big)+\frac{2y_4\lam}{(\lam+y_4)^2+|y'-a|^2}+cy_4^2+\frac{1}{3}\log \frac{2}{T_{\infty}(p)}
		\]
		for some $\lam>0$, $a \in \R ^3$ and $c\leq 0$.  From this one easily verifies that
		\[
		\int_{\R ^3} T_{\infty}(p) e^{3 W_{\infty}(y',0)} \,\d y'=4 \pi^2,
		\]
		which also implies that $T_{\infty}(p)=2$ in view of  \eqref{26}. The remaining part of the proof is the same as that of Lemma \ref{lem:3.6-1}, and we omit the details for simplicity.
	\end{proof}

	From Lemmas \ref{lem:3.4}, \ref{lem:3.6-1} and \ref{lem:3.6-2}, we obtain a precise description of the asymptotic behavior of the flow $w(t)$  in the divergent case. 
	For $t \geq 0$, let
	\[
	S:=S(t)=\dashint_{\S ^3} x \,\d s_{\hat g}=\dashint_{\S^3} \hat \Phi \,\d s_{\hat h}
	\] 
	denote the center of mass of the boundary metric $\hat g(t)$. When $h(t)$ is sufficiently close to $g_{\B^4}$, this point is well approximated by
	\[
	p:=p(t)=\dashint_{\S^3} \hat\Phi \,\d s_{g_{\S^3}}.
	\]
	For notational convenience, we extend the function $f$ to the region  $\{p\in\B^4: |p| \geq \frac{1}{2}\}$ by setting  $f(p):=f(\frac{p}{|p|})$.

	\begin{prop}\label{prop:3.7}
		Suppose  that problem   \eqref{maineq} admits no solution in the conformal class of $g_{\B ^4}$. Let $u(t)$ be a solution to \eqref{flow}--\eqref{20}, and let $v(t)$ be the corresponding normalized flow introduced in Section \ref{sec:3.1}. 
		Then, as $t \to  +\infty$,  one has
		\be\label{prop:3.7-1}
		\left\{\begin{aligned}
			&v(t)\to 0 &&\mbox{ in }\,H^4(\B ^4,g_{\B ^4}), \\ &h(t)\to g_{\B ^4} && \mbox{ in }\,H^4(\B ^4,g_{\B ^4}), \\& T_{h(t)}\to 2 && \mbox{ in }\,L^2(\S ^3,g_{\S ^3}),\\&\hat \Phi(t) \to p(t) && \mbox{ in }\,L^2(\S ^3,g_{\S ^3}).
		\end{aligned}\right.
		\ee 
		Moreover, as $t \to  +\infty$,  we have \[ \|f\circ \hat \Phi(t)- f(p(t))\|_{L^2(\S ^3,\hat h)}\to 0 \quad \text{ and } \quad  \al  f(p(t))\to 2. \] 
	\end{prop}
	
	\begin{proof}
	We first show that 	$h(t)\to g_{\B ^4}$  in $H^4(\B ^4,g_{\B ^4})$ and $\hat \Phi(t) \to p(t)$ in $L^2(\S ^3,g_{\S ^3})$. Assume by contradiction that this is false. Then there exists a sequence $t_k\to +\infty$  such that 
		\[
		\liminf _{k \to \infty}\, (\|h(t_k)-g_{\B^4}\|_{H^4(\B ^4,g_{\B ^4})}+\|\hat \Phi(t_k)-p(t_k)\|_{L^2(\S^3,g_{\S^3})})>0.
		\]
		By Lemmas \ref{lem:3.4}, \ref{lem:3.6-1} and \ref{lem:3.6-2}, after passing to a subsequence (still denoted by $\{t_k\}$), we may assume that $u(t_k) \to  u_{\infty}$, $g_k=e^{2 u(t_k)} g_{\B^4} \to  g_{\infty}:=e^{2 u_{\infty}} g_{\B^4}$ in $H^4(\B ^4,g_{\B ^4})$, and that $\al(t_k) \to  \al$ as $k\to+\infty$. 
		Lemmas \ref{lem:3.4}, \ref{lem:3.6-1} and \ref{lem:3.6-2} then imply that the metric $\alpha^{\frac23}g_\infty$ has vanishing $Q$-curvature and mean curvature, and its $T$-curvature is equal to $f$ on $\S^3$, which contradicts the assumption that problem \eqref{maineq} has no solution in $[g_{\B^4}]$.  This proves the claim.

		The remaining limits in \eqref{prop:3.7-1} follow immediately.  Moreover, since  $f$ is smooth on $\S^3$ and  $\hat \Phi(t) \to p(t)$ in $L^2(\S ^3,g_{\S ^3})$, we obtain $f\circ \hat \Phi(t) \to f(p(t))$ in $L^2(\S ^3,g_{\hat h})$. Finally, 		by \eqref{19}, \eqref{26} and H\"older inequality, 
		\[
		2-\al  f(p)=\al \dashint_{\S ^3}(f_{\hat \Phi}-f(p)) \, \d s _{\hat h} \to 0\quad  \mbox{ as }\, t\to+\infty ,
		\]
		which completes the proof.
	\end{proof}

	\section{Spectral decomposition}\label{sec:5}
	From now on, we assume that $f$ cannot be realized as the $T$-curvature of any
	conformal metric in the standard conformal class of $g_{\B^4}$. Hence, the $T$-curvature
	flow \eqref{flow3} with  initial data $w_0$ does not converge, and  Proposition \ref{prop:3.7} applies throughout without further mention.

	We consider the following eigenvalue problem on $(\B^4,g)$: find an eigenpair $(\Lam,\phi)\in\R\times \mathcal H_g$  with $\phi\not\equiv 0$ such that
	\begin{equation}\label{eigenvalue_problem}
		\left\{
		\begin{aligned}
			&P_g^4\phi=0 &&\mbox{ in }\B ^4,\\
			&P_g^{3,b}\phi=\Lam\phi && \mbox{ on }\,\S ^3,\\
			&\frac{\pa\phi}{\pa\nu_g}=0 && \mbox{ on }\,\S ^3.
		\end{aligned}\right.
	\end{equation}
	It is well known that problem \eqref{eigenvalue_problem} admits a complete sequence of eigenfunctions
	$\{\varphi_i^g\}_{i\in\N_0}$ whose boundary traces on $\S^3$, denoted by $\hat\varphi_i^g:=\varphi_i^g|_{\S^3}$, form an orthonormal basis of $L^2(\S^3,\hat g)$. 
	In particular, they are normalized so that $\int_{\S^3}\hat\varphi_i^g\,\hat\varphi_j^g\, \d s_{\hat g}=\delta_{ij}$,
	$i,j\in\N_0$,
	and hence $\|\hat\varphi_i^g\|_{L^2(\S^3,\hat g)}=1$ for all $i\in\N_0$.
	Moreover, the associated eigenvalues satisfy
	\be\label{nonegativeeigenvalues}
	\Lam_0^g=0<\Lam_1^g\le \Lam_2^g\le \ldots,
	\ee
	listed in nondecreasing order and repeated according to multiplicity. Similarly, we denote by $\{\varphi_i^h\}_{i\in\N_0}$ and $\{\varphi_i\}_{i\in\N_0}$ the eigenfunctions associated with the metrics $h$ and $g_{\B^4}$, respectively, with corresponding eigenvalues $\Lambda_i^h$ and $\Lam_i:=\Lam_i^{g_{\B^4}}$.

	The eigenfunctions can be chosen in such a way that
	$\varphi_i^h=\varphi_i^g\circ\Phi$. Moreover, we deduce from Proposition \ref{prop:3.7} that 
	and $\varphi_i^h\to\varphi_i$ smoothly as $t\to+\infty $ for all $i\in\N _0$. 
	In particular, 
	\be\label{65}
	\Lam_i^h=\Lam_i^g\to \Lam_i\quad \mbox{ for all }\,i\in\N _0.
	\ee

	We also consider the Steklov eigenvalue problem on $(\B^4,g_{\B^4})$: find  an eigenpair $(\lam,\psi)\in\R\times H^1(\B^4,g_{\B^4})$ with $\psi\not\equiv 0$ such that
	\be\label{Steklov}
	\left\{\begin{aligned}
		&\Delta_{g_{\B ^4}} \psi=0&&\mbox{ in }\B ^4,\\
		&\frac{\pa  \psi}{\pa \nu_{g_{\B ^4}}}=\lam \psi&& \mbox{ on }\,\S ^3.	
	\end{aligned}\right.
	\ee 
	By elliptic regularity, any Steklov eigenfunction is smooth up to the boundary.
	
	It is well-known that the Steklov eigenspaces on $(\B ^4,g_{\B ^4})$ are given by the boundary restrictions to $\S^3$  of homogeneous harmonic polynomials on $\R^4$.  More precisely, for each $k\in \N_0$, let  $E(k)$ denote the space of homogeneous harmonic polynomials of degree $k$ in $\R^4$. Then the associated Steklov eigenvalue is $\lam=k$, with multiplicity $(k+1)^2$; see, for instance, \cite[Example 1.3.2]{GP2017}.  Accordingly, the Steklov spectrum of \eqref{Steklov} can be arranged as
	\be\label{lambda}
	0=\lam _0<\lam _1=\ldots=\lam _4=1<\lam _5=2\leq\ldots,
	\ee  
	with a corresponding  basis of eigenfunctions  $\{\psi_i\}_{i\in\N _0}$  whose boundary traces on $\S^3$, denoted by $\{\hat \psi_i\}_{i\in\N _0}$, form an orthonormal basis of $L^2(\S^3,g_{\S^3})$.
	On the other hand, the eigenvalues of $-\Delta_{g_{\S^3}}$
	are given by  $k(k+2)$ for $k\in\N _0$, and we denote by $\hat E(k)$ the corresponding eigenspace. It is easy to see that $\hat E(k)=E(k)|_{\S ^3}$.
	Therefore, if $\lam_i=k$, then $\hat\psi_i:=\psi_i|_{\S^3}\in \hat E(k)$, and we obtain
	\be \label{egenfunction-sphere}
	-\Delta_{g_{\S ^3}}\hat \psi_i=\lam _i(\lam _i+2)\hat \psi_i\quad \mbox{ on }\,\S ^3.
	\ee

	\begin{prop}
		The eigenvalues of problem \eqref{eigenvalue_problem} on $(\B^4,g_{\B^4})$
		are given by $\Lam_k = k(k+1)(k+2)$, $ k\in\N_0$,
		each with multiplicity $(k+1)^2$. Moreover, the corresponding eigenspace associated with $\Lam_k$
		is spanned by the harmonic extensions of Steklov eigenfunctions
		with Steklov eigenvalue $k$.
		In particular, when listed in nondecreasing order and repeated according to multiplicity, the spectrum begins as
		\be\label{eigenvalue}
		0=\Lam_0<6=\Lam_1=\Lam_2=\Lam_3=\Lam_4<24=\Lam_5\le\ldots.
		\ee 
	\end{prop}
	
	\begin{proof}
		Let $(\phi,\Lam)$ be an eigenpair of problem \eqref{eigenvalue_problem}
		on $(\B^4,g_{\B^4})$.
		We first claim that $\Lam\in \{\lam_i(\lambda_i+1)(\lambda_i+2):i\in\N_0\}$,
		where $\{\lambda_i\}_{i\in\N_0}$ is the Steklov spectrum
		\eqref{lambda}.
		
		Indeed,	assume by contradiction that
		\be\label{1}
		\Lam\notin\{\lam_i(\lam_i+1)(\lam_i+2):i\in\N_0\}.
		\ee
		Since $\{\hat\psi_i\}_{i\in\N_0}$
		is an orthonormal basis of $L^2(\S^3,g_{\S^3})$, we may expand
		\be\label{expanding}
		\hat\phi=\sum_{i=0}^\infty a_i\,\hat\psi_i,
		\ee 
		where $a_i:=\int_{\S^3} \hat\phi \hat\psi_i \, \d s_{g_{\S^3}}$. 
		Because $P^4_{g_{\B^4}}\phi=0$ in $\B^4$,      Lemma 6.2 in \cite{AC2017}  together with \eqref{expanding}  yields
		\[
		\Lam \hat\phi =P^{3,b}_{g_{\B^4}}\phi
		=\cP^3_{g_{\S^3}}\hat\phi
		=\sum_{i=0}^\infty a_i\,\cP^3_{g_{\S^3}}\hat\psi_i=\sum_{i=0}^\infty a_i P^{3,b}_{g_{\B^4}}\psi_i,
		\]
		where we also used  $\Delta_{g_{\B ^4}} \psi_i=0$  in $\B ^4$ and Lemma 6.2 in \cite{AC2017} in the last equation.
		Since $(\lam_i,\psi_i)$ is a Steklov eigenpair for problem \eqref{Steklov},  we further deduce from \eqref{P^4&P^3} and \eqref{egenfunction-sphere} that 
		\begin{align*}
			\Lam \hat \phi = \sum_{i=0}^\infty a_i P^{3,b}_{g_{\B^4}}\psi_i=&\, \sum_{i=0}^\infty a_i \Big(-\frac{1}{2}\frac{\pa  \Delta_{g_{\B ^4}} \psi_i}{\pa \nu_{g_{\B ^4}}}
			-\Delta_{g_{\S ^3}}\frac{\pa \psi_i}{\pa \nu_{g_{\B ^4}}}
			-\Delta_{g_{\S ^3}}\hat\psi_i\Big)
			\\=&\,- \sum_{i=0}^\infty a_i (\lam_i+1) \Delta_{g_{\S ^3}}\hat \psi_i\\=&\, \sum_{i=0}^\infty a_i \lam_i (\lam_i+1) (\lam_i+2)\hat \psi_i.
		\end{align*}
		Combining this with \eqref{expanding}, we obtain
		\[
		0=\sum_{i=0}^\infty a_i(\lam_i(\lam_i+1)(\lam_i+2)-\Lam)\hat\psi_i.
		\]
		By  \eqref{1} and the fact that $\{\hat\psi_i\}$ is a basis of $L^2(\S^3,g_{\S^3})$, we get
		$a_i=0$ for all $i$. Hence $\hat\phi\equiv0$, contradicting the fact that
		$\phi$ is an eigenfunction.
		This proves the claim.
		
		Finally, on $(\B^4,g_{\B^4})$ the Steklov spectrum satisfies $\lambda_i=k$
		with multiplicity $(k+1)^2$.
		For each $k\in\N_0$, let $\{\psi_i:\lambda_i=k\}$ be a basis of Steklov eigenfunctions.
		Since each $\psi_i$ is harmonic in $\B^4$ and satisfies
		$\frac{\pa\psi_i}{\pa\nu_{g_{\B^4}}}=k\psi_i$ on $\S^3$,
		it follows from \eqref{P^4&P^3} and \eqref{egenfunction-sphere} that
		each $\psi_i$ gives rise to an eigenfunction of problem \eqref{eigenvalue_problem}
		with eigenvalue $\Lam_k=k(k+1)(k+2)$.
		Hence, the eigenspace corresponding to $\Lam_k$ is spanned by
		$\{\psi_i:\lambda_i=k\}$ and has dimension $(k+1)^2$.
		Arranging the eigenvalues in nondecreasing order and counting multiplicities
		then yields \eqref{eigenvalue}, which completes the proof.
	\end{proof}

	\subsection{Upper bound on the changing rate of $F_2$}
	\begin{lem}\label{lem:3.8}
		Assume that either  $f\equiv\text{constant}$ or  problem  \eqref{maineq} admits no solution in the conformal class of $g_{\B ^4}$.
		Then, as $t\to +\infty$, we have
		\[
		\frac{\d F_2(t)}{\d t}\leq -(2+o(1))G_2(t)
		+(12+o(1)) F_2(t) ,
		\]
		where $o(1)\to 0$ as $t\to+\infty $.
	\end{lem}
	\begin{proof}
		From \eqref{39}, \eqref{44},  and Lemma \ref{lem:3.4}, one has 
		\begin{align*}
			\frac{\d F_2(t)}{\d t}\leq &\,-2 G_2(t)+6\al \int_{\S ^3}f _{\hat \Phi}(T_h-\al  f_{\hat \Phi})^2\, \d s _{\hat h}+o(1)(G_2(t)+F_2(t))
			\\=&\,-(2+o(1))G_2(t)
			+(6\al  f(p)+o(1)) F_2(t)\\&\,+6\al \int_{\S ^3}(f _{\hat \Phi}-f(p))(T_h-\al  f_{\hat \Phi})^2\, \d s _{\hat h}.
		\end{align*}
		If $f\equiv\text{constant}$, then by \eqref{19} we have  $\al f(p)=2$,  and last term vanishes, yielding the desired conclusion. Otherwise, 
		applying H\"older inequality,  Sobolev's embedding $H^{\frac{3}{2}}(\S ^3,g_{\S ^3})
		\hookrightarrow L^4(\S ^3,g_{\S ^3})$
		and Proposition  \ref{prop:3.7}, we deduce
		\begin{align*}
			&\Big|\al \int_{\S ^3}(f_{\hat \Phi}-f(p))(\al  f_{\hat \Phi}-T_h)^2\, \d s _{\hat h} \Big|\\
			=&\, O(\|f_{\hat \Phi}-f(p)\|_{L^2(\S ^3,g_{\S ^3})}
			\|\al  f_{\hat \Phi}-T_h\|^2_{L^4(\S ^3,g_{\S ^3})})\\ =&\, o(1)\|\al  f_{\hat \Phi}-T_h\|^2_{H^{\frac{3}{2}}(\S ^3,g_{\S ^3})}\\
			= &\,o(1)(F_2(t)+G_2(t)).
		\end{align*}
		This, together with Proposition \ref{prop:3.7},  yields the desired conclusion and completes the proof.
	\end{proof}

	\subsection{Dominance of $B$ and the estimate of $F_2$}

	Expanding $\al  f-T_g$ and 
	$\al  f_{\hat \Phi}-T_h$   in the orthonormal eigenbases
	$\{\hat\varphi_i^g\}_{i\in\N_0}$ and $\{\hat\varphi_i^h\}_{i\in\N_0}$ of $L^2(\S^3,\hat g)$ and $L^2(\S^3,\hat h)$, respectively, we write
	\be \label{4.10}
	\al  f-T_g=\sum_{i=0}^\infty \beta^i\hat\varphi_i^g,\quad 
	\al  f_{\hat \Phi}-T_h=\sum_{i=0}^\infty \gamma^i \hat\varphi_i^h
	\ee 
	where
	\[\beta^i:=\int_{\S ^3}(\al  f-T_g)\hat\varphi_i^g \, \d s _{\hat g} 
	\quad \text{ and }\quad \gamma^i:=\int_{\S ^3}(\al  f_{\hat \Phi}-T_h)\hat\varphi_i^h \, \d s _{\hat h}.\]
	By Parseval's identity, it follows that 
	\be\label{F2(t)}
	F_{2}(t)=\sum_{i=0}^{\infty}|\beta^i|^2.
	\ee 
	Moreover, since the eigenfunctions can be chosen so that $\varphi_i^h=\varphi_i^g\circ\Phi$, we have $\beta^i=\gamma^i$ for all $i\in\N_0$. 
	In addition, condition \eqref{19} implies that
	\be \label{beta^0}
	\beta^0=0.
	\ee 
	Consequently, invoking  \eqref{63}, \eqref{4.10}, \eqref{beta^0}, and  Lemma 6.2 in \cite{AC2017}, we obtain 
	\be \label{G2(t)}
	G_2(t)=\sum_{i=1}^\infty\Lam_i^g|\beta^i|^2.
	\ee

	Let $x=(x_1,x_2,x_3,x_4)$ denote the coordinate functions in $\R ^4$ restricted to $\S^3$. We may choose
	\be\label{normal0}
	\hat\varphi_i= \frac{\sqrt{2}}{\pi}x_i,\quad i=1,2,3,4,
	\ee 
	so that $\int_{\S^3}\hat \varphi_i^2\, \d s_{g_{\S^3}}=1$.
	Define 
	\be \label{67}
	b^i:=\int_{\S ^3} x_i (\al  f_{\hat \Phi}-T_h) \, \d s _{\hat h} \quad \mbox{ for }\,i=1,2,3,4,
	\ee 
	and set $b:=(b^1,b^2,b^3,b^4)$.
	By Proposition \ref{prop:3.7}, we have 
	\be\label{betab}
	\beta^i=\frac{\sqrt{2}}{\pi} b^i+o(1)\|\al  f-T_g\|_{L^2(\S ^3,\hat g)}\quad \mbox{ for }\,i=1,2,3,4,
	\ee
	where  $o(1)\to 0$ as $t\to+\infty$. In addition, we introduce
	\be\label{def:B}
	B:=(B^1,B^2,B^3,B^4)= \frac{\sqrt{2}}{\pi}b.
	\ee

	For the  case $f= T_{g_{\B ^4}}=2$,
	we  estimate $b^i$ using a Kazdan-Warner type identity (see, e.g., \cite[p.205]{CY1995})
	\be \label{4.9}
	\int_{\S ^3}\la \nabla_{g_{\S ^3}}T_h,\nabla_{g_{\S ^3}}x_i  \ra _{g_{\S ^3}} e^{3\hat v}\, \d s _{g_{\S ^3}}=0
	\quad \mbox{ for }\,i=1,2,3,4.
	\ee 
	It follows from
	\eqref{20} and \eqref{19} that $\al \equiv 1$ when $f=2$.
	Combining this with \eqref{67} and the fact 
	\be\label{S3eigenvalue}
	-\Delta_{g_{\S ^3}}x_i=3x_i \quad \mbox{ on }\,\S ^3
	\quad \mbox{ for }\,i=1,2,3,4,
	\ee
	we obtain, by integration by parts,
	\begin{align*}
		b^i
		=-\frac{1}{3}\int_{\S ^3}  (T_{g_{\B ^4}}-T_h)e^{3\hat v}\Delta_{g_{\S ^3}} x_i \, \d s _{g_{\S ^3}}
		=\int_{\S ^3}  (T_{g_{\B ^4}}-T_h)\la \nabla_{g_{\S ^3}}\hat v,\nabla_{g_{\S ^3}}x_i  \ra _{g_{\S ^3}}\, \d s _{\hat h}.
	\end{align*}
	Consequently, using \eqref{62} and Proposition \ref{prop:3.7}, we deduce that
	\be \label{69}
	|b^i|= O( F_2(t)^{\frac{1}{2}}\|\hat v\|_{H^1(\S ^3,g_{\S ^3})})
	=o(1)F_2(t)^{\frac{1}{2}}\quad \mbox{ for }\,i=1,2,3,4,
	\ee
	where $o(1)\to 0$ as $t\to+\infty $. It then follows from \eqref{F2(t)} and \eqref{69} that 
	\be\label{beta5}
	\sum_{i=5}^{\infty}|\beta^i|^2=(1+o(1))F_2(t).
	\ee

	\begin{proof}[Proof of Theorem \ref{thm:4.1}]
		We divide the proof into  three steps. Throughout the argument, $o(1)$ denotes a quantity that tends to $0$ as $t\to+\infty$.

		\textbf{Step 1: (Decay estimate of $F_2$).} 
		In view of  \eqref{nonegativeeigenvalues}, \eqref{65}, \eqref{eigenvalue},  and \eqref{beta5},  we obtain
		\be \label{4.17}
		G_2(t)=
		\sum_{i=1}^4\Lam_i^g|\beta^i|^2+\sum_{i=5}^\infty\Lam_i^g|\beta^i|^2
		\geq (\Lam_5+o(1))\sum_{i=5}^\infty|\beta^i|^2
		=(24+o(1))F_2(t). 
		\ee
		Combining \eqref{4.17} and Lemma \ref{lem:3.8}, we obtain
		\[
		\frac{\d}{\d t}F_2(t)\leq -\delta F_2(t)
		\]
		for some uniform constant $\delta>0$.
		Integrating this differential inequality yields
		\be \label{71}
		F_2(t)\leq Ce^{-\delta t}
		\ee 
		for all $t\geq 0$,  where  $C>0$ is a uniform constant.
		
		\textbf{Step 2: (Decay estimate of $\hat v$).} 
		From \eqref{26} we observe that 
		\be\label{normal1}
		\int_{\S ^3}(e^{3\hat v}-1)\, \d s _{g_{\S ^3}}=
		\int_{\S ^3}\, \d s _{\hat h} -\int_{\S ^3}\, \d s _{g_{\S ^3}}=0.
		\ee
		Moreover,  the normalization condition \eqref{23} implies 
		\be\label{noemal2}
		\int_{\S ^3}x\, \d s _{\hat h} =\int_{\S ^3}x(e^{3\hat v}-1)\, \d s _{g_{\S ^3}}=0.
		\ee
		Hence we may expand $e^{3\hat v}-1=\sum_{i=0}^\infty \hat V^i\hat \varphi_i$
		with respect to the basis $\{\hat \varphi_i\}_{i\in\N_0}$, where $\hat V^i:=\int_{\S^3}(e^{3\hat v}-1) \hat \varphi_i \, \d s_{g_{\S^3}} $. In particular, it follows from \eqref{normal0}, \eqref{normal1} and \eqref{noemal2} that 
		$\hat V^0=\ldots=\hat V^4=0$. 	
		
		On the other hand, let 	$\hat v= \sum_{i=0}^\infty \hat v^i\hat \varphi_i$
		be the expansion in the basis $\{\hat \varphi_i\}_{i\in\N_0}$ with coefficients $\hat v^i:=\int_{\S^3} \hat v \hat \varphi_i \, \d s_{g_{\S^3}}$.
		By Proposition  \ref{prop:3.7},
		and  Sobolev's embedding
		$H^{\frac{3}{2}}(\S ^3,g_{\S ^3})
		\hookrightarrow L^p(\S ^3,g_{\S ^3})$
		for every $ p\in [1,+\infty)$
		together with  Lemma \ref{lem:3.2}, we obtain
		\[
		3\hat v^i=3\int_{\S ^3}\hat v\hat \varphi_i \, \d s _{g_{\S ^3}}
		=\int_{\S ^3}(e^{3\hat v}-1) \hat \varphi_i\, \d s _{g_{\S ^3}}+O(\|\hat  v\|^2_{H^{\frac{3}{2}}(\S ^3,g_{\S ^3})})
		=\hat V^i+o(1)\|\hat v\|_{H^{\frac{3}{2}}(\S ^3,g_{\S ^3})}.
		\]
		In particular, we have
		\be \label{72}
		\sum_{i=0}^4|\hat v^i|^2= o(1)\|\hat v\|_{H^{\frac{3}{2}}(\S ^3,g_{\S ^3})}^2.
		\ee

		Next, rewriting the second equation in \eqref{flow5} as
		$P^{3,b}_{g_{\B ^4}}v=(T_h-T_{g_{\B ^4}}) e^{3\hat v}+T_{g_{\B ^4}}(e^{3\hat v}-1)$,
		and using Young's inequality together with the uniform boundedness of $\hat v$,
		and by using Sobolev's embedding together with Lemma \ref{lem:3.2}, we infer that for  any $\var >0$,
		\begin{align*}
			\int_{\S ^3}|P_{g_{\B ^4}}^{3,b}v|^2\, \d s _{g_{\S ^3}}
			\leq&\, C(\var )F_2(t)
			+4(1+\var )\int_{\S ^3}(e^{3\hat v}-1)^2\, \d s _{g_{\S ^3}}\\
			\leq&\, C(\var )e^{-\delta t}
			+36(1+\var )\int_{\S ^3}|\hat v|^2\, \d s _{g_{\S ^3}}+o(1)\|\hat v\|^2_{H^{\frac{3}{2}}(\S ^3,g_{\S ^3})},
		\end{align*}
		where we used \eqref{71} in the last inequality.
		Using \eqref{samev}, this implies (in terms of the Fourier coefficients $\hat v^i$)  that
		\[\sum_{i=0}^\infty\Lam_i^2|\hat v^i|^2\leq C(\var )e^{-\delta t}
		+(\Lam_1^2(1+\var )+o(1))\sum_{i=0}^\infty|\hat v^i|^2.
		\]
		Choosing $\var >0$ such that $\Lam_1^2(1+\var )<\Lam _5^2$
		and using \eqref{72}, we obtain
		\[\|\hat v(t)\|_{H^3(\S ^3,g_{\S ^3})}^2\leq C\sum_{i=0}^\infty(1+\Lam_i^2)|\hat v^i|^2
		\leq Ce^{-\delta t}+o(1)\|\hat v(t)\|_{H^3(\S ^3,g_{\S ^3})}^2.
		\]
		Therefore, for $t$ sufficiently large,
		\be \label{73}
		\|\hat v(t)\|_{H^3(\S ^3,g_{\S ^3})}^2\leq Ce^{-\delta t},
		\ee 
		as claimed.
		
		\textbf{Step 3: (Convergence of the M\"{o}bius map $\hat \Phi(t)$).} By \eqref{71}  and Lemma \ref{35}, we have
		\[\Big\|\d\hat \Phi(t)^{-1}\frac{\d\hat \Phi(t)}{\d t}\Big\|^2_{T_{id}G}= O(\|\hat \xi(t)\|^2_{L^\infty(\S ^3)})
		= O(F_2(t))= O(e^{-\delta t}).
		\]
		Here 
		$G$ denotes the (finite-dimensional) group of conformal diffeomorphisms of $\S^3$, and $\|\cdot\|_{T_{id}G}$	is any fixed norm on its Lie algebra $T_{id}G$ (equivalently, on the space of conformal Killing vector fields on $\S^3$). It follows that $\hat\Phi(t)$ converges smoothly and exponentially fast to some limit map
		$\hat\Phi_\infty$ on $\S^3$ as $t\to+\infty$. 
		Combining this with \eqref{73}, we conclude that
		\[
		\hat g(t)=(\hat\Phi(t)^{-1})^*\hat h(t)\to
		\hat g_\infty:=(\hat\Phi_\infty^{-1})^*g_{\S^3},
		\]
		and that $\hat w(t)\to\hat w_\infty$ exponentially fast in $H^3(\S^3,g_{\S^3})$,
		where $\hat g_\infty=e^{2\hat w_\infty}g_{\S^3}$.
		Consequently, $w(t)\to w_\infty$
		exponentially fast in  $H^4(\B^4,g_{\B^4})$
		as $t\to+\infty$, where $w_\infty\in C^\infty(\B^4)$, and the limiting metric
		$g_\infty=e^{2w_\infty}g_{\B^4}$ satisfies \eqref{thm:4.1-1}.
		This completes the proof.
	\end{proof}

	If the flow $\{w(t)\}$
	does not converge, then the coefficients $B(t)$ defined in \eqref{def:B} dominate the higher Fourier modes, and the flow asymptotically shadows the trajectory of an ODE on a four-dimensional subspace.
	\begin{lem}\label{lem:5.1}
		For each $i=1,2,3,4$, one has 
		\[\frac{\d B^i}{\d t}=o(1)F_2(t)^{\frac{1}{2}},\]
		where $o(1)\to 0$ as $t\to+\infty $.
	\end{lem}
	\begin{proof}
		Since $B^i=\frac{\sqrt{2}}{\pi}b^i$ with $b^i$ defined in \eqref{67}, it suffices to prove the estimate for $b^i$.
		From \eqref{flow-g*}, \eqref{samev},  \eqref{eigenvalue} and  the self-adjointness of $\cP^{3}_{g_{\S ^3}}$, we compute
		\[
		\int_{\S ^3}x_i T_h\, \d s _{\hat h} 
		=
		\int_{\S ^3}x_i (\cP^{3}_{g_{\S^3}}\hat v +2)\, \d s _{g_{\S ^3}}=\int_{\S ^3}x_i \cP^{3}_{g_{\S^3}}\hat v\, \d s _{g_{\S ^3}}
		=6\int_{\S ^3}x_i\hat v\, \d s _{g_{\S ^3}}.
		\]
		Consequently,
		\begin{align*}
			\frac{\d b^i}{\d t}
			=&\,\frac{\d}{\d t}\Big(\int_{\S ^3}x_i(\al  f_{\hat \Phi}-T_h)\, \d s _{\hat h} \Big)\\
			=&\,\int_{\S ^3}x_i\Big(\frac{\d(\al  f_{\hat \Phi})}{\d t}+3\hat  v_t(\al  f_{\hat \Phi}-2e^{-3\hat v})\Big)\, \d s _{\hat h} \\
			=&\,\al _t\int_{\S ^3}x_i f_{\hat \Phi} \, \d s _{\hat h} 
			+\al \int_{\S ^3} x_i(\d f_{\hat \Phi}\cdot\hat\xi+3\hat v_t(f_{\hat \Phi}-f(p)))\, \d s _{\hat h} \\
			&\,+3\int_{\S ^3} x_i \hat v_t(\al  f(p)-2e^{-3\hat  v})\, \d s _{\hat h} \\:=&\, I+II+III,
		\end{align*}
		where $\hat\xi=(\d \hat\Phi)^{-1}\frac{\d \hat\Phi}{\d t}$. It remains to estimate $I$, $II$ and $III$.

		\textbf{Estimate of $I$.}
		By \eqref{42}-\eqref{alpha_t}, \eqref{26}, \eqref{62} and H\"older inequality,  we have 
		\begin{align*}
			|I|=&\,\Big|\al _t\int_{\S ^3}x_i f_{\hat \Phi} \, \d s _{\hat h} 
			\Big|=\Big|\al _t\int_{\S ^3}x_i (f_{\hat \Phi}-f(p)) \, \d s _{\hat h} \Big|\\= &\, O(|\al_t|\|f_{\hat \Phi}-f(p)\|_{L^2(\S ^3,\hat h)})\\= &\, O(\|f_{\hat \Phi}-f(p)\|_{L^2(\S ^3,\hat h)}
			F_2(t)^{\frac{1}{2}})=o(1)F_2(t)^{\frac{1}{2}},	
		\end{align*}
		where we used Proposition \ref{prop:3.7} in the last step.
		
		\textbf{Estimate of $II$.}  Using \eqref{33} and integration by parts, we write
		\begin{align*}
			II=&\,\al \int_{\S ^3} x_i(\d f_{\hat \Phi}\cdot\hat\xi+3 \hat v_t(f_{\hat \Phi}-f(p)))\, \d s _{\hat h} \\
			=&\,3\al \int_{\S ^3}x_i \hat w_t\circ\hat \Phi (f_{\hat \Phi}-f(p)) \, \d s _{\hat h} \\
			&\,+\al \int_{\S ^3}x_i \operatorname{div}_{g_{\S ^3}}(\hat \xi e^{3\hat v}(f_{\hat \Phi}-f(p))\, \d s _{g_{\S ^3}}\\
			=&\,3\al \int_{\S ^3}x_i \hat w_t\circ \hat \Phi (f_{\hat \Phi}-f(p)) \, \d s _{\hat h} 
			-\al \int_{\S ^3}\hat \xi_i (f_{\hat \Phi}-f(p)) \, \d s _{\hat h} . \end{align*}
		Hence, by \eqref{42},  \eqref{flow3}, \eqref{26}, H\"older inequality and Lemma \ref{35} that 
		\begin{align*}
			|II|=&\,
			O(\|f_{\hat \Phi}-f(p)\|_{L^2(\S ^3,\hat h)}
			(\|\hat w_t\circ\hat \Phi\|_{L^2(\S ^3,\hat h)}+\|\hat\xi\|_{L^\infty(\S ^3)}))\\
			= &\,O(\|f_{\hat \Phi}-f(p)\|_{L^2(\S ^3,\hat h)} F_2(t)^{\frac{1}{2}})=o(1)F_2(t)^{\frac{1}{2}},
		\end{align*}
		where we used Proposition \ref{prop:3.7} in the last step.
		
		\textbf{Estimate of $III$.}	Using \eqref{33} and integration by parts, we obtain
		\begin{align*}
			III =&\,3\int_{\S ^3} x_i \hat v_t(\al  f(p)-2e^{-3\hat v})\, \d s _{\hat h} \\
			=&\,3\int_{\S ^3} x_i \hat w_t\circ\hat \Phi(\al  f(p)-2e^{-3\hat v})\, \d s _{\hat h} \\
			&\,+\int_{\S ^3} x_i (\al  f(p)-2e^{-3\hat v}) \operatorname{div}_{g_{\S ^3}}(\hat\xi e^{3 \hat v})\, \d s _{g_{\S ^3}}\\
			=&\,3\int_{\S ^3} x_i \hat w_t\circ \hat \Phi(\al  f(p)-2e^{-3 \hat v})\, \d s _{\hat h} 
			-\int_{\S ^3} (\al  f(p)-2e^{-3\hat v})\hat \xi_i\, \d s _{\hat h}\\
			&\,-6\int_{\S ^3} x_i \la \nabla_{g_{\S^3}} \hat v, \hat\xi \ra_{g_{\S^3}}  \, \d s _{g_{\S ^3}}.
		\end{align*}
		By \eqref{flow3}, \eqref{26}, H\"older inequality, Lemma \ref{35} and  Proposition  \ref{prop:3.7}, we have
		\begin{align*}
			|III|= &\,O((\|1-e^{-3\hat v}\|_{L^2(\S ^3,g_{\S ^3})}
			+|\al  f(p)-2|)(\|\hat w_t\circ \hat \Phi\|_{L^2(\S ^3,\hat h)}+\|\hat\xi\|_{L^\infty(\S ^3)}))\\
			&\,+O(\|\hat v\|_{H^1(\S ^3,g_{\S ^3})}\|\hat \xi\|_{L^\infty(\S ^3)})
			\\=&\, o(1)F_2(t)^{\frac{1}{2}}.
		\end{align*}

		Combining the estimates for $I$, $II$, and $III$ yields $\frac{\d b^i}{\d t}=o(1)F_2(t)^{1/2}$.
		The desired conclusion for $B^i$ follows immediately.
	\end{proof}

	From Lemma \ref{lem:5.1} we obtain the following estimate for $F_2$. 
	\begin{lem}\label{lem:5.2}
		The following estimate holds for sufficiently large $t$:
		\[F_2(t)= (1+o(1))|B|^2,\]
		where  $o(1)\to 0$ as $t\to+\infty $.
	\end{lem}
	\begin{proof}
		Set   $ \hat F_2:=\sum_{i= 5}^{\infty}|\beta^i|^2$. It follows from \eqref{62} and \eqref{betab} that 
		\be\label{lem:5.2-1}
		F_2=|\beta|^2+\hat F_2=|B|^2+\hat F_2+o(1) F_2.
		\ee
		We complete the proof in  three steps.
		
		\textbf{Step 1. Excluding the case $2|B|^2\leq \hat F_2$ for  large $t$. }
		Assuming that $2|B|^2\leq \hat F_2$ for sufficiently large $t$. Then  \eqref{lem:5.2-1} implies that   \[\hat F_2\geq \frac{2}{3}(1+o(1))F_2\quad \text{ for sufficiently large $t$.}\] 
		Arguing as in the derivation of \eqref{4.17}, we obtain for large $t$ that
		\[
		G_2\geq \Lam_5^g \hat F_2=(\Lam_5+o(1))\hat F_2\geq
		\frac{1}{2}\Lam_5 F_2=12 F_2,
		\]
		where we used  $\Lam_5=24$.  Applying Lemma \ref{lem:3.8}, we conclude that
		\[\frac{\d F_2}{\d t}\leq -(12+o(1))F_2\leq -F_2\quad \text{ for sufficiently large $t$.}\]
		As in the proof of Theorem \ref{thm:4.1},
		this implies that the flow
		$g(t)$ converges exponentially fast to a metric of $T$-curvature
		a positive constant multiple of $f$, zero $Q$-curvature and mean curvature, contradicting our hypothesis in this section.
		Therefore, there exists an arbitrarily large time $t_1\ge 0$ such that
		$2|B(t_1)|^2> \hat F_2(t_1)$.
		
		\textbf{Step 2. Introducing $\delta(t)$ and estimating $\hat{F}_2$.}
		Write $F_2=(1+\delta(t))|B|^2$.
		Then \eqref{lem:5.2-1} gives that 
		\be\label{lem:5.2-3}
		\hat{F}_2=F_2-\frac{F_2}{1+\delta(t)}+o(1) F_2=\frac{\delta(t)}{1+\delta(t)} F_2+o(1) F_2.
		\ee
		Since $F_2-|B|^2\geq o(1)F_2$ by \eqref{lem:5.2-1},
		we have $\delta(t)\geq o(1)$ and hence $\delta(t)>-\frac{1}{2}$ for all $t\geq t_0$ and
		sufficiently large $t_0\geq 0$.
		In addition, for sufficiently large $t_1$ as above, we have
		$\delta(t_1)=\frac{F_2(t_1)}{|B(t_1)|^2}-1<4$ since $F_2(t_1)\leq 3|B(t_1)|^2+o(1) F_2(t_1)$. By continuity,  $\delta(t)<4$ for all $t$ sufficiently close to $t_1$.

		\textbf{Step 3. A differential inequality for $\delta(t)$.}
		From  \eqref{F2(t)}-\eqref{G2(t)}, and Lemma \ref{lem:3.8}
		we obtain
		\begin{align*}
			\frac{\d F_2}{\d t}=&\,\frac{\d\delta}{\d t}|B|^2+2(1+\delta)B\cdot\frac{\d B}{\d t}\\=&\,-\sum_{i= 0}^{\infty}(2\Lam_i-12+o(1))|\beta^i|^2\\
			\leq  &\,-36\hat F_2+o(1)F_2(t)\\=&\,
			-\Big(\frac{36\delta}{1+\delta}+o(1)\Big)F_2,
		\end{align*}
		in view of  $\Lam_i\geq \Lam_5=24$ for all $i\geq 5$  and \eqref{lem:5.2-3}.
		Since Lemma \ref{lem:5.2} implies $|B\cdot\frac{\d B}{\d t}|= o(1)F_2$
		for $t$ near $t_1$, it follows that
		\[\frac{\d\delta}{\d t}|B|^2\leq -\Big(\frac{36\delta}{1+\delta}+o(1)\Big)F_2
		=-(36\delta+o(1))|B|^2
		\]
		and then
		\be\label{lem:5.2-4}
		\frac{\d\delta}{\d t}\leq -(36\delta+o(1)).
		\ee
		In particular, for sufficiently large $t_1$,
		we obtain  $\delta(t)<4$ for all $t\geq t_1$. The differential inequality
		\eqref{lem:5.2-4} shows that $\delta(t)\to 0$
		as $t\to+\infty $, which completes the proof.
	\end{proof}

	\section{The shadow flow}\label{sec:6}

	Recall from Proposition \ref{prop:3.7} that 
	the center of mass $S(t)$ of  $\hat g(t)$
	is given approximately by the shadow flow 
	\[
	p=p(t)=\dashint_{\S ^3}\hat \Phi(t)\, \d s _{g_{\S ^3}}.
	\]
	In what follows, we  relate $B$ (equivalently $b$) to the gradient of $f$ evaluated at $\hat p= \frac{p}{|p|}$.
	\subsection{Scaled stereographic projection}\label{sec:5.2}

	Recall that $\Psi$ denotes the conformal map from the upper half-space $\R_{+}^{4}$ to the unit ball $\B^{4}$ introduced in  \eqref{Mobius-map}.  Its inverse $\Psi^{-1}$, when  restricted to the boundary $\pa \B^4=\S^3$, maps into $\pa \R_{+}^{4}=\R^3$ and is given by \eqref{75}. For any $q=(q_1,q_2,q_3)\in\R ^3$ and $r>0$, we define the conformal map $\Psi_{q,r}:\R ^4_+\to\B ^4$ by 
	\be\label{sec:5.2-0}
	\Psi_{q,r}:=\Psi\circ \delta_{q,r},
	\ee
	where $\delta_{q,r}$ is the affine linear transformation 
	\begin{align*}
		\delta_{q,r}:\R ^4_+&\longrightarrow \R ^4_+,\\y&\longmapsto (q,0)+ry.
	\end{align*}
	Restricting  $\Psi_{q, r}$ to the boundary $\pa \R_{+}^{4}=\R^3$, we write $\hat \Psi_{q, r}=\Psi_{q, r}|_{\R ^3}=\hat\Psi\circ\hat\delta_{q, r}$, where  $\hat \delta_{q, r}=\delta_{q, r}|_{\R ^3}$.  A direct computation yields that 
	\be\label{sec:5.2-1}
	\frac{\pa \hat\Psi_{q,r}}{\pa  p_i}\Big|_{q=0,r=1}=\frac{\pa \hat\Psi}{\pa  y_i}:=e_i,\quad i=1,2,3,
	\ee
	and
	\be\label{sec:5.2-2}
	\frac{\pa \hat \Psi_{q,r}(y)}{\pa  r}\Big|_{q=0,r=1}=\sum_{i=1}^3 y_i e_i(y),\quad y\in\R ^3,
	\ee
	where  the vector fields $\{e_1,e_2,e_3\}$ are given explicitly  by
	\begin{align}
		e_1(y)=&\,\frac{1}{(1+|y|^2)^2}(2(1-2|y_1|^2+|y|^2),-4y_1y_2,-4y_1y_3,-4y_1)\nonumber\\
		=&\,(1+x_4-|x_1|^2,-x_1x_2,-x_1x_3,-x_1(1+x_4)),\label{76-1}	\\
		e_2(y)=&\,\frac{1}{(1+|y|^2)^2}(-4y_2y_1,2(1-2|y_2|^2+|y|^2),-4y_2y_3,-4y_2)\nonumber\\
		=&\,(-x_2x_1,1+x_4-|x_2|^2,-x_2x_3,-x_2(1+x_4)),\label{76-2}\\
		e_3(y)=&\,\frac{1}{(1+|y|^2)^2}(-4y_3y_1,-4y_3y_2,2(1-2|y_3|^2+|y|^2),-4y_3)\nonumber\\
		=&\,(-x_3x_1,-x_3x_2,1+x_4-|x_3|^2,-x_3(1+x_4)),\label{76}	
	\end{align}
	where we used the stereographic coordinates $y\to x\in\S^3$ given by \eqref{75}.
	In particular, we obtain
	\be \label{77}
	\sum_{i=1}^3y_ie_i(y)=(x_1x_4,x_2x_4,x_3x_4, |x_4|^2-1).
	\ee 
	
	For $t_0\geq 0$ and $t\geq 0$ close to $t_0$, let
	\[
	\Phi_{t_0}(t):=\Phi(t_0)^{-1}\Phi(t) .
	\]
	and 
	\[
	\hat\Phi_{t_0}(t):=\Phi_{t_0}(t)|_{\S ^3}=\hat \Phi(t_0)^{-1}\hat \Phi(t).
	\]
	Define $q=q(t)$, $r=r(t)$ such that
	\be\label{sec:5.2-3}
	\hat \Phi_{t_0}(t)\circ \hat\Psi=\hat\Psi_{q(t),r(t)}.
	\ee 
	By \eqref{sec:5.2-1} and \eqref{sec:5.2-2}, the vector field
	$\hat\xi=  (\d\hat\Phi(t_0))^{-1}\frac{\d\hat\Phi}{\d t}|_{t=t_0}=\frac{\d\hat\Phi_{t_0}}{\d t}|_{t=t_0}$ has the following representation
	\[
	\hat \xi =\frac{\d}{\d t}(\hat\Phi_{t_0}\circ\hat\Psi)\Big|_{t=t_0}
	=\sum_{i=1}^3\frac{\pa \hat\Psi_{q,r}}{\pa  p_i}\frac{\d p_i}{\d t}
	+\frac{\pa \hat\Psi_{q,r}}{\pa  r}\frac{\d r}{\d t}
	=\sum_{i=1}^3\Big(\frac{\d p_i}{\d t}+y_i\frac{\d r}{\d t}\Big)e_i,
	\]
	where the derivative of $\hat\Psi_{q,r}$ are evaluated at $(q,r)=(0,1)$.
	We then define the map
	\be \label{X_defn}
	X:=\int_{\S ^3}\hat\xi \, \d s _{g_{\S ^3}}
	=(X_1,X_2,X_3,X_4)\in\R ^4,
	\ee 
	and using cancellations by oddness, from \eqref{76-1}-\eqref{77}, we obtain
	\be \label{78}
	X_i=\frac{\d p_i}{\d t}\int_{\S ^3} (1-|x_i|^2)\, \d s _{g_{\S ^3}}=\frac{3\pi^2}{2}\frac{\d p_i}{\d t} \quad \mbox{ for }\,i=1,2,3,
	\ee 
	and
	\be \label{79}
	X_4=-\frac{\d r}{\d t}\int_{\S ^3} (1-|x_4|^2)\, \d s _{g_{\S ^3}}=-\frac{3\pi^2}{2}\frac{\d r}{\d t}.
	\ee

	\subsection{Chracterization of the limit of the shadow flow}\label{sec:5.3}

	Fix $t_0\geq 0$. Consider a rotation that maps $\hat p(t_0)$
	into the north pole $\mathcal{N}$.
	Under this normalization, the map $\hat \Phi(t_0):\S ^3\to\S ^3$
	can be written in the form
	$\hat \Phi(t_0)=\hat\Psi_\var \circ \pi$
	for some $\var =\var (t_0)>0$,
	where $\pi:=\Psi^{-1}|_{\S^3}$ and  $\hat\Psi_\var (y):=\hat\Psi(\var  y)=\hat\Psi_{0,\var }(y)$, as defined in \eqref{75} and \eqref{sec:5.2-0}.
	Consequently, in stereographic coordinates one has
	\[
	\hat \Phi(t)\circ\hat\Psi=\hat \Phi(t_0)\circ\hat\Phi_{t_0}(t)\circ\hat\Psi=\hat\Psi_\var \circ\hat\delta_{q,r},
	\]
	in view of   \eqref{sec:5.2-0} and \eqref{sec:5.2-3}.
	
	For the statements of the  following lemmas, we
	recall that  $f$  is extended as $f(p)=f(\frac{p}{|p|})$
	for $p\in\B ^4$ with $|p|>1/2$.
	In addition, for a fixed time $t=t_0$,
	we write $\hat\Phi=\hat\Phi(t_0)$
	for brevity, where  $\hat \Phi=\hat\Psi_\var \circ \pi$.

	\begin{lem}\label{lem:5.3}
		The following estimate holds:
		\[
		\|f_ {\hat\Phi}-f(p)\|_{L^2(\S ^3,g_{\S ^3})}+\|\nabla_{g_{\S ^3}} f _{\hat \Phi}\|_{L^{\frac{6}{5}}(\S ^3,g_{\S ^3})}= O(\var) .
		\]
	\end{lem}

	\begin{proof}
		The proof is similar  to that in  \cite[Lemma 4.3]{H2012},
		and we omit the details here.
	\end{proof}

	\begin{lem}\label{lem:5.4}
		The following estimate holds:
		\[
		\|\hat v\|_{H^3(\S ^3,g_{\S ^3})}
		= O(F_2^{1/2}+\|f_ {\hat \Phi}-f(p)\|_{L^2(\S ^3,g_{\S ^3})}).
		\]
	\end{lem}
	\begin{proof}
		We expand $\hat v=\sum_{i=0}^\infty \hat v^i\hat \varphi_i$
		with respect to the orthonormal basis
		$\{\hat\varphi_i\}_{i\in\N_0}$ as in the proof of Theorem \ref{thm:4.1}.
		By \eqref{72}, 
		we have
		\be \label{80}
		\sum_{i=0}^5|\hat v^i|^2= o(1)\|\hat v\|_{H^{\frac{3}{2}}(\S ^3,g_{\S ^3})}^2,
		\ee 
		where $o(1)\to 0$ as $t\to+\infty $.
		Again by using Sobolev's embedding $H^{\frac{3}{2}}(\S^3,g_{\S^3}) \hookrightarrow L^p(\S^3,g_{\S^3})$ for every $p\in [1,+\infty)$ together with  Lemma \ref{lem:3.2} and the estimate $\|\hat v\|_{H^3(\S^3,g_{\S^3})}^2=o(1)$ in Lemma \ref{lem:3.6-1}, we get
		\begin{align}
			\|e^{3\hat  v}-1\|_{L^2(\S^3,g_{\S^3})}^2  =&\,\int_{\S^3}(e^{3 \hat v}-1)^2 \,\d s_{g_{\S^3}}=9 \int_{\S^3}|\hat v|^2 \,\d s_{\S^3}+O(\|\hat v\|_{H^{\frac{3}{2}}(\S^3,g_{\S^3})}^2 )\nonumber\\
			=&\,9 \int_{\S^3}|\hat v|^2 \,\d s_{\S^3}+o(1)\|\hat v\|_{H^{\frac{3}{2}}(\S^3,g_{\S^3})}\nonumber\\=&\,9 \sum_{i=0}^{\infty}|\hat v^i|^2+o(1)\|\hat v\|_{H^{\frac{3}{2}}(\S^3,g_{\S^3})}.\label{4.19}
		\end{align}
		Next, rewriting the second equation in \eqref{flow5} on $\S^3$, we obtain from \eqref{samev} that 
		\[
		\cP_{g_{\S ^4}}^3 \hat v
		=T_h e^{3\hat v}-2
		=((T_h-\al  f_{\hat \Phi})+\al (f_{\hat \Phi}-f(p))+(\al  f(p)-2))e^{3\hat v}+2(e^{3\hat v}-1).
		\]
		Using \eqref{4.19}, Proposition \ref{prop:3.7} and applying Young's inequality, we infer that for any
		$\delta>0$ there exists  $C(\delta)>0$ such that 
		\begin{align}
			&\,\sum_{i=0}^\infty\Lam_i^2|\hat v^i|^2=\,\|\cP^{3}_{g_{\S ^3}}\hat v\|^2_{L^2(\S ^3,g_{\S ^3})}\nonumber\\
			\leq &\,C(\delta)(\|\al  f_{\hat \Phi}-T_h\|_{L^2(\S ^3,g_{\S ^3})}^2
			+\|f_ {\hat\Phi}-f(p)\|_{L^2(\S ^3,g_{\S ^3})})\nonumber\\&\,+
			4(1+\delta)\|e^{3\hat v}-1\|_{L^2(\S ^3,g_{\S ^3})}^2\nonumber\\
			\leq&\, C(\delta)(\|\al  f_{\hat \Phi}-T_h\|_{L^2(\S ^3,g_{\S ^3})}^2
			+\|f_ {\hat\Phi}-f(p)\|^2_{L^2(\S ^3,g_{\S ^3})})\nonumber\\&\,+
			(\Lam_1^2(1+\delta)+o(1))\sum_{i=0}^\infty|\hat v^i|^2.\label{81}
		\end{align}
		Combining this with \eqref{80}, we may choose $\delta>0$ sufficiently small such that the last term on the
		right-hand side of \eqref{81} can be absorbed into the left-hand side for all sufficiently large $t$.
		This yields the desired estimate.
	\end{proof}

	We are now in a position to relate the components of $b$ defined in \eqref{67}
	to the first and second derivatives of $f$ evaluated at the shadow point $p=p(t)$.
	\begin{lem}\label{lem:5.5}
		As $t\to+\infty$, the following asymptotic expansions hold:
		\[
		b^i =\frac{4\pi^2}{3}\al  \var \Big(\frac{\pa  f} {\pa  x_i}(p)+O(\var )\Big)\quad \mbox{ for }\,i=1,2,3,
		\]
		and 
		\[
		b^4 =-\frac{8\pi^2}{3}\al  \var ^2(\Delta_{g_{\S^3}} f(p)+O(1)|\nabla_{g_{\S^3}} f(p)|^2+O(\var|\ln \var|)).
		\]
	\end{lem}
	\begin{proof}
		Using \eqref{S3eigenvalue} and integrating by parts, we obtain  for $i=1,2,3,4$,
		\begin{align*}
			b^i=&\,\int_{\S ^3}x_i(\al  f_{\hat \Phi}-T_h)\, \d s _{\hat h} =-\frac{1}{3}\int_{\S ^3}(\al  f_{\hat \Phi}-T_h)\Delta_{g_{\S ^3}}x_i \, \d s _{\hat h} \\
			=&\,\frac{1}{3}\int_{\S ^3}\la \nabla_{g_{\S ^3}}x_i, \nabla_{g_{\S ^3}} (\al f _{\hat \Phi}-T_h) \ra _{g_{\S ^3}}\, \d s _{\hat h} 
			+\int_{\S ^3}\la \nabla_{g_{\S ^3}}x_i, \nabla_{g_{\S ^3}}\hat v \ra _{g_{\S ^3}}(\al  f_{\hat \Phi}-T_h)\, \d s _{\hat h}\\=&\, \frac{1}{3}\int_{\S ^3}\la \nabla_{g_{\S ^3}}x_i, \nabla_{g_{\S ^3}} (\al f _{\hat \Phi}-T_h) \ra _{g_{\S ^3}}\, \d s _{\hat h} + O(\|\hat  v\|_{H^1(\S ^3, g_{\S ^3})}\|\al  f_{\hat \Phi}-T_h\|_{L^2(\S ^3,\hat{h})}), 
		\end{align*}
		where the last estimate follows from \eqref{42}, H\"older's inequality, and Sobolev embedding. Using Kazdan-Warner type identity \eqref{4.9} and \eqref{42}, we further have 
		\begin{align}
			b^i
			=&\,\frac{\al}{3} \int_{\S ^3}\la \nabla_{g_{\S ^3}}x_i, \nabla_{g_{\S ^3}}f _{\hat \Phi} \ra _{g_{\S ^3}}\, \d s _{g_{\S ^3}}
			+\frac{\al}{3}\int_{\S ^3}\la \nabla_{g_{\S ^3}}x_i, \nabla_{g_{\S ^3}}f _{\hat \Phi} \ra _{g_{\S ^3}}(e^{3\hat v}-1)\, \d s _{g_{\S ^3}}\nonumber\\&\,+O(\|\hat  v\|_{H^1(\S ^3, g_{\S ^3})}\|\al  f_{\hat \Phi}-T_h\|_{L^2(\S ^3,\hat{h})})\nonumber\\=&\, \frac{\al}{3} \int_{\S ^3}\la \nabla_{g_{\S ^3}}x_i, \nabla_{g_{\S ^3}}f _{\hat \Phi} \ra _{g_{\S ^3}}\, \d s _{g_{\S ^3}}+\al \int_{\S ^3}x_i(f_{\hat \Phi}-f(p))(e^{3\hat v}-1)\, \d s _{g_{\S ^3}}\nonumber\\
			&\,
			-\al \int_{\S ^3}\la \nabla_{g_{\S ^3}}x_i, \nabla_{g_{\S ^3}}\hat v \ra _{g_{\S ^3}}(f_{\hat \Phi}-f(p))\, \d s _{\hat h}
			\nonumber\\&\,+ O(\|\hat  v\|_{H^1(\S ^3, g_{\S ^3})}\|\al  f_{\hat \Phi}-T_h\|_{L^2(\S ^3,\hat{h})})\nonumber\\=&\, \frac{\al}{3}\int_{\S ^3}\la \nabla_{g_{\S ^3}}x_i, \nabla_{g_{\S ^3}}f _{\hat \Phi} \ra _{g_{\S ^3}}\, \d s _{g_{\S ^3}}+ O(\|\hat v\|_{H^3(\S ^3,g_{\S ^3})}\|f_ {\hat \Phi}-f(p)\|_{L^2(\S ^3,g_{\S ^3})})\nonumber
			\\&\,+O(\|\hat  v\|_{H^1(\S ^3, g_{\S ^3})}\|\al  f_{\hat \Phi}-T_h\|_{L^2(\S ^3,\hat{h})}),\label{lem:5.5-3}
		\end{align}
		where we also used   \eqref{42}, \eqref{4.19},  H\"older inequality and Sobolev embedding in the final estimate.

		We now represent $\hat\Phi(t_0)$ in stereographic coordinates by writing
		$\hat\Phi(t_0)\circ\hat\Psi=\hat\Psi_\var :\R ^3\to\S ^3$,  where
		$\hat\Psi_\var (y)=\hat\Psi(\var  y)$. Then 
		$\d \hat\Psi_\var (0)=2\var  \operatorname{id}$.  Moreover, the estimates
		$|\d\hat\Psi_\var (y)|=  O(\var) $ and 
		$|\nabla \d \hat\Psi_\var (y)|= O(\var^2)$ hold  uniformly for all $y\in \R^3$,  as a consequence of   \eqref{76-1}-\eqref{76}.
		Noting that $p=\hat\Psi_\var(0)$, the Taylor expansion of $f\circ\hat\Psi_\var$ at $y=0$ yields
		\begin{align}
			f(\hat\Psi_\var (y))-f(p) 
			=&\,\d f(p)\cdot \d\hat\Psi_\var (0) y+\frac{1}{2}\nabla \d f(p)(\d\hat\Psi_\var (0)y,\d\hat\Psi_\var (0)y)+O(\var ^3|y|^3)\nonumber\\
			=&\,2\var  \d f(p) y+2\var ^2 \nabla \d f(p)(y,y)+O(\var ^3|y|^3).\label{lem:5.5-6}
		\end{align}

		Since $f$ is smooth, we deduce from \eqref{lem:5.5-6} that 
		\begin{align}
			\|f_ {\hat \Phi}-f(p)\|_{L^2(\S ^3,g_{\S ^3})}^2
			=&\int_{B(0, \frac{1}{\var} )}|f(\hat\Psi_\var (y))-f(p)|^2\frac{8\, \d y  }{(1+|y|^2)^3}+O\Big(\int_{\R^3\setminus B(0, \frac{1}{\var} )}\frac{ \d y  }{(1+|y|^2)^3}\Big)\nonumber\\
			=&\, O\Big(\var ^2|\nabla_{g_{\S^3}} f(p)|^2\int_{B(0, \frac{1}{\var} )}\frac{|y|^2\, \d y  }{(1+|y|^2)^3}\Big)+O(\var ^3)\nonumber\\
			= &\,O(\var ^2|\nabla_{g_{\S^3}} f(p)|^2)+O(\var ^3).\label{lem:5.5-6.14}
		\end{align}

		Next, we apply \eqref{42} and \eqref{lem:5.5-6} to analyze the leading contribution in 
		\begin{align} 
			\frac{\al}{3}\int_{\S ^3}\la \nabla_{g_{\S ^3}}x_i, \nabla_{g_{\S ^3}}f _{\hat \Phi} \ra _{g_{\S ^3}}\, \d s _{g_{\S ^3}}
			=&\,\al\int_{\S ^3}x_i(f_{\hat \Phi}-f(p))\, \d s _{g_{\S ^3}}\nonumber\\
			=&\,\al\int_{B(0, \frac{1}{\var} )}x_i(f(\hat\Psi_\var (y))-f(p))\frac{8\, \d y  }{(1+|y|^2)^3}+O(\var ^3)\nonumber\\
			=&\,\al\var \int_{B(0, \frac{1}{\var} )}x_i\sum_{j=1}^3\frac{\pa  f} {\pa  x_j}(p) y_j\frac{16\, \d y  }{(1+|y|^2)^3}\nonumber\\
			&\,
			+\al\var ^2\int_{B(0, \frac{1}{\var} )}x_i\sum_{j,k=1}^3\frac{\pa ^2 f} {\pa  x_j\pa  x_k}(p) y_j y_k\frac{16\, \d y  }{(1+|y|^2)^3}
			+O(\var ^3|\ln \var |)\nonumber
			\\:=&\, I^i+II^i+O(\var ^3|\ln \var |).\label{lem:5.5-5}
		\end{align}

		Since $x_i= \Psi^i(y,0)$ (see in \eqref{PsiR^3}),   symmetry implies that for $i=1,2,3$,
		\[
		I^i
		=\,\frac{32}{3}\al\var\frac{\pa  f} {\pa  x_i}(p)\int_{B(0, \frac{1}{\var} )}\frac{|y|^2\, \d y  }{(1+|y|^2)^4}=\,\frac{4\pi^2}{3}\al \var \frac{\pa  f} {\pa  x_i}(p)+O(\var ^4),
		\]
		while   $I^4=0$.
		Similarly, for $i=1,2,3$, we obtain $II^i=0$,
		and for $i=4$,  
		\begin{align*}
			II^4=&\,\al\var ^2\int_{B(0, \frac{1}{\var} )}\frac{1-|y|^2}{1+|y|^2}\sum_{j,k=1}^3\frac{\pa ^2 f} {\pa  x_j\pa  x_k}(p) y_j y_k\frac{16\, \d y  }{(1+|y|^2)^3}\\
			=&\,\al\var ^2\sum_{j,k=1}^3\frac{\pa ^2 f} {\pa  x_j\pa  x_k}(p)\int_{B(0, \frac{1}{\var} )}\frac{1-|y|^2}{1+|y|^2} y_j y_k\frac{16\, \d y  }{(1+|y|^2)^3}\\
			=&\,\al\frac{\var ^2\Delta_{g_{\S^3}} f(p)}{3}\int_{B(0, \frac{1}{\var} )}\frac{16|y|^2(1-|y|^2)\, \d y  }{(1+|y|^2)^4}
			\\=&\,-\frac{8\pi^2}{3}\al \var ^2  \Delta_{g_{\S^3}} f(p)+O(\var ^3).
		\end{align*} 
		Consequently, we obtain
		\be\label{finnal}
		\frac{\al}{3}\int_{\S ^3}\la \nabla_{g_{\S ^3}}x_i, \nabla_{g_{\S ^3}}f _{\hat \Phi} \ra _{g_{\S ^3}}\, \d s _{g_{\S ^3}}=\left\{\begin{aligned}
			&\frac{4\pi^2}{3}\al \var \frac{\pa  f} {\pa  x_i}(p) +O(\var ^3|\ln \var |) &&\text{ if }\,i=1,2,3,
			\\ &-\frac{8\pi^2}{3}\al \var ^2  \Delta_{g_{\S^3}} f(p) +O(\var ^3|\ln \var |) &&\text{ if }\,i=4.
		\end{aligned}\right.
		\ee

		Combining \eqref{lem:5.5-3} with \eqref{lem:5.5-6.14} and \eqref{finnal}, we conclude that for $i=1,2,3$,
		\be\label{bi}
		b^i= \frac{4\pi^2}{3}\al \var \frac{\pa  f} {\pa  x_i}(p) +E_i,
		\ee
		where the error term $E_i$ satisfies
		\begin{align*}
			E_i=&\,O(\|\hat v\|_{H^3(\S ^3,g_{\S ^3})}^2+\|f_ {\hat \Phi}-f(p)\|_{L^2(\S ^3,g_{\S ^3})}+\|\al  f_{\hat \Phi}-T_h\|^2_{L^2(\S ^3,\hat{h})}))+O(\var ^3|\ln \var |)
			\\=&\,O(F_2+\|f_ {\hat \Phi}-f(p)\|_{L^2(\S ^3,g_{\S ^3})}^2)+O(\var ^3|\ln \var |)\\=&\,O(|b|^2)+O(\var^2),
		\end{align*}
		in view of  \eqref{lem:5.5-3},  \eqref{finnal},  H\"older inequality,  and Lemmas \ref{lem:5.2}, \ref{lem:5.3} and \ref{lem:5.4}.  Similarly,  for $i=4$, 
		\be\label{b4}
		b^4=-\frac{8\pi^2}{3}\al \var ^2  \Delta_{g_{\S^3}} f(p) +E_4,
		\ee 
		where the error term $E_4$ satisfies
		\begin{align*}
			E_4=&\,O(\|\hat v\|_{H^3(\S ^3,g_{\S ^3})}^2+\|f_ {\hat \Phi}-f(p)\|_{L^2(\S ^3,g_{\S ^3})}+\|\al  f_{\hat \Phi}-T_h\|^2_{L^2(\S ^3,\hat{h})}))+O(\var ^3|\ln \var |)
			\\=&\,O(F_2+\|f_ {\hat \Phi}-f(p)\|_{L^2(\S ^3,g_{\S ^3})}^2)+O(\var ^3|\ln \var |)\\=&\,O(|b|^2)+O(\var ^2|\nabla_{g_{\S^3}} f(p)|^2)+O(\var ^3|\ln \var |),
		\end{align*}
		in view of  \eqref{lem:5.5-3},  \eqref{lem:5.5-6.14}, \eqref{finnal},  H\"older inequality,  and Lemmas \ref{lem:5.2} and \ref{lem:5.4}. From \eqref{42}, \eqref{69}, \eqref{bi} and \eqref{b4} we  obtain
		$|b|^2=O(\var^2|\nabla_{\S^3} f(p)|^2)+O(\var^4|\ln \var|^2)$,
		which also implies that 
		\be\label{Ei}
		E_i=\left\{\begin{aligned}
			&O(\var^2)&&\text{ if }\, i=1,2,3,\\
			&O(\var^2|\nabla_{\S^3} f(p)|^2)+O(\var^3|\ln \var|)&&\text{ if }\, i=4.
		\end{aligned}\right.
		\ee
		
		Finally, the proof is completed  from  \eqref{bi} and \eqref{Ei}. 
	\end{proof}

	\begin{lem}\label{lem:5.6}
		As $t\to+\infty$, the following asymptotic expansion holds:
		\[
		b=\frac{\pi^2}{2}\Big(\frac{\d q_1}{\d t},\frac{\d q_2}{\d t},\frac{\d q_3}{\d t},-\frac{\d r}{\d t}\Big)
		+O(\var ^2|\nabla_{g_{\S^3}} f(p)|^2)+O(\var ^3).
		\]
	\end{lem}
	\begin{proof}
		From Lemmas \ref{lem:5.2} and \ref{lem:5.4},
		we obtain 
		\be \label{86}
		F_2= O(\var ^2|\nabla f(p)|^2)+O(\var ^3).
		\ee 
		By  \eqref{flow3} and \eqref{34}, we obtain
		\[
		3b^i=3\int_{\S ^3} x_i (\al  f_{\hat \Phi}-T_h) \, \d s _{\hat h} 
		=3\int_{\S ^3}x\pa _t \hat w\circ {\hat \Phi} \,\, \d s _{\hat h}=\int_{\S ^3}\hat\xi_i \, \d s _{\hat h}:=X_i+I^i,
		\]
		where $X_i$ is defined as in \eqref{X_defn}  and $I^i:=\int_{\S ^3}\hat\xi_i(e^{3\hat v}-1)\, \d s _{g_{\S ^3}}$
		satisfies
		\begin{align*}
			|I^i|= &\,O(\|\hat v\|_{L^\infty(\S ^3)}\|\hat\xi\|_{L^2(\S ^3,g_{\S ^3})})
			\\= &\, O(\|\hat v\|_{H^3(\S ^3,g_{\S ^3})}\|\hat\xi\|_{L^\infty(\S ^3)})\\
			=&\, O(F_2+\|f_ {\hat\Phi}-f(p)\|^2_{L^2(\S^3,g_{\S^3})})
			\\=&\, O(\var ^2|\nabla f(p)|^2)+O(\var ^3),
		\end{align*}
		in view of  \eqref{lem:5.5-6.14}, \eqref{86}, Lemmas \ref{35} and  \ref{lem:5.3}.
		Combining this with \eqref{78} and \eqref{79} completes the proof.
	\end{proof}

	The concentration point $p(t)$ admits the following asymptotic description.
	\begin{lem}\label{lem:5.7}
		As $t\to+\infty$, the following asymptotic expansions hold:
		\[
		\frac{\d p^i}{\d t}=(2+o(1))\var \frac{\d q^i}{\d t},\quad i=1,2,3,
		\]
		and
		\[
		\frac{\d}{\d t}(1-|p(t)|^2)=(2+o(1))(1-|p(t)|^2)\frac{\d r}{\d t},
		\]
		where   $o(1)\to 0$ as $t\to+\infty $.
	\end{lem}
	\begin{proof}
		Since 
		\[
		p=\dashint_{\S ^3}\hat\Phi \,\d s_{g_{\S ^3}}
		=\frac{1}{2\pi^2}\int_{\R ^3}\hat\Psi_\var \circ\hat\delta_{q,r}\frac{8\, \d y  }{(1+|y|^2)^3}.
		\]
		Differentiating this identity (for fixed $\var (t_0)$ and with $q=q(t)$, $r=r(t)$) at $t=t_0$, we obtain
		\[
		\frac{\d p}{\d t}
		=\frac{1}{2\pi^2}\int_{\R ^3}\var \sum_{i=1}^3\frac{\pa \hat\Psi}{\pa  y_i}(\var  y)\Big(\frac{\d q^i}{\d t}+y_i\frac{\d r}{\d t}\Big)\frac{8\, \d y  }{(1+|y|^2)^3}.
		\]
		In particular, using \eqref{76-1}-\eqref{77} and the symmetries of the  map $\hat \Phi(t_0)$, we obtain
		\[
		\frac{\d p^1}{\d t}=\frac{\var }{2\pi^2}\frac{\d q^1}{\d t}\int_{\S ^3}(1+\hat\Phi_4(t_0)-|\hat\Phi_1(t_0)|^2)\, \d s _{g_{\S ^3}}
		=(2+o(1))\var \frac{\d q^1}{\d t}
		\]
		by Proposition  \ref{prop:3.7}. Similarly, we can obtain the estimates for $\frac{\d p^2}{\d t}$ and $\frac{\d p^3}{\d t}$.
		Finally, by \eqref{76-1}-\eqref{77} and Proposition  \ref{prop:3.7}, we have
		\begin{align*}
			\frac{\d}{\d t}|p|^2=&\,2p^4\frac{\d p^4}{\d t}=(2+o(1))\frac{\d p^4}{\d t}\\
			=&\,-\frac{(2+o(1))}{2\pi^2}\frac{\d r}{\d t}\int_{\S ^3}(1-|\hat\Phi_4(t_0)|^2)\, \d s _{g_{\S ^3}}
			\\=&\,-(2+o(1))(1-|p|^2)\frac{\d r}{\d t}.
		\end{align*}
		This completes the proof.
	\end{proof}

	\begin{lem}\label{lem:5.8}
		As $t\to+\infty$, the following asymptotic expansion holds:
		\[1-|p(t)|^2=(12+o(1))\var ^2,\]
		where $o(1)\to 0$ as $t\to+\infty $.
	\end{lem}
	\begin{proof}
		The proof is similar to that  Lemma 5.8 in \cite{MS2006}, and we omit the details here.
	\end{proof}

	From the above lemmas, we are now in position to give a neat characterization of the limit of the shadow flow $p$. Below, we denote by $\frac{\d p^i}{\d t}$ the component of $\frac{\d p}{\d t}$ which is tangential to $\S^3$ at $\frac{p}{|p|}$.
	\begin{prop}\label{prop4.9}
		(i) The following estimates hold as $t\to+\infty $:
		\begin{gather*}
			\frac{\d p^i}{\d t}=\frac{32}{3}\al \var ^2\Big(\frac{\pa  f} {\pa  x_i}(p)+o(1)\Big),\quad i=1,2,3,\\\frac{\d}{\d t}(1-|p|^2)=384\al \var ^4(\Delta_{g_{\S ^3}} f(p)+O(1)|\nabla_{g_{\S ^3}} f(p)|^2+o(1)),
		\end{gather*}
		and  
		\[
		1-|p(t)|^2=(12+o(1))\var ^2,
		\]
		where $o(1)\to 0$ as $t\to +\infty$.\\
		(ii) As $t\to+\infty $, the metrics $\hat g(t)$ concentrate at the  critical points $P$ of $f$ satisfying $\Delta_{g_{\S ^3}} f(P)\leq 0$.
	\end{prop}
	\begin{proof}
		(i) The first statement follows from  Lemmas \ref{lem:5.5}--\ref{lem:5.8}.

		(ii) From (i) and \eqref{42}, we find 
		\[
		\Big|\frac{\d}{\d t}(1-|p|^2)\Big|\leq C(1-|p|^2)^2,
		\]
		which yields
		\[
		1-|p(t)|^2\geq\frac{C_0}{1+t}
		\]
		for some constant $C_0>0$. Combining this with Lemma \ref{lem:5.8}, we have
		\be \label{87}
		\var ^2\geq\frac{C_1}{t+1}\quad \mbox{ for all }\,t\geq 0,
		\ee 
		where   $C_1>0$ is a uniform constant.
		This together with (i) implied that 
		\be \label{88}
		\frac{\d f(p)}{\d  t}=\nabla f(p)\frac{\d p}{\d t}\geq\frac{C_2}{t+1}(|\nabla_{\S^3} f(p)|^2+o(1)),
		\ee 
		for some constant $C_2>0$, where $o(1)\to 0$ as $t\to +\infty$.  Since the integral of $(1+t)^{-1}$  over the interval $(0,+\infty)$ is divergent, the flow $\{p(t)\}_{t\geq 0}$ must accumulate at
		a critical point of $f$, denoted by $P$.  To see limit
		point of $p(t)$ must be of $\Delta_{g_{\S ^3}} f(P)\leq 0$, we assume, on the contrary, $\Delta_{g_{\S ^3}} f(P)>0$. We rescale time $t$ as $s(t)$ by solving the equation
		\be \label{91}
		\left\{\begin{aligned}
			&\frac{\d s }{\d t}=\min\Big\{\frac{1}{2},\var ^2(t,w_0)\Big\},&& \\ &s(0)=0.&&
		\end{aligned}\right.
		\ee 
		By \eqref{87}, it yields $s(t) \to + \infty$ as $t \to +\infty$. Recalling (i) and Lemma \ref{lem:5.7}, one has
		\[
		\frac{\d}{\d s}(1-|p(t(s))|^2)=32 \al(1-|p(t(s))|^2)(\Delta_{\S^3} f(p(t(s)))+o(1)) .
		\]
		Then there exists a uniform constant $C_3>0$ depending on $\min _{\nabla_{g_{\S^3}} f(a)=0}|\Delta_{\S^3} f(a)|$ and $M_f$, such that
		\[
		\frac{\d}{\d s} \log (1-|p(t(s))|^2) \geq C_3,
		\]
		which obviously contradicts the fact that $1-|p|^2 \to 0$ as $t \to + \infty$. 
		
		Therefore the shadow flow $\{p(t)\}_{t\geq 0}$ converges to a unique point $ P\in \S^3$.
	\end{proof}

	The energy functional associated with $w(t)$ satisfies the following asymptotic behavior.
	\begin{lem}\label{lem:5.10}
		As $t\to+\infty $, we have
		\[
		E_f[w(t)]\to -\frac{16\pi^2}{3}\log f(Q),
		\]
		where $Q:=\lim_{t\to+\infty }p(t)$
		is the unique limit point of the shadow flow $\{p(t)\}$
		associated with $\{w(t)\}$.
	\end{lem}
	\begin{proof}
		By \eqref{sameenergy_vw} and Proposition  \ref{prop:3.7}, we have
		$E[w(t)]=E[v(t)]\to 0$ as $t\to+\infty $. On the other hand, using the change of variables induced by $\hat\Phi(t)$, we compute
		\[
		\dashint_{\S ^3}f  e^{3\hat w(t)}\, \d s _{g_{\S ^3}}
		=\dashint_{\S ^3}f _{\hat\Phi} e^{3\hat v(t)}\, \d s _{g_{\S ^3}}\to f(Q)
		,
		\]
		which completes the proof.
	\end{proof}

	\section{Proof of the main results}\label{sec:7} 
	Throughout this section, we set $C^\infty_*:=\{w\in C^\infty(\B ^4): g=e^{2w}g_{\B ^4}\in[g_{\B ^4}]_0\}$.
	Following the approach of  Chang-Yang \cite[p.237]{CY1987}, for each  $(p,\var)\in\S ^3\times \R^+$, we introduce  stereographic coordinates 
	with the point $-p$ sent to  infinity (so that $p$ becomes the north pole). In these coordinates, we may define   $\hat \Phi_{p,\var }:=\Psi_\var \circ \pi$, as recalled in   Section \ref{sec:6}.
	By a similar argument for , $\hat\Phi_{p,\var }\rightharpoonup \hat\Phi_{p,0}\equiv p$
	weakly in $H^3(\S^3,g_{\S^3})$ as $\var \to 0^+$.
	Moreover, the conformal maps satisfy  $\hat\Phi_{p,\var }=\hat\Phi_{p,\var ^{-1}}^{-1}=\hat\Phi_{-p,\var ^{-1}}$
	for all $(p,\var)\in\S ^3\times \R^+$. This allows us to define a map
	\begin{align*}
		j:	\S ^3\times \R^+&\longrightarrow C^\infty_*,\\(p,\var )&\longmapsto w_{p,\var },
	\end{align*}
	with the  property   (cf. \cite{C1998}) that
	\[
	\d s _{\hat g_{p,\var }}\rightharpoonup 2\pi^2\delta_p\quad \text{ as }\, \var \to 0^+,
	\]
	in the weak sense of measures. Here the associated metric is given by 
	$g_{p,\var }=e^{2w_{p,\var }}g_{\B ^4}= \Phi_{p,\var }^* g_{\B ^4}$,   and the boundary conformal factor is   \[
	\hat w_{p, \var }=-\frac{1}{3}\log\det(\d \hat\Phi_{p,\var }).
	\]
	Since 
	$\hat \Phi_{p,1}=\operatorname{id}$, it follows that  $\hat j(p,1):=j(p, 1)|_{\S ^3}=0$ for all $p\in\S ^3$.

	Let $w_0\in C^\infty_*$ be given. For $t\geq 0$,
	let $w=w(t,u_0)$ denote the solution of the flow \eqref{flow3}
	at the time $t$ with initial condition $w|_{t=0}=w_0$.
	Let $\Phi(t,w_0)$ be the corresponding family of normalized conformal diffeomorphisms such that the normalization condition \eqref{23} holds for the pull-back metric
	$h=h(t,w_0)=e^{2v}g_{\B ^4}=\Phi^*g$,
	where  $g=g(t,w_0)=e^{2w}g_{\B ^4}$,
	$\Phi=\Phi(t,w_0)$,
	and  $v=v(t,w_0)$ is chosen accordingly.
	Finally, define 
	\[
	p=p(t,w_0):=\dashint_{\S ^3}\hat\Phi(t,w_0)\, \d s _{g_{\S ^3}}
	\]
	to be the approximate center of mass of the boundary metric $\hat g(t,w_0)$.
	In particular, up to a rotation, whenever $p\neq 0$
	we may write $\hat\Phi=\hat\Phi_{p,\var }$ for a uniquely  determined  parameter $\var =\var (t,w_0)\in (0,1)$.
	Moreover, the flow depends continuously on the initial datum 
	$w_0$ in any  $C^k$-topology (and hence in the smooth topology).

	We reparametrize the flow $w(t,w_0)$ by introducing a rescaled time variable $s=s(t)$ solving \eqref{91}.
	Note that \eqref{87} implies that $s(t)\to+\infty$ as $t\to+\infty $.
	For $s\in [0,+\infty)$, we then define
	$U(s,w_0):=U(t(s),w_0)$, $P(s,w_0):=p(t(s),w_0)$.
	In view of Proposition \ref{prop4.9}, for $\var <\frac{1}{2}$,
	the rescaled flow satisfies (in stereographic coordinates)
	\[
	\frac{\d P^i}{\d s }=\frac{32}{3}\al \var ^2\Big(\frac{\pa  f} {\pa  x_i}(P)+o(1)\Big),\quad i=1,2,3,
	\]
	and
	\[
	\frac{\d}{\d s }(1-|P|^2)=384\al \var ^4(\Delta_{g_{\S ^3}} f(P)+O(1)|\nabla_{g_{\S ^3}} f(P)|^2+o(1)),
	\]
	where $o(1)\to 0$ as $\var \to 0$.
	In what follows, we again denote the rescaled time variable by $t$, thereby reserving the letter $s$ for other purposes.

	For each $\beta\in\R $, we define the sublevel set of the functional $E_f$ by
	\be\label{sublevelset}
	E_f^{\beta}:=\{u\in C^\infty_*: \,E_f[u]\leq\beta\}.
	\ee 
	To proceed with the proof of Theorem \ref{thm:main}, we label all critical points $p_1,\ldots, p_N$
	of $f$  such that  $f(p_i)\leq f(p_j)$ for $1\leq i\leq j\leq N$. For each $1\leq i\leq N$, we set \[\beta_i:= -\frac{16\pi^2}{3}\log f(p_i)=\lim_{s\to 0}E_f[w_{p_i,s}].\] For notational convenience, we assume that the critical values $f(p_i)$, $1\leq i\leq N$, are all distinct. Consequently, there exists a constant
	$\theta_0>0$ such that
	$\beta_i-\theta_0>\beta_{i+1}$ for all $i$.

	As in \cite{MS2006,XZ2016,H2012,CX2011,S2005}, in order to establish the existence of conformal metrics solving problem \eqref{maineq}, it is essential to prove the following key topological statement.
	\begin{prop}\label{prop:5.11}
		(i) For any $\beta_0>\beta_1$, the set $E_f^{\beta_0}$ is contractible.\\
		(ii) For any $\theta\in (0,\theta_0]$ and for each $i$, the set $E_f^{\beta_i-\theta}$ is homotopy equivalent to the set
		$E_f^{\beta_{i+1}+\theta}$.\\
		(iii) For each critical point $p_i$ of $f$ satisfying $\Delta_{g_{\S ^3}} f(p_i)>0$,
		the sets $E_f^{\beta_i+\theta_0}$ and $E_f^{\beta_i-\theta_0}$ are homotopy equivalent.\\
		(iv) For each critical point $p_i$ of $f$ satisfying $\Delta_{g_{\S ^3}} f(p_i)<0$,
		the set $E_f^{\beta_i+\theta_0}$ is homotopy equivalent to the space obtained from $E_f^{\beta_i-\theta_0}$ by attaching a sphere of dimension $3-\morse(f,p_i)$.
	\end{prop}
	
	\begin{proof} 
		(i) Fix $\beta_0>\beta_1$. For $ s\in [0,1]$, define a homotopy	
		$H_1(s, w_0):=(1-s) w_0+c(s, w_0)$,
		where  $c(s, w_0)$	s a suitable constant such
		that  $H_1(s, w_0)\in C_*^{\infty}$. Given $s\in [0,1]$  and $w_0$, by Lemma \ref{lem:5.10} and the choice of $\beta_0$, there exists a minimal time $T=T(s, w_0)$ such that $E_f[w(T, H_1(s, w_0))] \leq  \beta_0$. Moreover,  by Lemma \ref{lem:2.1} and the assumption on $f$,   the map  $T(s, w_0)$  depends continuously on $s$ and $w_0$. We may therefore define $H(s, w_0):= w(T(s, w_0), H_1(s, w_0))$,which provides a contraction of $E_f^{\beta_0}$ within itself. Indeed, since $T(0, w_0)=0$, we have $H(0, w_0)=w(0, w_0)=w_0$. On the other hand, $H_1(1, w_0)=0$ and $E_f[0] \leq \beta_1<\beta_0$, hence $T(1, w_0)=0$ and $H(1, w_0)=w(0,0)=0$. This shows that $E_f^{\beta_0}$ is contractible.
		
		(ii)  	Statement (ii) follows  by  a contradiction argument. We refer to \cite[Proposition 5.11]{MS2006} for details.
	\end{proof}

	To complete the proofs of parts (iii) and (iv) in Proposition \ref{prop:5.11}, we first establish two preliminary lemmas.
	\begin{lem}\label{lem:5.12}
		Let  $v\in C^4(\B ^4)\cap \cH_{g_{\B ^4}}$  satisfy $P^4_{g_{\B ^4}}v=0$ in $ \B^4$, and assume that the   induced  normalized metric $h$ satisfies
		\eqref{23} and \eqref{26}. Then	there exists a uniform constant $C>0$ such that
		\[
		\|\hat v\|^2_{H^{\frac{3}{2}}(\S ^3,g_{\S^3})}\leq CE[v], 
		\]
		provided that $\|\hat v\|_{H^{\frac{3}{2}}(\S ^3,g_{\S^3})}$ is sufficiently small.
	\end{lem}
	\begin{proof}
		Using  \eqref{EW}, \eqref{P43},  \eqref{23}, \eqref{samev},  and \eqref{72}, and arguing as in the proof of Theorem \ref{thm:4.1}, we obtain
		\begin{align*}
			E[v]=&\,\int_{\S ^3}(\hat v\cP^{3}_{g_{\S ^3}}\hat v+8\hat v)\, \d s _{g_{\S ^3}}\\
			=&\,\int_{\S ^3}\Big(\hat v\cP^{3}_{g_{\S ^3}}\hat v+\frac{8}{3}(e^{3\hat v}-1)-12\hat v^2\Big)\, \d s _{g_{\S ^3}}+O(\|\hat v\|_{H^{\frac{3}{2}}(\S ^3,g_{\S^3})}^3)\\
			=&\,\sum_{i=0}^\infty(\Lam_i-12)| \hat v^i|^2+O(\|\hat v\|_{H^{\frac{3}2}(\S ^3,g_{\S^3})}^3)\\\geq&\, \sum_{i=5}^\infty(\Lam_i-12)| \hat v^i|^2+o(\|\hat v\|_{H^{\frac{3}{2}}(\S ^3,g_{\S^3})}^2)\\
			\geq&\,\frac{\Lam_5-12}{\Lam_5+1}\sum_{i=0}^\infty(\Lam_i+1)| \hat v^i|^2+o(\|\hat v\|_{H^{\frac{3}{2}}(\S ^3,g_{\S^3})}^2).
		\end{align*}
		Since  $\Lam_5=24$ (see \eqref{eigenvalue})	and $\|\hat v\|_{H^{\frac{3}{2}}(\S ^3,g_{\S^3})}$ is assumed sufficiently small, the desired estimate follows as in the proof of Lemma \ref{lem:3.2}.
	\end{proof}
	
	Before stating the next lemma, we introduce some notation.

	For $r_0>0$ and each critical point $p_i\in\S ^3$ of $f$, define
	\begin{align*}
		\cB_{r_0}(p_i):
		=\Big\{ w\in C^\infty_*: &\,g=e^{2w}g_{\B ^4}\mbox{ induces
			a normalized metric} \,h=\Phi^* g=e^{2v}g_{\B ^4} \\
		& \mbox{ with } \hat \Phi=\hat \Phi_{p,s}\mbox{ for some }p\in\S ^3, \, s\in (0,1] \mbox{ such that } \\
		&\|\hat v\|^2_{H^{\frac{3}{2}}(\S ^3,g_{\S^3})}+|p-p_i|^2+s^2<r_0^2\Big\}.
	\end{align*}
	For convenience, we regard  $(s,p,v)$
	as local coordinates for $w\in \cB_{r_0}(p_i)$.
	Since each critical point of $f$ is assumed nondegenerate, the Morse lemma yields local coordinates near $p_i$ such that one may write $p=p^{+}+p^{-}$ in $T_{p_i} \S^3$ (with $p_i$ identified with 0) and
	\[
	f(p)=f(p_i)+|p^+|^2-|p^-|^2.
	\]

	We now state the next technical lemma, which is mainly concerned with expansions of the relevant quantities in terms of the local coordinates $(s, p, v)$.
	\begin{lem}\label{lem:5.13}
		For $r_0>0$ and  $w=(s,p,v)\in \cB_{r_0}(p_i)$ for some $1\leq i\leq N$. Then there hold \\
		(i) 
		\[\dashint_{\S ^3}f \circ \hat\Phi_{p,s} \, \d s _{\hat h} 
		=f(p)+2s^2\Delta_{g_{\S ^3}}f(p)+O(s^3)+o(1)s\|\hat v\|_{H^{\frac{3}{2}}(\S ^3,g_{\S^3})},\]
		where $o(1)\to 0$ as $r_0\to 0$.\\
		(ii) 
		\[\Big|\frac{\pa E_f [w]}{\pa  s}
		+\frac{64\pi^2s}{3}\frac{\Delta_{g_{\S ^3}}f(p)}{f(p)}\Big|
		= O(s^2+(s+|p-p_i|)\|\hat v\|_{H^{\frac{3}{2}}(\S ^3,g_{\S^3})}) .\]\\
		(iii) For any $q\in T_p\S ^3$, there holds
		\[\Big|\frac{\pa E_f [w]}{\pa  p}\cdot q+\frac{16\pi^2}{3}\frac{\d f(p)\cdot q}{f(p)}\Big|= O(s(s+\|\hat v\|_{H^{\frac{3}{2}}(\S ^3,g_{\S^3})})|q|).\]\\
		(iv) There exists a uniform constant $c_0>0$ such that
		\[
		\Big\la \frac{\pa E_f [w]}{\pa  v},v\Big\ra \geq c_0\|\hat v\|_{H^{\frac{3}{2}}(\S ^3,g_{\S^3})}^2+o(1)s\|\hat v\|_{H^{\frac{3}{2}}(\S ^3,g_{\S^3})},
		\]
		where  $\la \cdot,\cdot\ra $ denotes the duality
		pairing of $H^{\frac{3}{2}}(\S^3,g_{\S^3})$ with its dual, and $o(1)\to 0$ as $r_0\to 0$.
	\end{lem}
	\begin{proof} Since the proof follows essentially the same strategy as that of \cite[Lemma 5.13]{MS2006}, and related computations in even  dimensions can be found in \cite{H2012,CX2011}, we omit the details of the proofs of parts (i)-(iii) and only provide the proof of (iv) below.
		
		(iv) Set	$A=A(w)=\dashint_{\S ^3}f \circ \hat\Phi_{p,s} \, \d s _{\hat h}$.
		We compute from \eqref{variation}, \eqref{EW} and \eqref{sameenergy_vw} that 
		\begin{align*}
			\Big\la \frac{\pa E_f [w]}{\pa  v},v\Big\ra 
			=&\,\Big\la \frac{\pa E[v]}{\pa  v},v\Big\ra 
			-8A^{-1}\int_{\S ^3}f \circ\hat\Phi_{p,s}\hat v e^{3\hat v}\, \d s _{g_{\S ^3}}\\
			:=&\,\Big\la \frac{\pa E[v]}{\pa  v},v\Big\ra 
			-8\int_{\S ^3}\hat ve^{3\hat v}\, \d s _{g_{\S ^3}}-I,
		\end{align*}
		where
		\begin{align*}
			I=&\,8A^{-1}\int_{\S ^3}(f\circ\hat\Phi_{p,s}-f(p))\hat v e^{3\hat v}\, \d s _{g_{\S ^3}}
			+8(A^{-1}f(p)-1)\int_{\S ^3}\hat v e^{3\hat v}\, \d s _{g_{\S ^3}}\\
			:=&\,II+III.
		\end{align*}
		As in the proof of Lemma \ref{lem:5.12},   we have 
		\begin{align*}
			&\,\Big\la \frac{\pa E[v]}{\pa  v},v\Big\ra 
			-8\int_{\S ^3}\hat v e^{3\hat v}\, \d s _{g_{\S ^3}}\\
			=&\,\int_{\S ^3}(2\hat v \cP^{3}_{g_{\S ^3}} \hat v+4T_{g_{\B ^4}} \hat v)\, \d s _{g_{\S ^3}}
			-8\int_{\S ^3}\hat ve^{3\hat v}\, \d s _{g_{\S ^3}}\\
			=&\,\int_{\S ^3}2\hat v \cP^{3}_{g_{\S ^3}} \hat v \, \d s _{g_{\S ^3}}
			-8\int_{\S ^3}\hat v(e^{3\hat v}-1)\, \d s _{g_{\S ^3}}\\
			=&\,2\int_{\S ^3}(\hat v\cP^{3}_{g_{\S ^3}}\hat v-12 \hat v^2)\, \d s _{g_{\S ^3}}
			+o(1)\|\hat v\|_{H^{\frac{3}{2}}(\S ^3,g_{\S^3})}^2
			\\\geq &\,c_0\|\hat v\|_{H^{\frac{3}{2}}(\S ^3,g_{\S^3})}^2+o(1)\|\hat v\|_{H^{\frac{3}{2}}(\S ^3,g_{\S^3})}^2
		\end{align*}
		for some constant $c_0>0$.
		Moreover,  we derive that 
		\[|II|=O(\|f \circ \hat\Phi_{p,s}-f(p)\|_{L^2(\S ^3,g_{\S^3})}\|\hat v(e^{3\hat v}-1)\|_{L^2(\S ^3,g_{\S^3})}) = o(1)s\|\hat v\|_{H^{\frac{3}{2}}(\S ^3,g_{\S^3})},\]
		while from (i) we obtain
		\[
		|III|=O\Big(|A-f(p)|\int_{\S ^3}|\hat v|e^{3\hat v}\, \d s _{g_{\S ^3}}\Big)= O((s^2+s\|\hat v\|_{H^{\frac{3}{2}}(\S ^3,g_{\S^3})})\|\hat v\|_{H^{\frac{3}{2}}(\S ^3,g_{\S^3})}).
		\]
		Combining these estimates, we complete  the claim (iv).
	\end{proof}

	Since Lemmas \ref{lem:5.12} and \ref{lem:5.13} have been established, we may proceed as in the proof of Proposition 5.11 in \cite{MS2006} to conclude parts (iii) and (iv) of Proposition \ref{prop:5.11}. We therefore omit the details and refer the reader to \cite{MS2006}; see also \cite{H2012,CX2011} for the corresponding arguments in even  dimensions.

	Now we are ready to prove Theorem \ref{thm:main}.
	\begin{proof}[Proof of Theorem \ref{thm:main}]
		Suppose by contradiction that  the flow $\eqref{flow2} $ diverges for every initial value $u_0$ and
		there is no conformal metric in the  conformal class of $g_{\B^4}$ with the $T$-curvature equal
		to $f$.
		As we will see, the flow \eqref{flow}
		then may be used to show that
		for sufficiently large $\beta_0$, the set $E_f^{\beta_0}$
		is contractible.
		Moveover, the flow defines a homotopy equivalence of the set $N_0=E_f^{\beta_0}$
		with a set $N_\infty$ whose homotopy type is that of a point $\{p_0\}$ with cells of dimension
		$3-\morse(f,p)$ attached for every critical point $p$ of $f$ on $\S ^3$
		such that $\Delta_{g_{\S ^3}}f(p)<0$.
		We then, by standard Morse theory (see \cite[Theorem 4.3 on p.36]{changinfinite}), obtain the identity
		\be \label{94}
		\sum_{i=0}^3t^im_i=1+(1+t)\sum_{i=0}^3t^i k_i
		\ee 
		for the Morse polynomials of $N_\infty$ and $N_0$
		and a connection term with coefficients $k_i\geq 0$, where $m_i$ is defined as in \eqref{14}.
		Equating the coefficients in the polynomials on the left and right hand side of \eqref{94},
		we obtain \eqref{eq:system}, which violates the hypothesis in Theorem \ref{thm:main}
		and thus leads to the desired contradiction.
	\end{proof}

	The proof of Corollary \ref{cor:1} now follows as an immediate consequence.
	\begin{proof}[Proof of Corollary \ref{cor:1}]
		If we substitute $t = -1$ into the Morse polynomial identity \eqref{94}, derived under the assumption of non-existence, we obtain	
		\[
		\sum_{i=0}^3 (-1)^i m_i = 1.
		\]	
		In view of the definition of $ m_i $ in \eqref{14}, this is equivalent to	
		\[
		\sum_{\{a\in\S^3 : \nabla_{g_{\S^3}} f(a)=0, \Delta_{g_{\S^3}}  f(a)<0\}} (-1)^{\morse(f,a)} = -1,
		\]	
		which contradicts the hypothesis \eqref{13} of Corollary \ref{cor:1}. Therefore, that hypothesis forces the existence of a solution, and Corollary \ref{cor:1} is proved.
	\end{proof}

	\subsection*{Acknowledgement} 
	P.T. Ho  is partially supported by   the National Science and Technology Council (NSTC),
	Taiwan, with grant Number: 112-2115-M-032-006-MY2.
	C.B. Ndiaye is partially supported by NSF grant DMS--2000164 and AMS 2025–2026 Claytor-Gilmer Fellowship.  The research of L. Sun is partially supported by CAS Project for Young Scientists in Basic Research Grant (No.\,YSBR-031), Strategic Priority Research Program of the Chinese Academy of Sciences (No.\,XDB0510201), National Key R\&D Program of China (No.\,2022YFA1005601), NSFC China (No.\,12471115).
	
	\appendix

\bibliographystyle{plain}
\bibliography{T-202601ref.bib}

\end{document}